\documentclass[11pt]{amsart}
\usepackage{hyperref}
\usepackage[utf8]{inputenc}
\usepackage[nohug]{ptdiagrams}
\usepackage{xcolor}
\usepackage[demo]{graphicx}
\usepackage{caption}
\usepackage{subcaption}
\usepackage[mathscr]{eucal}
\usepackage{amsfonts, amsmath, amsthm, amssymb, latexsym}
\usepackage{enumitem}

\newtheorem{theorem}{Theorem}[section]
\newtheorem{lemma}[theorem]{Lemma}

\newtheorem{prop}[theorem]{Proposition}
\newtheorem{cor}[theorem]{Corollary}
\theoremstyle{definition}
\newtheorem{defi}[theorem]{Definition}
\newtheorem{example}[theorem]{Example}
\theoremstyle{remark}
\newtheorem{remark}[theorem]{Remark}

\numberwithin{equation}{section}

\newcommand{\onto}{\,\,\twoheadrightarrow\,\,}
\newcommand{\la}{\label}
\newcommand{\GL}{\mathrm{GL}}
\newcommand{\Ker}{\mathrm{Ker}}

\newcommand{\Spec}{\mathrm{Spec}}
\newcommand{\Map}{\mathrm{Map}}

\newcommand{\Ho}{\mathrm{Ho}}

\newcommand{\Cell}{\mathrm{Cell}}

\newcommand{\hocolim}{\mathrm{hocolim}}
\newcommand{\ch}{\mathrm{ch}}

\def\deg{\mathrm{deg}\,}

\newcommand{\id}{\mathrm{id}}
\newcommand{\into}{\hookrightarrow}

\newcommand{\Top}{\mathtt{Top}}
\def\bS{\mathbb S}

\def\bP{\mathbb P}

\def\C{\mathbb C}
\def\F{\mathbb F}

\def\bs#1{\boldsymbol{#1}}

\def\c{\mathbb{C}}
\def\e{\boldsymbol{e}}
\def\A{\mathcal{A}}
\def\E{\mathcal{E}}
\def\M{\mathcal{M}}
\def\cQ{\mathcal{Q}}
\def\L{\boldsymbol{L}}
\def\cL{\mathscr{
L}}

\def\Z{\mathbb{Z}}
\def\Q{\mathbb{Q}}

\def\N{\mathbb{N}}

\def\R{\mathcal{R}}
\def\D{\mathcal{D}}
\def\O{\mathcal{O}}

\def\sl2{{\mathfrak{s}\mathfrak{l}}_2}

\def\Hom{\mathrm{Hom}}

\def\Mod{\mathrm{Mod}}
\def\Alg{\mathrm{Alg}}

\def\im{\mathrm{Im}}

\def\dgca{\mathtt{DGCA}}
\def\dAff{\mathtt{dAff}}

\def\Ell{\mathcal{E}{ll}}
\def\Coh{\mathrm{Coh}}
\def\ELL{\mathrm{Ell}}
\def \cF{\mathcal{F}}
\newcommand{\CP}{\mathbb{CP}}
\newcommand{\RP}{\mathbb{RP}}

\renewcommand{\Re}{{\rm Re}}
\renewcommand{\Im}{{\rm Im}}

\begin{document}
\title[Topological Realization of Algebras of Quasi-invariants]{Topological Realization of Algebras of Quasi-invariants, I\\
(with an Appendix by M. V. Feigin and K. E. Feldman)}
\author{Yuri Berest}
\address{Department of Mathematics,
Cornell University, Ithaca, NY 14853-4201, USA}
\email{berest@math.cornell.edu}
\author{Ajay C. Ramadoss}
\address{Department of Mathematics,
Indiana University,
Bloomington, IN 47405, USA}
\email{ajcramad@indiana.edu}
\begin{abstract}
This is the first in a series of papers, where we introduce and study topological spaces
that realize the algebras of quasi-invariants of finite reflection groups. Our main result can be viewed as a generalization of a well-known theorem of A. Borel
that realizes the ring of invariant polynomials a Weyl group $W$ as a cohomology ring of the classifying space $BG$ of the associated Lie group $G$. In the present paper,
we state a general realization problem for the algebras of quasi-invariants of Weyl groups 
and give its solution in the rank one case (for $G = {\rm SU}(2)$).  We call the resulting $G$-spaces $ F_m(G,T) $ the $m$-quasi-flag manifolds and their Borel
homotopy quotients $ X_m(G,T) $ the spaces of $m$-quasi-invariants. We compute the
equivariant $K$-theory and the equivariant (complex analytic) elliptic cohomology of these spaces and identify them with exponential and elliptic quasi-invariants of $W$. 
We also extend our construction of spaces quasi-invariants to `fake' Lie groups of type ${\rm SU}(2)$, a distinguished class of finite loop spaces  $ \Omega B $ of homotopy type of $ \bS^3 $ introduced by Rector in \cite{Rec71}. We study the cochain spectra $ C^*(X_m, k) $ associated to the spaces of quasi-invariants and show that these are Gorenstein commutative ring spectra in the sense of Dwyer, Greenlees and Iyengar \cite{DGI06}.
\end{abstract}
\maketitle

\tableofcontents

\section{Introduction}
\la{S1}
Quasi-invariants are natural generalizations of classical invariant polynomials of finite reflection groups. In the case of Coxeter groups, they first appeared in mathematical physics --- in the work of O. Chalykh and A. Veselov \cite{CV90, CV93} --- in the early 1990s, and since then have found applications in many other areas:  most notably,  representation theory, algebraic geometry and combinatorics (see \cite{FV02}, \cite{EG02b}, \cite{Ch02}, \cite{BEG03}, \cite{FeV03}, \cite{GW06}, \cite{BM08}, \cite{T10}, \cite{BC11}, \cite{BEF20}, \cite{Gri21}). For arbitrary (complex) reflection groups, the quasi-invariants were introduced in \cite{BC11}. This last paper developed a general approach in the context of representation theory of rational double affine Hecke algebras, extending and refining the earlier results of \cite{BEG03} in the Coxeter case. We will use \cite{BC11} as our main reference on algebras of quasi-invariants; in particular, we will follow the notation introduced in that paper.

We begin by recalling the definition of quasi-invariants in the Coxeter case.
Let $ W $ be a finite real reflection group acting in its reflection representation $V$. Denote by 
$ \A := \{H\} $ the set of reflection hyperplanes of $W$ in $V$ and write $ s_H \in W $ for the reflection operator in $H$. The group $W$ acts naturally on the polynomial algebra $ \c[V] $ and, since the $s_{H} $'s generate $W$, the invariant polynomials $p \in \c[V]^W $ are determined by the equations
\begin{equation}
\la{Einv}
s_{H}(p) = p \ , \ \, \forall\,H \in \A\ .
\end{equation}
%
To define quasi-invariants we `weaken' the equations \eqref{Einv} in the following way. For each reflection hyperplane $ H \in \A $, we choose a linear form $ \alpha_H \in V^* $ such that
$ H = \Ker(\alpha_H) $ and fix a non-negative integer $ m_{H} \in \Z_{+} $, assuming that $ m_{w(H)} = m_H $ for all $w \in W $.  In other words, we choose a system of roots of $W$ in $V^*$, which (abusing notation) we still denote by $\A$, and fix a $W$-invariant function $m: \A \to \Z_+,\,H \mapsto m_H $, assigning to roots (or hyperplanes) in $\A$ integral {\it multiplicities}. Now, with these extra data in hand, we replace the polynomial equations \eqref{Einv} by the polynomial congruences
\begin{equation}
\la{Qinv}
s_{H}(p) \,\equiv\, p \ \,\mbox{mod}\,\langle \alpha_H \rangle^{2 m_H}\ , \ \, \forall\,H \in \A\ ,
\end{equation}
where $ \langle \alpha_H \rangle $ denotes the principal ideal  in $ \c[V] $ generated by the form $ \alpha_H $. For each $H \in \A$, the congruence \eqref{Qinv} simply says that the polynomial $ s_H(p) - p $ is divisible in $\C[V] $ by the power of the linear form $ \alpha_H $ determined by the value of the multiplicity of $H$ in $\A$. It is easy to see that, for a fixed $m$, the set of all polynomials satisfying \eqref{Qinv} forms a graded subalgebra in $ \C[V] $.
Using the standard notation (originally introduced in \cite{FV02}), we denote this subalgebra by $Q_m(W)$ and call it (following \cite{CV90}) the algebra of {\it $W$-quasi-invariant polynomials of multiplicity $m$}. Note that, for $ m = 0 $, we have $ Q_0(W) = \C[V]$, while  $ \C[V]^W \subseteq Q_m(W) \subseteq \C[V] $ in general. Thus, for varying $m$, the quasi-invariants interpolate between the $W$-invariants and all polynomials. 

Despite their elementary definition, the algebras $Q_m(W)$ have a nontrivial structure: they do not seem to admit a simple combinatorial description, nor do they have a natural presentation in terms of generators and relations. Nevertheless, these algebras possess many remarkable properties (e.g., the Gorenstein property, see Theorem~\ref{TGoren}) and carry
very interesting structures related to representation theory (e.g., the action of the rational double affine Hecke algebras, see \cite{BEG03}).

The goal of the present work is to give a topological realization of the algebras of quasi-invariants as (equivariant) cohomology rings of certain spaces naturally attached to compact connected Lie groups. Our main result can be viewed as a generalization of a well-known theorem of A. Borel \cite{Bo53}  that realizes the algebra of invariant polynomials of a Weyl group $ W $ as the cohomology ring of the classifying space $BG$ of the associated Lie group $G$. As the algebras $Q_m(W)$ are defined over $\C$, we should clarify what we really mean by ``topological realization''. It is a fundamental consequence of Quillen's rational homotopy theory \cite{Qui69} that every reduced, locally finite, graded commutative algebra $A$ defined over a field $k$ of characteristic zero is topologically realizable, i.e. $ A \cong H^*(X, k) $ for some (simply-connected) space $X$. When equipped with cohomological grading, the algebras $Q_m(W)$ have all the above-listed properties ({\it cf.} Lemma~\ref{L1}); hence, the natural question: ``{\it For which values of $m$ the $ Q_m(W)$'s are realizable}?'' has an immediate answer: {\it for all $m$}. 

A more interesting (and much less obvious) question is whether one can realize quasi-invariants
topologically as {\it a diagram of algebras} $ \{Q_m(W)\} $ (indexed by $m$) together with  natural structures that these algebras carry. It is one of the objectives of this work to formulate a realization problem for the algebras of 
quasi-invariants in a precise (axiomatic) form by selecting a list of the desired properties.
In the present paper, we state this problem for the classical Weyl groups (i.e., the crystallographic Coxeter groups over  $\mathbb{R}$ or $\C$) in terms of classifying spaces of compact Lie groups (see Section~\ref{realizationprob}); in our subsequent paper, we will try to formulate a $p$-local version of the realization problem for non-crystallographic (in fact, non-Coxeter) groups defined over the $p$-adic numbers in terms of $p$-compact groups.
We begin with a general overview and some motivation.

\vspace*{1ex}

\noindent
{\bf Quasi-invariants and cohomology theories.} 
In mathematical physics, the quasi-invariants naturally arise in three different flavors: rational (polynomial), trigonometric (exponential) and elliptic. Having in hand the topological spaces $ X_m(G,T)$ that realize the algebras $Q_m(W)$, it is natural to expect that the above three types of quasi-invariants correspond to three basic cohomology theories evaluated at $ X_m(G,T) $: namely, the ordinary (singular) cohomology, topological $K$-theory and elliptic cohomology. We will show that this is indeed the case: in fact, quasi-invariants can be defined for an arbitrary (complex-oriented generalized) cohomology theory, though in general their properties have yet to be studied.

\vspace*{1ex}

\noindent
{\bf Quasi-flag manifolds.} 
For a compact connected Lie group $G$, our spaces of quasi-invariants can be naturally realized as the Borel homotopy quotients of certain $G$-spaces  $ F_m(G,T) $:
$$ 
X_m(G,T) = EG \times_G F_m(G,T)
$$
We call $ F_m(G,T) $ the {\it $m$-quasi-flag manifold of $G$} as in the special case $ m = 0 $, we have $ F_0(G,T) =G/T $, the classical flag manifold. We remark that, in general, the spaces $ F_m(G,T)$ are defined only as $G$-equivariant homotopy types, although our construction provides some natural models
for them as finite $G$-CW complexes. By restricting the action of the Lie group $G$ on $ F_m(G,T)$ to its maximal torus $ T \subseteq G  $, it is natural to ask for {\it $T$-equivariant} cohomology (resp., $T$-equivariant $K$-theory, elliptic cohomology, \ldots) of $F_m(G,T)$. The $T$-equivariant cohomology is related to the $G$-equivariant one by the well-known general formula
\begin{equation}\la{GTW}
H^*_G(F_m, \C) \cong H^*_T(F_m, \C)^W , 
\end{equation}
where $W$ is the Weyl group associated to $(G,T)$. Since $ H^*_G(F_m, \C) = H^*(X_m, \C) \cong Q_m(W) $, formula \eqref{GTW} shows that the $W$-quasi-invariants can be, in fact, realized as $W$-invariants: $\, Q_m(W) \cong  H^*_T(F_m, \C)^W $  in the graded commutative algebras $ H^*_T(F_m, \C) $. The latter algebras come equipped with natural $ H^*_{T}({\rm pt}, \C)$-module structure induced by the canonical map $ F_m \to {\rm pt} $. Identifying $  H^*_{T}({\rm pt}, \C)  \cong \C[V] $ and taking onto account the  $W$-action, we can view $ H^*_T(F_m, \C) $ as modules over the crossed product algebra $ \C[V] \rtimes W $. We will show that these $\C[V] \rtimes W$-modules coincide --- up to a half-integer shift of multiplicities --- with the modules of {\it $\C W$-valued quasi-invariants}, $ \mathbf{Q}_{m+\frac{1}{2}}(W) $, introduced and studied in \cite{BC11}. As observed in \cite{BC11}, for {\it integer} $m$,
the action of $ \C[V] \rtimes W $ on $ \mathbf{Q}_{m}(W) $ naturally extends to  the rational double affine Hecke (a.k.a. Cherednik) algebra $ {\mathcal H}_m(W) $ associated to $(W,m)$. We will show that the topological construction of the quasi-flag manifolds $F_m(G,T)$ generalizes to half-integer values of $m$, although at the expense of producing spaces equipped only with $T$-action. By \cite{BC11}, we get then an action of $ {\mathcal H}_{m+1}(W) $ on the $T$-equivariant cohomology of $ F_{m+\frac{1}{2}}(G,T)$. This phenomenon seems to generalize to other cohomology theories, 
defining, in particular, an action of trigonometric (resp., non-degenerate) DAHA on $T$-equivariant $K$-theoretic (resp, elliptic) quasi-invariants. Constructing these actions algebraically and giving them a topological explanation is an interesting problem that we leave for the future.

\vspace*{1ex}

\noindent
{\bf Topological refinements.}
The realization of algebras of quasi-invariants raises many natural questions regarding topological analogues (`refinements') of basic properties that these algebras possess. A general framework to deal with such questions is provided by stable homotopy theory. Indeed, our spaces of quasi-invariants  $ X_m(G,T)$ are closely related to the classifying spaces of compact Lie groups, and the latter have been studied extensively in recent years by means of stable homotopy theory (see, e.g., \cite{DGI06}, \cite{BG14},  \cite{G18}, \cite{G20}, \cite{BCHV21}). From this perspective, the main object of study is the mapping spectrum  
\begin{equation}\la{mapsp}
C^*(X, k) := \Map\,(\Sigma^{\infty} X_{+},\,H k)  
\end{equation}
called the {\it cochain spectrum} of a topological space $X$. As its notation suggests, $ C^*(X,k) $  is a commutative ring spectrum that --- for an arbitrary commutative ring $k$ --- plays the same role as the usual (differential graded) $k$-algebra of cochains on $X$ in the case when $k$ is a field of characteristic zero. In particular, the (stable) homotopy groups of the spectrum \eqref{mapsp} are isomorphic to the singular cohomology groups of the space $X$:
$$
\pi_{-i}[C^*(X, k)]\,\cong\, H^{i}(X, k)
$$
The ring spectrum \eqref{mapsp} thus refines (in a homotopy-theoretic sense) the cohomology ring $H^*(X,k)$. For example,
if $G$ is a compact connected Lie group and $k$ is a field of characteristic $0$, the Borel Theorem mentioned above identifies $ H^\ast(BG, k) $ with the
algebra $ k[V]^W $ of invariant polynomials of $W$. The cochain spectrum $ C^*(BG, k) $  of the classifying space $BG$ can thus be viewed as a refinement of the algebra $ k[V]^W $. In the same manner,  we will regard the cochain spectra $ C^*(X_m(G,T), k) $ of our spaces $ X_m(G,T) $  as  homotopy-theoretic refinements of the algebras of quasi-invariants $ Q_m(W) $. The point is that the known algebraic properties of $ Q_m(W) $ should have topological analogues for $ C^*(X_m(G,T), k) $. For example, one of the main theorems about quasi-invariants (see Theorem~\ref{TGoren}) says that the (graded) algebras $ Q_m(W) $ defined over $\C$ are Gorenstein if $W$ is a Coxeter group. It is therefore natural to expect that the corresponding ring spectra $ C^*(X_m(G,T), k) $ are also Gorenstein --- but now in a {\it topological} sense \cite{DGI06} and over an arbitrary field $k$. We will show that this expectation is indeed correct, at least in the rank one case (see Theorem~\ref{ThGor} and Theorem~\ref{ThGor1}), and the spectra of quasi-invariants have a number of other interesting properties. Our results are only first steps in this direction, and many natural questions (motivated, in particular, by representation theory) have yet to be answered.

\vspace*{1ex}

\noindent
{\bf Homotopy Lie groups.} The spaces  of quasi-invariants of compact Lie groups, $ X_m(G,T) $, can be constructed functorially in a purely homotopy-theoretic way. In the rank one case, we use to this end the so-called {\it fibre-cofibre construction} --- a classical (though not very well-known) construction in homotopy theory introduced by T. Ganea \cite{Ga65}.  A generalization of Ganea's construction allows us to define the analogues of $ X_m(G,T) $ for certain finite loop spaces closely related to compact Lie groups, and perhaps most interestingly, for $p$-compact  groups --- $p$-local analogues of finite loop spaces also known as {\it homotopy Lie groups}. In this last case, the classical Weyl groups are replaced by pseudo-reflection groups defined over the field  $\Q_p$ of $p$-adic numbers. It is well known that all such pseudo-reflection groups can be realized as complex reflection groups (see \cite{CE74}), and we thus  provide realizations of algebras of quasi-invariants of complex reflection groups defined in \cite{BC11},
albeit in a $p$-local setting. The simplest exotic examples are the rank one $p$-compact groups 
$\, \hat{\bS}^{2n-1}_p $, called the {\it Sullivan spheres}, whose `Weyl groups' are the cyclic groups $ W = \Z/n $ of order $ n>2 $ such that $\,n\,|\,(p-1)$. These examples are already quite rich: we will treat them in a separate paper.

\vspace*{1ex}

We divide our work into three parts. The present paper (Part I) focuses entirely on the
`global' rank one case: here, we define and study the spaces of quasi-invariants for the Lie group $  G = SU(2) $ and for a certain class of finite loop spaces $ \Omega B $
of homotopy type of $\, \bS^3 $ known as {\it Rector spaces}.
In Part II, we formulate a $p$-local version of the realization problem
for algebras of quasi-invariants defined over $ \Q_p $ and give its solution in the `local'
rank one case: namely, for the $p$-compact groups associated with Sullivan spheres $\, \hat{\bS}^{2n-1}_p $. In Part III, we then use the spaces introduced in
Part I and Part II as `building blocks' for constructing spaces of quasi-invariants for arbitrary compact connected Lie groups and for `generic' $p$-compact groups related to Clark-Ewing spaces. 

\subsection*{Contents of the present paper}
We now describe in more detail the results of the present paper. 
In Section~\ref{S2}, after  reviewing basic facts about quasi-invariants, we state our realization problem for Weyl groups in the classical framework of compact connected Lie groups. As mentioned above, we take an axiomatic approach: the properties that we choose to characterize the topological spaces of quasi-invariants are modeled on properties of algebraic varieties of quasi-invariants introduced and studied in \cite{BEG03}.
In fact, our main axioms (QI${}_1$)--(QI${}_5$) in
Section~\ref{realizationprob} are natural homotopy-theoretic analogues of basic geometric properties of the varieties of quasi-invariants listed in Section~\ref{Varquasi}.

In Section~\ref{S4}, we give a solution of our realization problem for $ G = SU(2) $ (see Theorem~\ref{MTh1}). To this end, as mentioned above, we employ the Ganea fibre-cofibre construction. This construction plays an important role in abstract homotopy theory (specifically, in the theory of LS-categories and related work on the celebrated Ganea Conjecture in algebraic topology, see e.g. \cite{CLOT03} and Example~\ref{exga} below). However, we could not find any applications of it in Lie theory or classical homotopy theory of compact Lie groups (perhaps, with the exception of the simple (folklore) Example~\ref{exmil}). We therefore regard Proposition~\ref{PP1} and Theorem~\ref{MTh1}  that describe the Ganea tower of the Borel maximal torus fibration of a compact connected Lie group as original contributions of the present paper. The $G$-spaces $ F_m(G,T) $ that we call the {\it $m$-quasi-flag manifolds of $G$} are defined to be the homotopy fibres of  
iterated (level $m$) fibrations in this Ganea tower (see Definition~\ref{qspaces}). In Section~\ref{S4.31} and Section~\ref{S4.4}, we describe some basic properties of the
$G$-spaces $F_m(G,T)$. First, we compute the $T$-equivariant cohomology of $ F_m(G,T) $ (see Proposition~\ref{PTcoh}) and identify it with a module of `nonsymmetric' ($\C W$-valued) quasi-invariants (see Corollary~\ref{RelBC}). In this way, we provide a topological interpretation of generalized quasi-invariants  introduced in \cite{BC11}.
Then, in Section~\ref{S4.4},  we define natural analogues of the classical Demazure (divided difference) operators for our quasi-flag manifolds $F_m(G,T)$. Our construction is purely topological (see Proposition~\ref{P22}): it generalizes the Bressler-Evans construction of the divided difference operators for the classical flag manifolds $ F_0(G,T) $ given in \cite{BE90}. 

In Section~\ref{S5}, we extend our topological construction of spaces of quasi-invariants to a large class of finite loop spaces $ \Omega B $ called the {\it Rector spaces} (or  fake Lie groups of type $SU(2)$). These remarkable loop spaces were originally constructed in \cite{Rec71}  as examples of nonstandard (`exotic') deloopings of $ \bS^3 $. Our construction does not apply to all Rector spaces, but only to those that accept homotopically nontrivial maps from $ \C \bP^\infty $. These last spaces admit a beautiful arithmetic characterization discovered by D. Yau in \cite{DY01}. We show that the `fake' spaces of quasi-invariants, $X_m(\Omega B,T) $, associated to the Rector-Yau spaces have the same {\it rational} cohomology as our `genuine' spaces of quasi-invariants $ X_m(G,T)$ (see Theorem~\ref{CohXMR}); however, in general, they are homotopically non-equivalent (see Corollary~\ref{rsdistinct}).

In Section~\ref{S6},  we compute the $G$-equivariant (topological) $K$-theory $ K^*_G(F_m)$ of the spaces $ F_m = F_m(G,T)$ and identify it with $ \cQ_m(W) $, the {\it exponential quasi-invariants} of the Weyl group $W = \Z/2\Z $ (see Theorem~\ref{ektmq}). Then, we relate $K^*_G(F_m) $ to the (completed) $G$-equivariant cohomology
$\,\widehat{H}^{\ast}_G(F_m, \Q) := \prod_{k=0}^\infty H^k_G(F_m, \Q)\,$ by constructing explicitly the $G$-equivariant Chern character map
\begin{equation}
\la{chGFI}
\ch_G(F):\ K^*_G(F_m) \,\to\, \widehat{H}^{\ast}_G(F_m, \Q)
\end{equation}
We show that  \eqref{chGFI} factors through the natural map $K^*_G(F_m) \to K^*(X_m) $ to the Borel $G$-equivariant $K$-theory $ K^*(X_m) = K^*(EG \times_G F_m) $ of $F_m$, inducing an isomorphism upon rationalization (see Proposition~\ref{ChMap}):
$\, K^*(X_m)_{\Q} \cong \widehat{H}^{\ast}_G(F_m, \Q) \cong \widehat{Q}_m(W)\,$. In this way, we link topologically the exponential and the usual (polynomial) quasi-invariants of $W$.
In Section~\ref{S6}, we also compute the $K$-theory of `fake' spaces
of quasi-invariants associated to the Rector-Yau loop spaces $ \Omega B $ (see Theorem~\ref{RectorQI}). The result of this computation has an important consequence --- Corollary~\ref{rsdistinct} --- that provides
a numerical $K$-theoretic invariant $ N_B $  distinguishing the spaces $ X_m(\Omega B,T) $  up to homotopy equivalence for different $B$'s.

In Section~\ref{S7},  we compute the $T$-equivariant $ \Ell_T^*(F_m) $ and $G$-equivariant $\Ell_G^*(F_m) $ complex analytic elliptic cohomology of $F_m$  (see Theorem~\ref{ThTell} and Theorem~\ref{ThElT}, respectively). We express the result in two ways: geometrically (as coherent sheaves on a given Tate elliptic curve $E$) and analytically (in terms of $\Theta$-functions and $q$-difference equations). We also compute the spaces (graded modules) of global sections of the elliptic cohomology sheaves of $F_m$ with {\it twisted} coefficients:
$$
\ELL_T^*(E, {\mathscr L}) :=
\bigoplus_{n=0}^{\infty}\, H_{\rm an}^0(E,\,\Ell_T^*(F_m) \otimes {\mathscr L}^{n})\quad 
\mbox{and}\quad  
\ELL_G^*(E, {\mathscr L}) := \ELL_T^*(E, {\mathscr L})^W\ ,
$$
where $ {\mathscr L}^n $ stands for the $n$-th tensor power of the {\it Looijenga bundle} ${\mathscr L}$ on the elliptic curve $E$, a canonical $W$-equivariant line bundle originally introduced and studied in \cite{Lo77}. This computation (see Theorem~\ref{EllTw}) is inspired by results of \cite{Ga14}, and technically, it is perhaps the most interesting cohomological computation of the paper.

Finally, in Section~\ref{S8}, we prove that our spaces of quasi-invariants $ X_m(G,T) $ are Gorenstein in the sense of stable homotopy theory: more precisely, the associated commutative ring spectra $ C^*(X_m, k) $ (see \eqref{mapsp}) are orientable Gorenstein (relative to $k$) and satisfy the Gorenstein duality of shift $ a = 1 - 4m $ (see Theorem~\ref{ThGor}). This result should be viewed as a homotopy-theoretic analogue of Theorem~\ref{TGoren} on Gorenstein property of algebras of quasi-invariants. We also prove the analogous result 
(see Theorem~\ref{ThGor1}) for the `fake' spaces of quasi-invariants $ X_m(\Omega B,T) $, although under the additional assumption that $ k = \F_p $ for some prime $p$. 

This work combines ideas and techniques from different parts of algebra and 
topology. To make it accessible to readers with different background we 
tried to include basic definitions and give references to all essential facts 
that we are using (even when these facts are considered to be obvious or well 
known by experts).

\subsection*{Appendix}
After the first version of this paper appeared in arXiv, we became aware of a very interesting
unpublished preprint \cite{FF}, which we reproduce ---  with authors' permission and minor corrections --- in Appendix \ref{ArMapCon}. In \cite{FF}, M. Feigin and K. Feldman give an explicit differential-geometric construction of spaces of quasi-invariants of rank one,  using the so-called Arnold-Maxwell Theorem (see \cite{Arnold1996}).   This last theorem provides, for any $ n \ge 1 $, a natural diffeomorphism $(\c\bP^1)^n\!/W \cong \bS^{2n}$, where $W = D(n) $ is the Coxeter group of type $D$ of rank $n$. The standard linear action of $U(2)$ on $\c\bP^1$ extends to the diagonal action on $(\c\bP^1)^n$, which, in turn, descends to the quotient $(\c\bP^1)^n/W$. This gives a $U(2)$-action on $\bS^{2n}$ for all $n\ge 1$, which we can restrict to $SU(2) \subset U(2) $. 
(We shall refer to this $SU(2)$-action as the Feigin-Feldman action). The main result of the Appendix (Theorem \ref{heqar}) asserts that, for $ n = 2m+1$, the rational $SU(2)$-equivariant cohomology of $\bS^{2(2m+1)}$ equipped with the Feigin-Feldman action is isomorphic to $Q_m(W)$. 
One might expect that this explicit construction also yields a solution to our
realization problem for $G =SU(2)$, satisfying the axioms (QI$_{1}$)-(QI$_{5}$) of Section~\ref{realizationprob}. 
Unfortunately, this is not the case. As pointed out in the Appendix, the above construction
does not appear to be natural in $m \in \Z_+$: there seem to exist no $SU(2)$-equivariant maps $\,\bS^{2(2m+1)} \to \bS^{2(2m' +1)} $, inducing the inclusions $ Q_{m'}(W) \into Q_{m}(W) $ for $ m' > m $. A comparison with the $G$-spaces $F_m(G,T)$ constructed in this paper shows (see Remark~\ref{fmvsArnld}) that there are no $T$-equivariant (and hence, {\it a fortiori}, \,no $G$-equivariant) isomorphisms between $F_m(G,T)$ and the spheres $\bS^{2(2m+1)}$ equipped with the Feigin-Feldman action. 
However, one can construct a natural $T$-equivariant map $ \varphi: F_m(G,T) \to \bS^{4m+2}$ 
(see \eqref{fmbract4}) that does induce an isomorphism on rational $T$-equivariant  cohomology. 

\subsection*{Acknowledgements}
We would like to thank Oleg Chalykh and Pavel Etingof  for many interesting discussions, 
questions and comments related to the subject of this paper. We are particularly grateful to O. Chalykh for clarifying to us his definition of quasi-invariants in the elliptic case (see Remark~\ref{RemEll}). We are also very grateful to Misha Feigin for sending us his unpublished paper \cite{FF}. 
The work of the first author was partially supported by NSF grant DMS 1702372 and the Simons Collaboration Grant 712995. The second author was partially supported by NSF grant DMS 1702323. 

\section{Realization problem}
\la{S2}
In this section, we state our topological realization problem for algebras of quasi-invariant polynomials of Weyl groups in terms of classifying spaces of compact connected Lie groups. 

\subsection{Quasi-invariants of finite reflection groups}
\la{S2.1}
We recall the general definition of quasi-invariants from \cite{BC11}. Let $V$ be a finite-dimensional vector space over $\C$, and let $W$ be a finite subgroup of $\GL(V)$ generated by pseudoreflections. We recall that an element $ s \in \GL(V) $ is a {\it pseudoreflection} if it has finite order $ n_s > 1 $ and acts as the identity on some hyperplane $ H_s $ in $V$. We let $\, \A =\{H_s\}\,$ denote the set of all hyperplanes corresponding to the pseudoreflections of $ W $ and observe that  $ W $ acts naturally on $ \A $ by permutations. The (pointwise) stabilizer $W_H$ of $\, H \in \A \,$ in $W$ is a cyclic subgroup of order $ n_H \ge 2$ that depends only on the orbit of $H$ in $\A $. The characters of $W_H$ then also form a cyclic group of order $n_H$ generated by the determinant character 
$ \det: \, \GL(V) \to \C^* $ of $ \GL(V) $ restricted to $W_H$. We write
$$
\e_{H,\, i} := \frac{1}{n_H}\,\sum_{w \in W_H} (\det w)^{-i} \, w \ , \quad
i = 0,\,1,\,\ldots\,,\,n_H -1\ ,
$$
for the corresponding idempotents in the group algebra $ \C W_H  \subseteq \C W$. 

Now, let $\C[V] = {\rm Sym}_{\C}(V^*)$ denote the polynomial algebra of $V$. This algebra carries a natural $W$-action (extending the linear action of $W$ on $V^*$) and
can thus be viewed as a $\C W$-module. We can then characterize the invariant polynomials $ p \in \C[V]^W $  by the equations 
\begin{equation}
\la{INV} 
\e_{H, -i}(p)\,=\,0\ , \quad
i = 1,\,\ldots\,,\,n_H -1\ ,
\end{equation}
which hold for all hyperplanes $ H\in \A$. To define quasi-invariants we relax the equations \eqref{INV} in the following way ({\it cf.} \eqref{Qinv}). 
For each hyperplane $ H \in \A $, we fix a linear form $\,\alpha_H \in V^* $, such that $\,H = \Ker(\alpha_H) \,$, and choose $\, n_H - 1 \,$ positive integers
$\, \{m_{H,i}\}_{i = 1, \ldots,\, n_H - 1} $ which we refer to as {\it multiplicities} of $H$. We assume that $\, m_{H,i} = m_{H',i}\,$ for each $i$ whenever $ H $ and $ H' $ are in the same orbit of $ W $ in $\A$. 
We write $ \M(W) := \{m_{H,i} \in \Z_{+}\,:\,  i = 1, \ldots,\, n_H - 1\}_{[H] \in \A/W} $ for the set of all such multiplicities regarding them as functions on the set $\A/W$ of $W$-orbits in $\A$.
\begin{defi}[\cite{BC11}] 
\la{comq}
A polynomial $ p \in \c[V]$ is called a {\it $W$-quasi-invariant 
of multiplicity} $ m = \{m_{H, i}\} \in \M(W) $ if it satisfies the conditions
\begin{equation}
\label{IN1} 
\e_{H, -i} (p)\, \equiv 0\ \, \mbox{mod}\,\langle\alpha_H
\rangle^{n_H m_{H,i}}  \ , \quad
i = 1,\,\ldots\,,\,n_H -1\ ,
\end{equation}
for all  $ H \in \A $. We write $ Q_m(W) $ for the subspace of all such polynomials in $\C[V]$.
\end{defi}
In general, $Q_m(W)$ is {\it not} an algebra: for arbitrary $W$ and $m \in \M(W)$, the subspace of quasi-invariant polynomials may not be closed under multiplication in $ \C[V]$. In Part II of our work, we will give necessary and sufficient conditions (on $W$ and $m$) that ensure the multiplicativity property of $Q_m(W)$.
In the present paper, we simply restrict our attention to {\it Coxeter groups}, i.e. the finite subgroups $W$ of $\GL(V)$ generated by {\it real} reflections. In this case the conditions \eqref{IN1} are equivalent to \eqref{Qinv} and the above definition of quasi-invariants reduces to the original definition of Chalykh and Veselov \cite{CV90} given in the Introduction. 

\vspace{1ex}

{\it Thus, from now on, we assume that $W$ is a real finite reflection group, 
$V$ being its $($complexified$)$ reflection representation.} 

\vspace{1ex}

The next lemma collects some elementary properties of quasi-invariants that follow easily from the definition (see, e.g., \cite{BEG03}).
\begin{lemma}
\la{L1}
Let $W$ be an arbitrary Coxeter group. Then, for any $\, m \in \M(W)\,$, 

$(1)$ $\,\c[V]^W \subset Q_m(W) \subseteq \c[V]\,$ with $\, Q_0(W) = \c[V] \,$ and $\,\cap_{m} Q_m(W) = \c[V]^W $.

$(2)$ $ Q_m(W) $ is a graded subalgebra of $ \c[V] $ stable under the action of $W$.

$(3)$ $ Q_m(W) $ is a finite module over $ \c[V]^W $ and hence a finitely generated $\c$-subalgebra
of $\c[V]$. 

\end{lemma}
We may think of quasi-invariants of $W$ as a family of subalgebras of $\c[V]$ interpolating between the $W$-invariants and all polynomials. 
To make this more precise we will identify
the set $\M(W)$ of multiplicities on $\A$ with the set of $W$-invariant functions $ m: \A \to \Z_+ $ and put on this set the following natural partial order\footnote{Abusing notation, in the Coxeter case, we will often write $ \alpha \in \A $ instead of $ H \in \A $ 
for $ H = \Ker(\alpha)$.}:
$$
m \le  m' \quad \stackrel{\rm def}{\Longleftrightarrow} \quad m_{\alpha} \le m_{\alpha}'\ ,\ \forall\, \alpha \in \A\, ,
$$
The algebras of $W$-quasi-invariants of varying multiplicities then form a contravariant diagram of shape $\M(W) $ --- a functor $\,\M(W)^{\rm op} \to {\tt Comm Alg}_{\C} $ with values in the category of commutative algebras --- that we simply depict as a
filtration on $ \C[V] $:
\begin{equation}
\la{Qtow}
\c[V] = Q_0(W)\, \supseteq\, \ldots\, \supseteq\, Q_m(W)\, \supseteq \,Q_{m'}(W) \,\supseteq\, \ldots \, \supseteq \,\c[V]^W
\end{equation}

The most interesting algebraic property of quasi-invariants is given by the following theorem, the proof of which (unlike the proof of Lemma~\ref{L1}) is not elementary.
\begin{theorem}[see \cite{FV02}, \cite{EG02b}, \cite{BEG03}]
\la{TGoren}
For any Coxeter group $W$ and any multiplicity $ m \in \M(W) $, $\,Q_m(W) $ is a free module over $ \c[V]^W $ of rank $ |W| $. Moreover, $\,Q_m(W) $ is a graded Gorenstein algebra with Gorenstein shift $\,a = \dim(V) - 2 \sum_{\alpha \in \A} m_{\alpha} \,$. 
\end{theorem}
\begin{remark}
For $m=0$ (i.e., for the polynomial ring $\, Q_0(W) = \c[V]$), Theorem~\ref{TGoren} is a well-known result due to C. Chevalley \cite{Che55}. For $ m \not= 0 $, it was first proven in the case of dihedral groups (i.e. Coxeter groups of rank 2) in \cite{FV02}. For arbitrary Coxeter $W$, Theorem~\ref{TGoren} was proven (by different methods) in \cite{EG02b} and \cite{BEG03}. It is worth mentioning that the classical arguments of \cite{Che55}  do not work for nonzero $m$'s.
\end{remark}
\begin{remark}
The first statement of Theorem~\ref{TGoren} makes sense and holds true for an arbitrary finite pseudoreflection group $W$ and for all multiplicities. In this generality, Theorem~\ref{TGoren} was proven in \cite{BC11} (see, {\it loc. cit.}, Theorem~1.1). However, for $W$ non-Coxeter, the module $Q_m(W)$ may not be Gorenstein even when it is an algebra.
\end{remark}
\subsection{Varieties of quasi-invariants}\la{varquasi}
\la{Varquasi}
The algebraic properties of quasi-invariants can be recast in geometric terms. To this end, following \cite{BEG03}, we introduce the affine schemes 
$\, V_m(W) := \Spec\,Q_m(W) \,$ called the {\it varieties of quasi-invariants} of $W$. The schemes $V_m(W)$ come equipped with natural projections $ p_m: V_m(W) \to V/\!/W$ and form a covariant diagram (tower) over the poset $\M(W) $:
\begin{equation}
\la{Qsch}
V =V_0(W) \to \ldots \to   V_m(W)  \xrightarrow{\pi_{m,m'}}  V_{m'}(W) \to \ldots 
\end{equation}
that is dual to \eqref{Qtow}. The following formal properties of \eqref{Qsch} hold:
\vspace{1ex}
\begin{enumerate}
  \item[(1)] Each $V_m(W) $ is a reduced irreducible scheme (of finite type over $\c$) equipped with an algebraic $W$-action, all morphisms in \eqref{Qsch} being $W$-equivariant. The morphism $ p_0: V_0(W) \to V/\!/W$  coincides with the canonical projection $ p: V \to V/\!/W $, and the  triangles 
\begin{equation*}
\la{Qtri}
\begin{diagram}[small] 
                V_m(W) & & \rTo^{\pi_{m,m'}} &           & V_{m'}(W)\\
                      &\rdTo_{p_m}  &               & \ldTo_{p_{m'}}     & \\
                     &        &        V/\!/W   &           & \\
               \end{diagram}  
\end{equation*}
commute for all $m' \ge m $. Thus, \eqref{Qsch} is a diagram of $W$-schemes over $V/\!/W$.\\

\item[(2)]
The diagram \eqref{Qsch} `converges' to $V/\!/W$ in the sense that the maps $p_m$ 
induce
\begin{equation*} 
\la{converge} 
\mathrm{colim}_{\M^{\rm alg}(W)}[V_m(W)]\, \stackrel{\sim}{\to}\, V/\!/W \, .
\end{equation*}

\item[(3)] Each projection $ p_m: V_m \to  V/\!/W $ factors naturally (in $m$) through $ V_m/\!/W $,
inducing isomorphisms of schemes $\,V_m/\!/W \cong V/\!/W \,$ for all  $\, m \in \M(W)\,$.\\

\item[(4)]
Each map $ \pi_{m,m'}: V_m \to V_{m'} $ in \eqref{Qsch} is a universal homeomorphism: i.e., a finite morphism of schemes that is surjective and set-theoretically injective on closed points. 
\end{enumerate}
\vspace{1ex}

\begin{remark}
The first three properties in the above list are formal consequences of Lemma~\ref{L1}.  
In contrast, Property (4) is a nontrivial geometric fact that does not follow immediately from definitions (see \cite[Lemma 7.3]{BEG03}). We recall that a morphism of schemes $ f: S \to T $ is called a {\it universal homeomorphism} if for every morphism $ T' \to T $ the pullback map $ T' \times_T S \to T' $ is a homeomorphism in the category of schemes. For a map of algebraic varieties $ f: S \to T $ defined over $\c $, this categorical property is known to be equivalent to the geometric property (4).
\end{remark}

We will construct a topological analogue of the diagram \eqref{Qsch}, where the schemes $ V_m(W) $ are replaced by topological spaces $X_m(G,T)$, with Properties (1)--(3) holding 
in a homotopy meaningful (i.e. homotopy invariant) way. The universal homeomorphisms in the category of schemes will be modeled homotopy theoretically by the classical fibre-cofibre construction.

\subsection{Borel Theorem}
\la{S3.2}
Next, we recall a fundamental result of A. Borel on cohomology of classifying spaces of compact Lie groups  \cite{Bo53}.
Let $G$ be a compact connected Lie group. Fix a maximal torus $ T \subseteq G $ and write $ N = N_G(T)$ for its normalizer in $G$. Let $W:= N/T$ be the associated Weyl group. The $W$ acts naturally on $T$ by conjugation: $W \times T \rightarrow T$, $\,w \cdot t = ntn^{-1}$, and on the classifying space $BT=EG/T$ via the right action of $G$ on $EG$: $W \times BT \rightarrow BT$, $w \cdot [x]_T = [x n^{-1}]_T$, where $w=nT \in W$ and $[x]_T$ denotes the $T$-orbit of $x$ in $EG$. Let $p: BT \twoheadrightarrow BG$ denote the natural fibration, 
i.e. the quitient map induced by the inclusion $T \hookrightarrow G$.

\begin{theorem}[Borel] \la{tborel}
The map $p^{\ast} : H^\ast(BG , \Q) \to H^{\ast}(BT , \Q)$ induced by $p$ on rational cohomology is an injective ring homomorphism whose image is precisely the subring of $W$ invariants in $ H^{\ast}(BT, \Q)\,$:
\begin{equation} \la{borelbg} H^\ast(BG , \Q) \cong H^{\ast}(BT , \Q)^{W}\ . \end{equation} 
\end{theorem}

In fact, more is true. Let $ V := \pi_1(T) \otimes \Q $, which is a $\Q$-vector space of dimension $ n = {\rm rank}(G)$. The natural action of $W$ on $T$ induces a group homomorphism $W \to \mathrm{Aut}[\pi_1(T)]$ that extends by linearity to a group homomorphism
\begin{equation} \la{weylglq} \varrho:\, W \to \GL_{\Q}(V) \ .\end{equation}
The latter is known to be faithful, with image being a reflection subgroup of $\GL_{\Q}(V)$ (see, e.g., \cite[Theorem 5.16]{DW98}). 
Furthermore, since $T$ is a connected topological group, there is a natural isomorphism $\pi_1(T) \cong \pi_2(BT)$ induced by the homotopy equivalence $T \stackrel{\sim}{\to} \Omega B T$; combining this with the rational Hurewicz isomorphism  $\pi_2(BT) \otimes \Q \cong H_2(BT, \Q)$, we get a natural isomorphism of $\Q$-vector spaces 
\begin{equation} \la{borel1} V \cong H_2(BT , \Q)\end{equation}
which shows that $H_2(BT, \Q)$ carries a reflection representation of $W$ as a Coxeter group. Dualizing \eqref{borel1} gives an isomorphism 
\begin{equation} 
\la{borel2}
H^2(BT, \Q) \cong V^{\ast} 
\end{equation}
which extends to an isomorphism of graded $\Q$-algebras 
\begin{equation} 
\la{borel3} 
H^{\ast}(BT , \Q) \cong \mathrm{Sym}_{\Q}(V^{\ast}) = \Q[V] 
\end{equation}
where the linear forms on $V$ (covectors in $V^{\ast}$) are given cohomological degree $2$ (in agreement with \eqref{borel2}). Borel's Theorem \ref{tborel} thus identifies $H^{\ast}(BG , \Q)$ with the ring $\Q[V]^W$ of polynomial invariants on the (rational) reflection representation of $W$. 

We are now in a position to state our main problem --- the realization problem for algebras of quasi-invariants of Weyl groups --- in an axiomatic way.

\subsection{Realization problem} 
\la{realizationprob}
Given a compact connected Lie group $G$ with maximal torus $T \subseteq G$ and associated Weyl group $ W = W_G(T)$, construct a diagram of spaces $X_m(G,T)$ over the poset  $\M(W)$:
\begin{equation} 
\la{realize}
BT=X_0(G,T) \to \ldots \to X_m(G,T) \xrightarrow{\pi_{m,m'}} X_{m'}(G,T) \to \ldots 
\end{equation}
together with natural maps $p_m: X_m(G,T) \to BG$, one for each $m \in \M(W) $, such that

\vspace{1ex}

\begin{enumerate}
  \item[(QI$_{1}$)]  Each $X_m(G,T) $ is a $W$-space (i.e., a CW complex equipped with an action of $W$), and all maps are $W$-equivariant. The map $\, p_0: X_0(G,T) \to BG \,$ coincides with the canonical map $ p: BT \to BG $, and for all $\, m' \ge m \,$, the following diagrams commute up to homotopy:

               $$ \begin{diagram}[small] 
                X_m(G,T)& & \rTo^{\pi_{m,m'}} &           & X_{m'}(G,T)\\
                      &\rdTo_{p_m}  &               & \ldTo_{p_{m'}}     & \\
                     &        &  BG         &           & \\
               \end{diagram}  
              $$
Thus, \eqref{realize} is a diagram of $W$-spaces  over $BG$.

\vspace{1ex}

\item[(QI$_{2}$)]
The diagram \eqref{realize} `converges' to $BG$ in the sense that the maps $p_m$ 
induce a weak homotopy equivalence of spaces:
\begin{equation*} 
\la{converge} 
\mathrm{hocolim}_{\M(W)}[X_m(G,T)] \stackrel{\sim}{\to} BG \, .
\end{equation*}

\item[(QI$_{3}$)] Each map $\, p_m:\, X_m(G,T) \to BG \,$ factors naturally (in $m$) through the fibre inclusion into the space $ X_m(G,T)_{hW} $ of homotopy orbits of the action of $W$ on $ X_m(G,T) $:

$$ \begin{diagram}[small] 
                X_m(G,T)& & \rTo^{p_m} &           & BG \\
                      &\rdTo_{i_m}  &               & \ruTo_{\bar{p}_{m}}     & \\
                     &        &    X_m(G,T)_{hW}       &           & \\
               \end{diagram}  
              $$
the induced map $\, \bar{p}_m:\, X_m(G,T)_{hW} \to BG \,$ being a cohomology isomorphism: thus, for all $ m \in \M(W) $,
we have algebra isomorphisms
$$
H_W^*(X_m,\,\Q) \cong H^*(BG,\, \Q) 
$$

\vspace{1ex}

\item[(QI$_{4}$)] Each map $\pi_{m,m'}$ in \eqref{realize} induces an injective homomorphism on cohomology so that the Borel homomorphism $p^\ast$ factors into a $\M(W)^{\mathrm{op}}$-diagram of algebra maps 
\begin{equation*} 
\la{factorborel} 
{H}^{\ast}(BG,\Q) \hookrightarrow \ldots \hookrightarrow {H}^{\ast}(X_{m'}, \Q) \stackrel{\pi_{m,m'}^{\ast}}{\hookrightarrow} {H}^{\ast}(X_m,\Q) \hookrightarrow \ldots \hookrightarrow {H}^\ast(BT,\Q) 
\end{equation*}
%

\vspace{1ex}

\item[(QI$_{5}$)]
With natural identification $ H^*(BT,\Q) = \Q[V] $ (see \eqref{borel3}), the maps $\pi_{0,m}^{\ast}: H^{\ast}(X_m, \Q) \to H^*(BT, \Q) $ in (QI$_{4}$) induce isomorphisms
\begin{equation*} \la{cohqi} 
H^{\ast}(X_m, \Q) \otimes \c \cong Q_m(W)
\end{equation*}
where $Q_m(W)$ are the subalgebras of quasi-invariants in $\c[V]$.
\end{enumerate}

%
%
\begin{remark} 
The first three properties of the spaces $X_m(G,T) $ are homotopy-theoretic analogues of the corresponding geometric properties of the varieties  $V_m(W)$ listed in Section~\ref{varquasi}. Properties (QI$_{4}$) and (QI$_{5}$) reflect the fact that the diagram \eqref{realize} is a topological realization of the diagram of algebras \eqref{Qtow}: in particular, the maps $\pi_{m,m'}^* $ in  (QI$_{4}$) induced by the cohomology functor correspond to the natural inclusions \eqref{Qtow} of algebras $Q_m(W)$ determined by their multiplicities.
\end{remark}

\begin{remark}
The spaces $X_m(G,T)$ will arise naturally as homotopy $G$-orbit spaces 
$$ 
X_m(G,T) = EG \times_G F_m(G,T)\,,
$$ 
where $ F_m(G,T) $ are the homotopy fibres of the maps $\,p_m: X_m(G,T) \to BG \,$ (see  Theorem~\ref{MTh1}).
These homotopy fibres form a diagram of $G$-spaces 
$$ 
G/T = F_0(G,T) \to \ldots \to F_m(G,T) \to F_{m'}(G,T) \to \ldots  
$$  
that induces the diagram \eqref{realize}. We will call $ F_m(G,T) $ the {\it quasi-flag manifolds} of the group $G$.
\end{remark}

\section{Spaces of quasi-invariants}
\la{S4}
In this section, we give a solution of our Realization Problem (see Section~\ref{realizationprob}) in the rank one case. Our main observation (see Proposition~\ref{PP1} and Theorem~\ref{MTh1}) is that, for  $G = SU(2)$, the diagram of spaces  \eqref{realize} satisfying all five axioms  (QI$_{1}$)-(QI$_{5}$) can be obtained inductively, using the so-called `fibre-cofibre construction' introduced in homotopy theory by T. Ganea \cite{Ga65}.  

\subsection{Fibres, cofibres and joins}
\la{S3.0}
First, we recall some basic definitions from topology. If $ f: \,X \to Y $ is a map of (well) pointed spaces, its {\it homotopy fibre} is defined by
$$
{\rm hofib}_*(f) :=
X \times_{Y} P_*(Y) = \{(x,\gamma)\,:\, \gamma(0) = \ast\ , \gamma(1) =f(x)\}\ ,
$$
where $ P_*(Y) := {\rm Map}_*(I,Y) = \{\gamma:\, I \to Y\,,\ \gamma(0) = \ast\}\, $ is the (based) path space over $Y$. 
Any map $ f: \,X \to Y $   can be replaced by a fibration in the sense that it admits a factorization
$ X \stackrel{\sim}{\to} X' \stackrel{p}{\onto} Y $ in $ \Top_* $, where the first arrow is a weak homotopy equivalence and the second is a (Serre) fibration. The homotopy fibre is a homotopy invariant in $ \Top_* $ so that the pullback along a weak equivalence  $ X \stackrel{\sim}{\to} X' $ induces a weak equivalence: $\,{\rm hofib}_*(f) \stackrel{\sim}{\to} {\rm hofib}_*(p) \,$.
On the other hand, for any fibration $ p: X' \onto Y $, the natural inclusion map
$$
p^{-1}(\ast) \stackrel{\sim}{\to} {\rm hofib}_*(p)\ ,\quad x \mapsto (x, \ast)
$$
is a (based) homotopy equivalence. Thus, the homotopy fibres of fibrations can be represented in $ \Ho(\Top_*) $ by usual (set-theoretic) fibres.

Dually, the {\it homotopy cofibre} of a map $ f: \,X \to Y $ is defined by
$$
{\rm hocof}_*(f) := Y \cup_{X} C_*(X)\, ,
$$
where $ C_*(X) := (X \times I)/(\{*\} \times I) \cup (X \times \{1\}) $ is the reduced cone on $X$. 
Any map $ f: \,X \to Y $ can be replaced by a cofibration in the sense that it admits a factorization
$ X \stackrel{j}{\into} Y' \stackrel{\sim}{\to} Y $ in $ \Top_* $, where the first arrow is a cofibration (i.e., an injective map) in $ \Top_* $ and the second is (weak) homotopy equivalence. The homotopy cofibre is a homotopy invariant so that the pushout along the homotopy equivalence  $ Y' \stackrel{\sim}{\to} Y $ induces an equivalence: $\,{\rm hocof}_*(j) \stackrel{\sim}{\to} {\rm hocof}_*(f) \,$. On the other hand, for a cofibration $ j: X \into Y' $, the homotopy cofibre $\, {\rm hocof}_*(j)\, $ is simply obtained by erecting the cone $ C_*(\im \,j) $ on the image of $ j $ in $ Y' $. The natural map collapsing this cone to the basepoint gives then a natural map
$$
{\rm hocof}_*(j) \cong Y' \cup C_*(\im\, j) \,\stackrel{\sim}{\to} \,Y'/X
$$
which is a (based) homotopy equivalence. Thus, the homotopy cofibres of cofibrations can be represented in $ \Ho(\Top_*) $ by usual (set-theoretic) cofibres.

Formally, $ {\rm hofib}_*(f) $ and  $ {\rm hocof}_*(f) $ can be defined in $\Ho(\Top_*)$ by the following homotopy limit and homotopy colimit:
\begin{equation}
\la{limcolim}
{\rm hofib}_*(f) = {\rm holim}\{\ast \to Y \xleftarrow{f} X\}\quad , \quad
{\rm hocof}_*(f) = {\rm hocolim}\{\ast \leftarrow X \xrightarrow{f} Y\}
\end{equation}
The advantage of these formal definitions is that they make sense in any homotopical context: in particular, in an arbitrary pointed model category or $\infty$-category.

Next, we recall that the {\it join} $ X \ast Y $ of two spaces is defined to be the space of all line segments joining points in $X$ to points in $Y$: i.e., $ X \ast Y $ is the quotient space of $ X \times I \times Y $ under the identifications $\,(x, 0, y) \sim (x', 0, y)\,$ and $(x,1,y) \sim (x,1,y') $ for all $x, x' \in X $ and $y,y' \in Y $. If $X$ and $Y$ are both (well) pointed, it is convenient to work with a {\it reduced} version of the join obtained by collapsing to a point the line segment joining the basepoints in $X$ and $Y$ (i.e., by imposing on $\,X \ast Y\,$ the extra  identification $(\ast, t,\ast) \sim (\ast, t', \ast) $ for all $t,t' \in I $). Note that inside $ X \ast Y $, there are two cones $CX$ and $CY$ embedded via the canonical maps $\,CX \into X\ast Y\,,\, (x,t) \mapsto (x,t,\ast) $, and 
$\,CY \into X\ast Y\,,\, (y,t) \mapsto (\ast, 1-t, y) $. Collapsing these cones converts $ X \ast Y $ into the suspension of the smash product of spaces: $ \Sigma(X \wedge Y) = (X \ast Y)/(CX \vee CY ) $. Since $CX$ and $CY$ are both contractible in $X \ast Y $, the quotient map $ X \ast Y \to \Sigma(X \wedge Y)$ is a homotopy equivalence. Thus, in the homotopy category $ {\tt Ho}(\Top_*) $ of pointed spaces, we have natural isomorphisms
\begin{equation}
\la{joineq}
X \ast Y\,\cong\,\Sigma(X \wedge Y) \,\cong\,(\Sigma X) \wedge Y \,\cong\,X \wedge (\Sigma Y)
\end{equation}
These are useful in practice for computing the homotopy types of joins.

Using standard notation, we will write the points of $ X \ast Y $ as formal linear combinations
$ t_0 x + t_1 y $, where  $ x \in X$, $ y \in Y $ and $ (t_0, t_1) \in \Delta^1 :=
\{(t_0, t_1) \in {\mathbb R}^2:\, t_0 + t_1 = 1,\,t_0,t_1 \ge 0\} $. The identification with topological presentation is given by $ (x,t,y) = t x + (1-t) y $. 
The advantage of this notation is that it naturally extends to `higher dimensions': the {\it iterated joins} of spaces
\begin{equation}
\la{itjoin}
X_0 \ast X_1 \ast \ldots \ast X_n = \{t_0 x_0 + t_1 x_1 + \ldots + t_n x_n\,:\,
(t_0,\ldots, t_n) \in \Delta^n,\, x_i \in X_i\}/\!\sim    
\end{equation}
where the equivalence relation is defined by 
$\,\sum_{i=0}^n t_i x_i \sim \sum_{i=0}^{n} t_i' x_i' \,$ if and only if $ t_i = t_i' $ (for all $i$) and $ x_i = x_i' \,$ whenever $\, t_i = t_i' > 0 $. Note that, under this equivalence relation, if $ t_i = 0 $ for some $i$, the point $ x_i $ in $\, t_0 x_0 + \ldots + 0 x_i + \ldots + t_n x_n \in X_0 \ast \ldots \ast X_n $ can be chosen arbitrarily (or simply omitted). 

There is also a convenient way to represent joins by homotopy colimits. For example, it is well-known that the join of two spaces is represented by the homotopy pushout
\begin{equation} 
\la{joinhopush}  
X \ast Y = \mathrm{hocolim}[X \leftarrow X \times Y \rightarrow Y] 
\end{equation}
where the maps are canonical projections and the ``hocolim" is taken either in the category of pointed or unpointed spaces depending on whether we consider reduced or unreduced joins. Formula \eqref{joinhopush} generalizes to iterated joins
\begin{equation} 
\la{itjoinhocolim} 
X_0 \ast X_1 \ast \ldots \ast X_n = \mathrm{hocolim}_{\mathscr{P}(\Delta^n)}(F_X)
\end{equation}
where $\mathscr{P}(\Delta^n)$ is the poset of all non-empty faces of the $n$-simplex $\Delta^n$ (ordered by reversed inclusions) and the diagram $F_X:\mathscr{P}(\Delta^n) \to \Top$ is defined by assigning to a face $\Delta_I \in \mathscr{P}(\Delta^n)$ the product of spaces $\prod_{i \in I} X_i$ (with indices corresponding to the vertices of $\Delta_I$) and to an inclusion of faces $\Delta_J \subset \Delta_I$ the canonical projection $\prod_{i \in I} X_i \to \prod_{j \in J} X_j$. It is easy to see that formula \eqref{itjoinhocolim} boils down to \eqref{joinhopush} in case of two spaces (see, e.g., \cite[Example 5.8.8]{MV15}).

We recall one well-known topological fact that we will use in this paper repeatedly:
\begin{lemma}[Milnor]
\la{Millemma}
If each space $X_i$ in the iterated join \eqref{itjoinhocolim} is $(c_i-1)$-connected, then
the space $ \,X_0 \ast X_1 \ast \ldots \ast X_n \,$ is $( \sum c_i + n -1)$-connected.
\end{lemma}
The proof of Lemma~\ref{Millemma} can be found in \cite[Proposition 3.7.23]{MV15}.

\subsection{The fibre-cofibre construction}
\la{S4.1}
This construction starts with a homotopy fibration sequence with a well-pointed base
\begin{equation}
\la{F0}
F\, \xrightarrow{j} \,X \xrightarrow{p} \,B
\end{equation}
and produces another homotopy fibration sequence on the same base:
\begin{equation}
\la{F1}
F_1\, \xrightarrow{j_1} \,X_1 \xrightarrow{p_1} \,B
\end{equation}
The space $X_1$ in \eqref{F1} is defined to be the homotopy cofibre of the 
fibre inclusion in \eqref{F0}: $\, X_1 := {\rm hocof}_*(j) \,$. 
The map $ p_1 $ --- called the (first) {\it whisker map} --- is obtained by extending  
$ p: X \to B $ to $ X_1 = X \cup C_*(F) $ so that $ C_*(F) $ maps to the basepoint of $B$. 
The  $F_1 $ is then defined to be the homotopy fibre of $p_1\,$: $\, F_1 := {\rm hofib}_*(p_1) $.

The above construction can be iterated {\it ad infinitum}; as a result, one gets a tower of fibration sequences over $B$:
\begin{equation}\la{Gatow}
\begin{diagram}[small]
F               & \rTo  &  F_1      &  \rTo  & F_2 & \rTo& \ldots \\
\dTo^{j}        &       & \dTo^{j_1} &        & \dTo^{j_2} \\
X               & \rTo^{\pi_0}  &  X_1      &  \rTo^{\pi_1}  & X_2 & \rTo^{\pi_2}& \ldots \\
\dTo^{p}        &       & \dTo^{p_1} &        & \dTo^{p_2} \\
B               & \rEq  &  B     &  \rEq  & B& \rEq & \ldots \\
\end{diagram}
\end{equation}
where $X_m$ and $F_m$ are defined by 
\begin{equation}
\la{XmFm}
X_m := {\rm hocof}_*(j_{m-1})\quad ,\quad F_m :=  {\rm hofib}_*(p_m)\ ,\quad \forall\,m \ge 1\,.
\end{equation}
Note that the horizontal arrows $ p_m $ in \eqref{Gatow} are whisker maps making
each row $\, F_m \xrightarrow{j_m} X_m \xrightarrow{p_m} B\,$ of the above diagram a homotopy fibration sequence. On the other hand, the vertical arrows $ \pi_m $ are canonical maps making each triple $\, F_m \xrightarrow{j_m} X_m \xrightarrow{\pi_m} X_{m+1} \,$ a homotopy cofibration sequence. The main observation of \cite{Ga65} is that the homotopy fibres in \eqref{Gatow} can be described explicitly in terms of iterated joins of the based loop spaces $ \Omega B $. More precisely, we have
\begin{theorem}[Ganea]
\la{GThm}
$(1)$ For all $ m \ge 1 $, there are natural homotopy equivalences 
$$ 
F_m\, \simeq\, F \ast \,\Omega B \,\ast \,\ldots\, \ast\, \Omega B \quad (m\mbox{-fold join})
$$
compatible with the fibre inclusions $ F_{m} \to F_{m+1} $ in \eqref{Gatow}.

$(2)$ The whisker maps $ p_m: X_m \to B $ induce a weak homotopy equivalence
$$
\mathrm{hocolim}\,
\{X \xrightarrow{\pi_0} X_1 \xrightarrow{\pi_1} X_2 \to \ldots \to X_m \to \ldots\}
\stackrel{\sim}{\to}\, B
$$
where the homotopy colimit is taken over the telescope diagram in the middle of \eqref{Gatow}.
\end{theorem}
Note that the second claim of Theorem~\ref{GThm} follows from the first by Lemma~\ref{Millemma}. 
\begin{example}[LS-categories] 
\la{exga}
Recall that the {\it LS-category} of a topological space $ B $
is defined to be  $ {\rm cat}(B) := n-1 $, where $n$ is the least cardinality of an open cover $\{U_1,\ldots, U_n\}$ of $B$ such that each $U_i$ is contractible as a subspace in $B$.
Given a pointed connected space $B$, one applies the fibre-cofibre construction to the canonical path fibration $\, \Omega B \to P_*B  \xrightarrow{p} B \,$. The result is the sequence of spaces  
$$ 
P_*B  \xrightarrow{\pi_0} (P_*B)_1 \xrightarrow{\pi_1} (P_*B)_2 \xrightarrow{\pi_2} (P_*B)_3 \to \ldots 
$$
called the {\it Ganea tower}  of the space $B$. The main theorem of \cite{Ga67} asserts that if $B$ is a normal space, its LS category $\, {\rm cat}(B) \le m \,$ if and only if the $m$-th whisker map $\,p_m: (P_*B)_m \to  B $ associated to the above tower splits 
(i.e., admits a section). Most applications of Ganea's construction in topology are related to or inspired by this observation (see, e.g., \cite{CLOT03}).
\end{example}

\begin{example}[Milnor bundles] 
 \la{exmil}
If $G$ is a topological group, we can apply the Ganea construction to the universal principal $G$-fibration $\,G \to EG \to BG \,$. Since $\, \Omega BG \simeq G \,$,
the diagram \eqref{Gatow} in this case is equivalent to
\begin{equation}\la{Mildiag}
\begin{diagram}[small]
G               & \rTo  &  E_1G      &  \rTo  & E_2G & \rTo& \ldots \\
\dTo        &       & \dTo &        & \dTo \\
EG               & \rTo  &  B_1 G      &  \rTo  & B_2G & \rTo& \ldots \\
\dTo       &       & \dTo &        & \dTo \\
BG               & \rEq  &  BG     &  \rEq  & BG& \rEq & \ldots \\
\end{diagram}
\end{equation}
\\
where $ E_n G $ is the (unreduced) iterated join of $(n+1)$ copies of the group $G$:
\begin{equation}
\la{EnG}
E_nG := G \ast G \ast \ldots \ast G \qquad \text{($n+1\,$ times)} \ .  
\end{equation}
Each space \eqref{EnG} carries a natural right $G$-action $E_nG \times G \to E_nG$ that 
makes each column of the diagram \eqref{Mildiag} a Borel fibration sequence. 
In the notation \eqref{itjoin}, we can write this right action explicitly as
\begin{equation} 
\la{rtaction} 
(t_0g_0 +t_1g_1+ \ldots+ t_ng_n) \cdot g = t_0g_0g+t_1g_1g + \ldots +t_ng_ng  \,.
\end{equation}
The $G$-action \eqref{rtaction} is free for all $n \ge 0$: hence $\,B_n G \simeq E_n G/G\,$, and therefore each column of \eqref{Mildiag} arises from the principal $G$-bundle
\begin{equation} 
\la{pbdl} 
G \rightarrow E_nG \rightarrow  B_n G 
\end{equation}
called the $n$-th {\it Milnor bundle} of $G$ (see  \cite{Mi56}). The canonical maps $E_nG \to E_{n+1}G$ in \eqref{Mildiag} are represented explicitly by the $G$-equivariant inclusions:
$$
t_0g_0 + \ldots +t_ng_n\  \mapsto\  t_0g_0 + \ldots +t_ng_n +0 \cdot e 
$$
making $\{E_nG\}_{n \geqslant 0}$ and $\{B_nG\}_{n \geqslant 0}$  into direct systems of 
spaces with
\begin{equation*} 
\la{defbg} EG\, \simeq \, \varinjlim E_nG \quad \text{and} \quad BG \,\simeq\, \varinjlim B_nG\,. 
\end{equation*}
Thus, the spaces $EG$ and $BG$ come equipped with canonical filtrations 
\begin{eqnarray*}
\la{efilt} E_0G \hookrightarrow \ldots \hookrightarrow E_nG \hookrightarrow E_{n+1}G \hookrightarrow \ldots \hookrightarrow EG\\*[1ex]
\la{bfilt} 
B_0G \hookrightarrow \ldots \hookrightarrow B_nG \hookrightarrow B_{n+1}G \hookrightarrow \ldots \hookrightarrow BG
\end{eqnarray*}
with consecutive terms being the Milnor $G$-bundles \eqref{pbdl}.
Now, by Lemma~\ref{Millemma}, the $n$-th Milnor bundle is $n$-universal in the sense that its total space is $(n-1)$-connected. In this way, the Milnor bundles provide $n$-universal
approximations the universal principal $G$-bundle $G \to EG \to BG $.
\end{example}
We remark that Example~\ref{exmil} can be viewed as a special case of Example~\ref{exga} if we take $ B = BG $.

\subsection{Derived schemes of quasi-invariants}
\la{S4.2}
The fibre-cofibre construction is essentially formal: it can be performed in an arbitrary (pointed) model category or $\infty$-category. To see why this construction is relevant
to our problem we will apply it first in a simple algebraic model category: the category $\dAff_{k, \ast}$ of pointed derived affine schemes over a field $k$ of characteristic $0$. As a model for $\dAff_{k, \ast}$, we take the category $(\dgca_{k}\! \downarrow \! k)^{\rm op} $ dual to the category of non-negatively graded commutative DG $k$-algebras $A$  equipped with augmentation map $A \to k $. Extending the standard algebro-geometric notation, we write $\, \Spec(A) \,$ for the object (an affine DG scheme) in $ \dAff_{k}  $ corresponding to the DG algebra $ A $ in $ \dgca_k $.
Since we assume that $ {\rm char}(k) = 0 $, the category $ \dgca_{k} $ carries a natural 
(projective) model structure, where  weak equivalences are the quasi-isomorphisms of DG algebras and
fibrations are the DG algebra maps which are surjective in positive (homological) degrees (see, e.g., \cite[Appendix B]{BKR13}). The category $ \dAff_{k} = \dgca_{k}^{\rm op} $ is equipped with the dual (injective) model structure. The homotopy (co)fibres  of morphisms in $ \dAff_{k} $ are defined in terms of homotopy (co)limits, using formulas \eqref{limcolim}.
Explicitly, given a morphism of pointed affine DG schemes $ f: \Spec(A) \to \Spec(B) $ corresponding to a DG algebra 
homomorphism $ f^*: B \to A $, its homotopy fibre and homotopy cofibre are given by
\begin{equation}
\la{hocof}
{\rm hofib}_*(f) \cong \Spec\left(A \otimes^{\L}_{B} k\right)\ ,\quad
{\rm hocof}_*(f) \cong \Spec\left(B \times^{\bs{R}}_{A} k\right)
\end{equation}
where $\, \otimes^{\L}_B \,$ and $\,\times^{\bs{R}}_{A}\,$ denote the derived tensor product
(homotopy pushout) and the derived direct product (homotopy pullback) in the model category $\dgca_k $.

\vspace{1ex}

We apply the fibre-cofibre construction in the category $ \dAff_{k, \ast} $ to the canonical (algebro-geometric) quotient map $\,p: V \onto V/\!/W \,$ in the situation 
of the following simple example. 
\begin{example}
\la{E2}
Let $ W = \Z/2\Z $, acting in its one-dimensional reflection representation $ V $. Choosing a basis vector in $V$, we can identify 
$ V = \c $ and $ k[V] = k[x] $, with $W$ acting on $k[x]$ by the rule $ s(p)(x) = p(-x)$.
In this case, $ \A = \{0\} $ and $m$ is a non-negative integer. Condition~\eqref{Qinv} says that 
$ p(x) $ is a quasi-invariant of multiplicity $m$ iff $ p(x) - p(-x) $ is divisible 
by $ x^{2m} $. Hence $ Q_m(W) $ is spanned by the monomials $ \{x^{2i}\,:\, i \ge 0\} $ and $ \{x^{2i+1}\,:\, i \ge m\} $, or equivalently
$$ 
Q_m(W) = k[x^2] \oplus x^{2m+1} k[x^2] = k[x^2, x^{2m+1}]\ .
$$
\end{example}
Thus, we take $ V $ to be the affine line acted upon by $ W = \Z/2\Z $ via the reflection at $ 0 $. Regarding $\, V \cong \Spec \,k[x]\,$ and $\, V/\!/W \cong \Spec\, k[x^2] \,$ as affine (DG) schemes pointed at $0$, we can compute the homotopy fibre $ F :=  {\rm hofib}_*(p) $ in $\dAff_{k, \ast}$, using formula \eqref{hocof}:
\begin{equation} \la{F}
F \cong \Spec\left(k[x] \otimes^{\L}_{k[x^2]} k\right) \cong \Spec \left(k[x] \otimes_{k[x^2]} k\right) \cong  \Spec(k[x]/x^2)\ .
\end{equation}
Note that the second isomorphism in \eqref{F} is due to the fact that $ k[x] $ is a free module (and hence, a flat algebra) over $ k[x^2] $. Thus, in $ {\rm dAff}_{k,\ast} $, we have the fibration sequence 
\begin{equation}
\la{AW}
F \,\xrightarrow{j}\, V \,\xrightarrow{p}\, V/\!/W
\end{equation}
where $F$ is given by \eqref{F}. The following simple observation, which was the starting point of the present paper, provides
 a motivation for our topological results in the next section.
\begin{prop}
\la{P1}
The fibre-cofibre construction in $ \,{\rm dAff}_{k,\ast} \,$  applied to the fibration \eqref{AW} produces the tower \eqref{Qsch}  of varieties of quasi-invariants for the reflection representation of $W = \Z/2\Z\,$:
\begin{equation}\la{Gatow2}
\begin{diagram}[small]
F               & \rTo  &  F_1      &  \rTo  & F_2 & \rTo& \ldots \\
\dTo^{j}        &       & \dTo^{j_1} &        & \dTo^{j_2} \\
X               & \rTo^{\pi_0}  &  V_1      &  \rTo^{\pi_1}  & V_2 & \rTo^{\pi_2}& \ldots \\
\dTo^{p}        &       & \dTo^{p_1} &        & \dTo^{p_2} \\
V/\!/W              & \rEq  &  V/\!/W     &  \rEq  & V/\!/W & \rEq & \ldots \\
\end{diagram}
\end{equation}
Thus, for all $\, m\ge 0 $, we have
\begin{equation}
\la{QFm}
V_m \cong  \Spec(Q_m)\ , \quad  F_m  \cong  \Spec\,[Q_m/(x^2)]\ , 
\end{equation}
where $\,Q_m = k[x^2, x^{2m+1}] $ and the maps $\pi_m,\, p_m $ and $j_m $ in \eqref{Gatow2} 
correspond to the natural inclusions $\, Q_{m+1} \into Q_m $, $\, k[x^2] \into Q_m $, and
the projection $ Q_m \onto  Q_m/(x^2) $, respectively.
\end{prop}
\begin{proof} The proof is an easy induction in $m$. For $m=0$, we have already shown  in \eqref{F} that 
$ F = F_0 $, with \eqref{AW} corresponding (i.e. dual) to the natural algebra maps $ k[x^2] \into k[x] \onto k[x]/(x^2)\,$. Now, assuming that $V_m$ is given by \eqref{QFm} together with $p_m: V_m \onto V/\!/W $
corresponding to the inclusion $\, k[x^2] \into Q_m $, we compute the fibre $F_m$
in the same way as in \eqref{F}, using formula \eqref{hocof}:
$$
F_m \,:= \,{\rm hofib}_*(p_m)\, 
\cong \,\Spec\left(Q_m \otimes^{\L}_{k[x^2]} k\right) \,
\cong \,\Spec\left(Q_m \otimes_{k[x^2]} k\right)\,
\cong\,  \Spec\,[Q_m/(x^2)]
$$
Again, crucial here is the fact that $ Q_m $ is a free module (and hence, a flat algebra)
over $k[x]^W$, which is a general property of quasi-invariants (see Theorem~\ref{TGoren}).
Next,  we have
\begin{equation}\la{Xm1}
V_{m+1} :=  {\rm hocof}_*(j_m)  \cong  \Spec \left(Q_m \times^{\bs{R}}_{Q_m/(x^2)} k\right) 
\cong \Spec \left(Q_m \times_{Q_m/(x^2)} k\right) \cong \Spec(Q_{m+1})
\end{equation}
The first isomorphism in \eqref{Xm1} is the result of formula \eqref{hocof} for homotopy cofibres in $ \dAff_{k,\ast} $. The second isomorphism is due to the fact that the canonical map $ Q_m \onto Q_m/(x^2) $ is surjective, and hence a fibration in the standard model structure on $ \dgca_k $ (this implies that
$ {\rm hocof}_*(j_m) $ coincides with the usual cofibre of $ j_m $ in the category of affine $k$-schemes). Finally, the last isomorphism in \eqref{Xm1} is given by the composition of canonical algebra maps
\begin{equation}\la{Qmin}
Q_m \times_{Q_m/(x^2)} k\,\into \,Q_m \times k \onto Q_m 
\end{equation}
It is easy to see that the map \eqref{Qmin} is injective, and its image is precisely
$ Q_{m+1} = k[x^2, x^{2m+3}] $. This gives an identification $\, Q_m \times_{Q_m/(x^2)} k \cong Q_{m+1}\,$ together with the inclusion $ Q_{m+1} \into Q_m $ that defines the morphism of 
schemes  $\pi_m: V_{m} \to V_{m+1} $.
\end{proof}
\begin{remark}
\la{Rem14}
Proposition~\ref{P1} does not extend directly to higher rank groups: the standard fibre-cofibre construction in $\dAff_{k,\ast} $ does {\it not} produce the tower of varieties of quasi-invariants, \eqref{Qsch}, for an arbitrary Coxeter group $W$ ({\it cf.} Proposition~\ref{PP1} below). 
\end{remark}

\subsection{Spaces of quasi-invariants of $SU(2)$}
\la{S4.3}
Let $G$ be a compact connected Lie group with a fixed maximal torus $ T $ and Weyl group  $ W = W_G(T) $. Associated to $ (G,T) $ there is a natural fibration sequence\footnote{If we choose a model for the universal $G$-bundle $EG$ (for example, the Milnor model described in Example~\ref{exmil}) and let $ BG = EG/G $ and $ BT = EG/T $, then \eqref{BTG} is represented by a canonical locally trivial fibre bundle $\,G/T \to EG/T \onto EG/G \,$ (see, e.g., \cite{Huse75}, Chap 4, Sect. 7).} 
\begin{equation}
\la{BTG}
G/T \xrightarrow{j} BT \xrightarrow{p} BG \ ,
\end{equation}
where $p$ is the map induced by the inclusion $ T \into G $ and
$\, j \,$ is the classifying map for the principal $T$-bundle $ G \to G/T $. 
\begin{prop}
\la{PP1}
Assume that $W$ is simply-laced (i.e., of ADE type). Then the fibre-cofibre construction applied to \eqref{BTG} produces a tower of fibrations
\begin{equation}
\la{Gatow3}
\begin{diagram}[small]
G/T               & \rTo  &  F_1(G,T)      &  \rTo  & F_2(G,T) & \rTo&\ \ldots\  & \rTo & F_m(G,T )& \rTo & \ldots  \\
\dTo^{j}        &       & \dTo^{j_1} &        & \dTo^{j_2} & & \ \ldots\ & & \dTo^{j_m} &  \\
BT              & \rTo^{\pi_0}  &  X_1(G,T)      &  \rTo^{\pi_1}  & X_2(G,T) & \rTo^{\pi_2}& \ldots  & \rTo & X_m(G,T )& \rTo^{\pi_m} & \ldots \\
\dTo^{p}        &       & \dTo^{p_1} &        & \dTo^{p_2} & & \ \ldots\ & & \dTo^{p_m} &  \\
BG              & \rEq  &  BG     &  \rEq  & BG & \rEq & \ \ldots\  & \rEq & BG & \rEq  & \ldots \\
\end{diagram}
\end{equation}
where the diagram of $($total$)$ spaces in the middle
%
%
together with maps $ p_m: X_m(G,T) \to BG $ satisfy the first three properties {\rm (QI$_{1}$), (QI$_{2}$)} and {\rm (QI$_{3}$)} of Section~\ref{realizationprob}.
\end{prop}
\begin{proof} 
If $W$ is simply-laced, all reflection hyperplanes of $W$ are in the same orbit, and the poset $\M(W)$ consists
only of constant multiplicities which we identify with $ \Z_+ $. By Ganea's Theorem~\ref{GThm}, the homotopy fibre $ F_m = F_m(G,T)$  at stage $m$ in \eqref{Gatow3} can be represented by 
the iterated join
\begin{equation}
\la{eq1}
F_m \, = \,G/T \ast \Omega BG \,\ast\, \stackrel{m}{\ldots} \,\ast\, \Omega BG
\,\simeq\, G/T\, \ast\,  G \,\ast\, \stackrel{m}{\ldots} \,\ast\, G \,=\, G/T\, \ast\,  E_{m-1}G\ ,
\end{equation}
where $E_{m-1}G $ is Milnor's model for the  $(m-1)$-universal $G$-bundle (see Example~\ref{exmil}).
The fibre \eqref{eq1} carries a natural left (holonomy) action $\,\Omega BG \times F_m \to F_m \,$ that under the identification \eqref{eq1}, corresponds to the diagonal action of $G\,$:
\begin{equation}
\la{eq2}
G \times F_m \to F_m\ ,\quad  g\cdot(t_0(g_0 T) + t_1 g_1 + \ldots + t_m g_m) =  t_0(g g_0 T) + t_1 gg_1 + \ldots + t_m gg_m
\end{equation}
where $g,\,g_0,\,g_1,\,\ldots,\, g_m \in G $ and $(t_0,\ldots, t_m) \in \Delta^m $, see \eqref{itjoin}.
The space $ X_m = X_m(G,T) $ can then be represented as the homotopy quotient
\begin{equation}
\la{eq3}
X_m = (F_m)_{hG} =  EG \times_G (G/T\, \ast\,  E_{m-1}G)
\end{equation}
and the fibration $ F_m \to X_m \to BG $ in \eqref{Gatow3} is identified with the Borel fibration
\begin{equation}
\la{eq4}
F_m \to (F_m)_{hG} \to BG
\end{equation}
Now, the Weyl group $ W = N_G(T)/T $ acts on 
$ G/T $ by $\,w\cdot(gT) = g n^{-1} T\,$, where $ w = nT \in W $. With identification \eqref{eq1}, this action naturally induces a $W$-action on $ F_m = G/T \ast E_{m-1}G $. The latter commutes with the $G$-action \eqref{eq2}, and hence extends to the space $X_m$ of homotopy $G$-orbits in $F_m $. Explicitly, with identification \eqref{eq3}, the action of $W$ on
$ X_m = EG \times _G (G/T \ast E_{m-1} G) $ is given by
\begin{equation}
\la{eq5}
w\cdot(x,\,t_0(g_0 T) + t_1 g_1 + \ldots + t_m g_m) =
(x,\,t_0(g_0 n^{-1} T) + t_1 g_1 + \ldots + t_m g_m)
\end{equation}
where $x \in EG$ and $w = n T \in W\,$. The inclusions
$ F_m \into F_{m+1} $ defined by 
$$
t_0(g_0 T) + t_1 g_1 + \ldots + t_m g_m\ \mapsto\
t_0(g_0 T) + t_1 g_1 + \ldots + t_m g_m + 0\, e
$$
are obviously $ (G \times W)$-equivariant, hence induce $W$-equivariant maps on homotopy $G$-quotients:
$\, \pi_m: X_m \to X_{m+1} \,$.  The whisker maps $\,p_m: X_m \to BG \,$ are induced by the trivial maps $ F_m \to {\rm pt} $ and hence are $W$-invariant. Thus, we have established property $({\rm P1})$ for the tower \eqref{Gatow3}. Property $({\rm P2})$
follows directly from part $(2)$ of Theorem~\ref{GThm}.  

For $({\rm P3})$, it suffices to show that
\begin{equation}
\la{Bork}
H^*_W(F_m, \Q) \cong \Q
\end{equation}
Indeed, since the actions of $G$ and $W$ on $F_m$ commute, we have  
$$ 
(X_m)_{hW} = EW \times_W (EG \times_G F_m) \simeq EG \times_G (EW \times_W F_m) = EG \times_G (F)_{hW}
$$ 
Whence
\begin{equation}
\la{Bork11}
H_W^*(X_m,\,\Q) \cong H^*_G((F_m)_{hW}, \Q)\ 
\end{equation}
On the other hand, if \eqref{Bork} holds, the Serre spectral sequence of the Borel fibration 
$$
(F_m)_{hW} \to EG \times_G (F_m)_{hW} \to BG
$$
degenerates, giving an isomorphism $\, H^*_G((F_m)_{hW}, \Q) \cong H^*(BG, \Q)\,$. Combining this last isomorphism with \eqref{Bork11} yields $\, H_W^*(X_m,\,\Q) \cong H^*(BG, \Q) \,$, as required by (QI$_{3}$).

Now, since $F_m $ is connected, \eqref{Bork} is equivalent to vanishing of higher cohomology:
\begin{equation}
\la{Bork1}
H^n_W(F_m, \,\Q) = 0 \quad \forall\, n > 0\ .
\end{equation}
We prove \eqref{Bork1} by induction on $ m $. For $ m = 0 $, we have $\, F_0 = G/T $ and
$\, (G/T)_{hW} \simeq (G/T)/W \cong G/N\,$, since the action of $W$ on $G/T$ is free. 
It follows that $ H^n_{W}(F_0, \Q) \cong H^n(G/N,\,\Q) = 0 $ for all $ n > 0 $ as
it is well known that the space $G/N$ is rationally acyclic for any compact connected  Lie group (see \cite[Theorem 20.3]{Bo67}).

Now, assume that \eqref{Bork1} holds for some $ m \ge 0 $ and consider 
$ (F_{m+1})_{hW} = (F_m \,\ast\,G)_{hW} $. Representing this space by homotopy colimits (see \eqref{joinhopush})
and using the fact that the homotopy colimits commute, we have
\begin{eqnarray*}\la{eq7}
 (F_{m+1})_{hW}  & \simeq &  \hocolim_W\,\hocolim\,[\,F_m \leftarrow F_m \times G \to G\,] \\*[1ex]
& \simeq &  \hocolim\,[\,(F_m)_{hW} \leftarrow (F_m \times G)_{hW} \to (G)_{hW}\,] \\*[1ex]
& \simeq &  \hocolim\,[\,(F_m)_{hW} \leftarrow (F_m)_{hW} \times G \to BW \times G\,]
 \end{eqnarray*}
This homotopy decomposition implies that the cohomology groups of  $ (F_{m+1})_{hW} $ and $ (F_{m})_{hW} $ are related by the following Mayer-Vietoris type long exact sequence:
$$
 H^{n-1}[(F_m)_{hW} \times G] \,\to\, 
H^{n}[(F_{m+1})_{hW}] \,\to\, H^{n}[(F_m)_{hW}] \oplus H^n[BW \times G] \,\to\, 
H^{n}[(F_m)_{hW} \times G] 
$$
Since $W$ is a finite, its rational cohomology  vanishes in positive degrees. Hence, by K\"unneth Theorem, we have $ H^\ast(BW \times G, \Q) \cong H^\ast(G, \Q) $. Furthermore, our induction assumption \eqref{Bork1} implies that $ H^\ast((F_m)_{hW} \times G, \Q) \cong H^*(G,\Q) $ and for each $ n \ge 1 $,
the last map in the above exact sequence is an isomorphism. Thus, for $ n\ge 2 $,
the above sequence breaks up into short exact sequences
$$
0 \to H^{n}((F_{m+1})_{hW}, \Q) \,\to\, H^n(G, \Q) \xrightarrow{\sim} H^n(G, \Q) \to 0
$$
which show that $ H_W^{n}(F_{m+1}, \Q) = 0 $ for all $ n \ge 2 $. On the other hand, in dimension $0$ and $1$, the above long exact sequence reads
$$
 H^{0}((F_m)_{hW},\Q) \oplus H^0(G,\Q) \,\onto\, H^0(G,\Q) \to  H^{1}((F_{m+1})_{hW}, \Q) \,\to 
H^1(G, \Q) \xrightarrow{\sim} H^1(G, \Q)
$$
where the first arrow is surjective and the last is an isomorphism. This shows that 
$ H_W^1 (F_{m+1}, \Q) $ also vanishes, thus finishing the induction and the proof of (QI$_{3}$).
\end{proof}
\begin{example}
\la{counterex}
Let us describe the cohomology $ H^*(X_1, \Q) $ of the first space $ X_1 = X_1(G,T) $ in the diagram \eqref{Gatow3} explicitly. By general properties of the Ganea construction (see Section~\ref{S4.1}), this space fits in the
homotopy cofibration sequence 
\begin{equation}
\la{cofseq1}
G/T \xrightarrow{j} BT \xrightarrow{\pi_0} X_1
\end{equation}
Since both $BT$ and $G/T$ have no cohomology classes in odd dimensions and the natural map $j^*: 
H^*(BT, \Q) \to H^*(G/T, \Q) $ is surjective, the long cohomology sequence associated to
\eqref{cofseq1} reduces to the short exact sequence
\begin{equation}
\la{seq2}
0 \to \tilde{H}^*(X_1, \Q) \xrightarrow{\pi^*_0}  \tilde{H}^*(BT, \Q) \xrightarrow{j^*} \tilde{H}^*(G/T, \Q) \to 0
\end{equation}
where $ \tilde{H}^* $ stands for the reduced cohomology. Since $X_1$ is connected, \eqref{seq2}
shows that the algebra map $ \pi^*_0: H^*(X_1, \Q) \to H^*(BT,\Q) $ is injective, and with 
identification $ H^*(BT, \Q) \cong \Q[V] $ (as in \eqref{borel3}), its image being
\begin{equation}
\la{seq3}
H^*(X_1, \Q) \cong \Q + \langle \Q[V]^W_+ \rangle\,\subset\,\Q[V]\ ,
\end{equation}
where $\langle \Q[V]^W_+ \rangle$ is the ideal in $\Q[V]$ generated by the $W$-invariant polynomials
of positive degrees. Formula \eqref{seq3} shows that $ X_1 $ has no odd cohomology; moreover, the map
$ p_1^*: H^*(BG, \Q) \to  H^*(X_1, \Q)  $ induced by the first whisker map in \eqref{Gatow3}
is injective, and $ H^*(X_1, \Q) $ is a finite module over $ H^*(BG, \Q) \cong \Q[V]^W $
via $ p_1^* $. By Hilbert-Noether Theorem, this implies that $ H^*(X_1, \Q) $ is a finitely generated
graded $\Q$-algebra, however it is {\it not} Cohen-Macaulay (and hence {\it not} Gorenstein) when $ \dim_{\Q}(V) \ge 2 $. To see this we set $ R :=  H^*(X_1, \Q) $, $\, S := H^*(BT, \Q) $ and $ S^W = H^*(BG, \Q) $ to simplify the notation. Since $S$ is a free $S^W$-module, the long exact sequence obtained by dualizing 
the short exact sequence $\,0 \to R \to S \to S/R \to 0\,$ over $S^W$ yields
$$
{\rm Ext}_{S^W}^i(R, S^W)\,\cong\, {\rm Ext}^{i+1}_{S^W}(S/R, S^W)\ ,\quad \forall \, i\ge 1
$$
Since $ S/R \cong \tilde{H}^*(G/T, \Q) $ by \eqref{seq2}, $\,\dim_{\Q}(S/R) = |W| - 1 < \infty \,$.
Hence ${\rm Ext}^n_{S^W}(S/R, S^W) \not= 0 \,$ and therefore $\, {\rm Ext}_{S^W}^{n-1}(R, S^W)  \not= 0 \,$, 
where $ n := \dim_{\Q}(V) $. It follows that when $ n > 1 \,$, $R$ is not free as a graded module over $S^W$, and hence  not Cohen-Macaulay as a graded algebra (see, e.g., \cite[Prop. 6.8]{Sm72}).
\end{example}

Example~\ref{counterex} shows that, unfortunately, the tower of spaces \eqref{Gatow3} constructed in Proposition~\ref{PP1} cannot satisfy all five axioms of our realization problem for an arbitrary compact Lie group. Indeed, if $\, {\rm rk}(G) = n \ge 2 $, then (QI$_{5}$) already fails for $ H^*(X_1(G,T), \Q) $, since $ H^*(X_1(G,T), \Q) $ is not a Gorenstein algebra, while $ Q_1(W) $ is (see Theorem~\ref{TGoren}). Note, however, that in the rank one case, for $ G = SU(2) $, we still have $\, H^*(X_1(G,T), \Q) \cong Q_1(\Z/2\Z) $ by formula
\eqref{seq3}. The next theorem shows that this is not a coincidence.
\begin{theorem}\la{MTh1} 
Assume that $ G = SU(2) $ and $W = \Z/2\Z$. Then the diagram of spaces \eqref{Gatow3} together with whisker maps $ p_m $ 
produced by the fibre-cofibre construction satisfies all five properties {\rm (QI$_{1}$)--(QI$_{5}$)} of Section~\ref{realizationprob}. 
In particular, for all $m \ge 0$, there are isomorphisms of graded commutative algebras
\begin{equation}
\la{cohalg1}
H^*(X_m(G, T),\, \Q)\,\cong\, Q_m(W)\ ,
\end{equation}
where $ Q_m(W)$ is the subring of $W$-quasi-invariants of multiplicity $m$ in $\Q[V]$.
Moreover, $\,X_m(G, T)\,$ are unique, up to rational homotopy equivalence, topological spaces realizing the algebras $ Q_m(W) $.

\end{theorem}
\begin{proof} Properties (QI$_{1}$)-(QI$_{3}$) have already been established for arbitrary $G$
in Proposition~\ref{PP1}; we need only to check (QI$_{4}$) and (QI$_{5}$).
As a topological space, $ SU(2) $ is homeomorphic to $ \bS^3 $ and $ G/T = \c {\mathbb P}^1 \cong \bS^2 $. Hence, applying a well-known formula for the join of spheres, we can identify the fibre \eqref{eq1}:
\begin{equation}
\la{eq6}
F_m = G/T\, \ast\, G^{\ast \,m} \,\cong \, 
\bS^2 \ast\, (\bS^3)^{\ast m} \,\cong\, \bS^{4m+2}\ .
\end{equation}
Thus, for $G = SU(2) $, \eqref{eq4} is equivalent to the sphere fibration: $\,\bS^{4m+2} \to X_m \to B\bS^3 $. We will look at the Serre spectral sequence of this fibration and apply the Leray-Hirsch Theorem. Since both the basespace and the fibre of \eqref{eq4} have no cohomology in odd dimensions, the Serre spectral sequence collapses, giving an isomorphism of graded vector spaces (see, e.g., \cite[Lemma~III.4.5(1)]{MT78})
$$
H^*(X_m, \Q) \cong H^{*}(BG, \Q) \otimes H^{*}(F_m, \Q) 
$$
Then, the Leray-Hirsch Theorem (see, e.g., \cite[Theorem~III.4.2]{MT78}) implies that $ H^*(X_m, \Q) $ is a free graded module over the algebra $ H^{*}(BG, \Q) = H^*(BSU(2),\,\Q) $, which is the rational polynomial algebra $ \Q[c_2] $  generated by the second Chern class $ c_2 \in H^4(BSU(2), \Q)\,$. 
This graded module has rank two, with  $ H^{*}(BG, \Q) $ identified with a direct summand in $ H^*(X_m, \Q) $ under the whisker map $ p_m^*: H^{*}(BG, \Q) \into H^*(X_m, \Q) $. The 
complement of $ H^{*}(BG, \Q) $ in  $ H^*(X_m, \Q) $ is generated by a cohomology class $ \xi $ of dimension $ 4m+2 $ whose image under the projection $ j_m^*: H^*(X_m, \Q) \to H^*(F_m, \Q) \cong H^*(\bS^{4m+2}, \Q) $ is the fundamental cohomology class of $ \bS^{4m+2}$. Thus, we have
\begin{equation}
\la{HXm}
H^*(X_m, \Q) \cong \Q[c_2] \oplus \Q[c_2] \xi
\end{equation}
where $ |c_2| = 4 $ and $ |\xi|= 4m+2 \,$. Next, we look at the homotopy cofibration sequence in \eqref{Gatow3}
\begin{equation}
\la{cfibm}
F_m \xrightarrow{j_m} X_m \xrightarrow{\pi_m} X_{m+1}
\end{equation}
arising from the Ganea construction. This gives a long exact sequence on (reduced) cohomology:
\begin{equation}
\la{longred}
\ldots\, \to\, \tilde{H}^{n-1}(F_m, \Q)\, \to \, \tilde{H}^{n}(X_{m+1}, \Q)\, \xrightarrow{\pi_m^*}\,  \tilde{H}^{n}(X_{m}, \Q)\, \xrightarrow{j_m^*}\, \tilde{H}^{n}(F_{m}, \Q) \,\to\, \ldots 
\end{equation}
Since neither $ F_m $ nor $ X_m $ (by \eqref{HXm}) have odd cohomology, we see
immediately from \eqref{longred} that all algebra maps $ \pi_m^* $ must be injective,
i.e. property (QI$_{4}$) holds for \eqref{Gatow3}. For each $m \ge 0 $, the composition of these maps
then gives an embedding
\begin{equation}
\la{compi}
\pi^*_0 \,\pi^*_1 \, \ldots\,\pi_{m-1}^*:\   H^*(X_m, \Q) \into H^*(X_{m-1}, \Q) \into \ldots \into H^*(BT, \Q)
\end{equation}
If we identify $ H^*(BT, \Q) = \Q[x] $ by choosing  $ x \in H^2(BT, \Q) = H^2(B\bS^1, \Q)$ to be the universal Euler class, which is the image of the canonical generator of $ H^2(B\bS^1, \Z) = H^2(K(\Z, 2), \Z)$, then the Chern class $ c_2 \in H^4(BG, \Q) $ maps by \eqref{compi} to $x^2 \in H^*(BT, \Q) $. Then, for degree reasons, the generator $ \xi \in H^{4m+2}(X_m, \Q) $
in \eqref{HXm} should map to (a scalar multiple of)  $ x^{2m+1} \in \Q[x] $.
Thus the algebra homomorphism \eqref{compi} identifies $ H^*(X_m, \Q) \cong \Q[x^2, x^{2m+1}] $, which 
is precisely the subring $Q_m$ of $W$-quasi-invariants in $ H^*(BT, \Q) = \Q[x] $. 
This gives property (QI$_{5}$) and completes the proof of the first part of the theorem.

The last claim of the theorem follows from Sullivan's formality theorem \cite{Su77}. Indeed, the algebras $ Q_m(W) $ have the presentation $ \Q[\xi, \eta]/(\xi^2 -  \eta^{2m+1})$, where $|\eta|= 4$ and $|\xi| = 4m+2$ (see Example~\ref{E2}). Hence, by \cite[Remark (v), p. 317]{Su77}, they are {\it intrinsically} formal. This means that, for each $ m\ge 0 $, there is only one rational homotopy type that realizes $Q_m\,$.
\end{proof}

From now on, we will assume that $ G = SU(2) $ and $ T = U(1) $ embedded in $SU(2)$ in the standard way as a maximal torus.
\begin{defi}
\la{qspaces}
We call the $G$-space $\, F_m(G,T) := G/T \,\ast\, E_{m-1} G \,$  the {\it $m$-quasi-flag manifold} and the associated homotopy quotient 
$$ 
X_m(G,T) := F_m(G,T)_{hG} = EG \times_G(G/T \,\ast \,E_{m-1}G) 
$$ 
the {\it space of $m$-quasi-invariants}  for $ G = SU(2) $. These spaces fit in the Borel fibration sequence 
\begin{equation}
\la{borfib}
F_m(G,T) \xrightarrow{j_m} X_m(G,T) \xrightarrow{p_m} BG 
\end{equation}
that generalizes the fundamental sequence \eqref{BTG}.
\end{defi}

\begin{remark}
\la{rem11}
By definition, $\, H^*(X_m(G,T),\, \Q) = H^*_{G}(F_m(G,T),\, \Q) \,$ for all $ m \ge 0 $.
With this identification, the algebra homomorphisms $\,H^*(X_m, \Q) \to H^*(BT, \Q)\,$ constructed in Theorem~\ref{MTh1} (see \eqref{compi}) are induced (on $G$-equivariant cohomology) by the natural inclusion maps 
\begin{equation}
\la{incl0}
i_0:\ G/T \into F_m(G,T)\ ,\quad 
gT \mapsto 1 \cdot (gT) + 0 \cdot x \ , 
\end{equation}
where $ x \in E_{m-1}G\,$. Note that the maps \eqref{incl0} are null-homotopic in the category $\Top$ of ordinary spaces, 
the null homotopy being $\, i_t:\, gT \mapsto (1-t) \cdot (gT) + t \cdot x \,$; however, they are {\it not} null-homotopic 
in the category of $G$-spaces and $G$-equivariant maps. In fact, the proof of Theorem~\ref{MTh1} shows that the maps
induced by \eqref{incl0} on $G$-equivariant cohomology are injective and hence nontrivial.
\end{remark}
%

\subsection{$T$-equivariant cohomology}
\la{S4.31}
Our next goal is to compute the $T$-equivariant cohomology of the $G$-spaces $F_m(G,T)$
by restricting the $G$-action to the maximal torus $ T \subset G $.
The computation is based on the following simple observations.
\begin{lemma}
\la{Teq}
For all $m \ge 0 $, there is a natural $T$-equivariant homeomorphism
\begin{equation}
\la{TFm}
F_m(G, T)  \,\cong\, \Sigma\,E_{2m}(T)\ ,
\end{equation}
where $\Sigma$ stands for the unreduced suspension in $ \Top $.
\end{lemma}
\begin{proof}
First, note that $ G $ is $T$-equivariantly homeomorphic to the 
(unreduced) join of two copies of $T$: the required homeomorphism
\begin{equation}
\la{GTT}
T \ast T \,\cong \, G
\end{equation}
can be explicitly written as $\,t \lambda + (1-t) \mu \,\mapsto\, t^{1/2}\, \lambda + (1-t)^{1/2}\, \mu j\,$, where $G = SU(2) $ is identified with the group of unit quaternions in $ {\mathbb H} = \C \oplus \C j $ and $ T = U(1) $ with unit complex numbers. Similarly, we can define a $T$-equivariant homeomorphism
\begin{equation}
\la{GTT2}
(G/T)^T \ast T \,\cong\,  G/T
\end{equation}
where $ (G/T)^T $ denotes the set of $T$-fixed points in $G/T$.
Combining \eqref{GTT} and \eqref{GTT2} with natural associativity isomorphisms for joins, we get
\begin{equation}
\la{FmT}
F_m(G,T) = (G/T) \ast G^{\,\ast\,m} \,\cong\,  (G/T)^T \ast T^{\,\ast\,(2m+1)} = \bS^0 \ast E_{2m}(T)
\end{equation}
which is equivalent to formula \eqref{TFm}.
\end{proof}
\begin{lemma}
\la{TLem}
For all $ n \ge 0 $, there are natural algebra isomorphisms 
\begin{equation}
\la{Tcoh}
H_{T}^*(\Sigma E_n(T)\,,\,\Q)\,\cong\, \Q[x]\,\times_{\Q[x]/(x^{n+1})}\,\Q[x]\,.    
\end{equation}
\end{lemma}
\begin{proof}
We compute
\begin{equation}
\la{hoBT}
[\Sigma E_n(T)]_{hT} \, \simeq\,  [\hocolim(\,{\rm pt} \leftarrow E_{n}(T) \to 
{\rm pt})]_{hT}  \,\simeq \, \hocolim(BT \leftarrow B_{n}(T) \to BT)\,,
\end{equation}
%
%
where the last equivalence follows from the fact that $ E_n(T)$ is an $n$-universal $T$-bundle, 
so that the $T$-action on $E_{n}(T) $ is free and hence $ E_{n}(T)_{hT} \simeq E_{n}(T)/T = B_{n}(T)$ (see Example~\ref{exmil}). To complete the proof it remains to note that 
$\, BT \simeq \C\bP^\infty $ and $\, B_{n}(T) \cong \C\bP^{n} $ for $T = U(1) $, with  natural map $ B_{n}T \to BT $ represented by the inclusion $ \C\bP^{n} \into 
\C\bP^{\infty}$ (see, e.g., \cite[Example 9.2.3]{Se97}). Hence, \eqref{hoBT} shows that
$\,[\Sigma\, E_n(T)]_{hT} \simeq \C\bP^\infty \bigvee_{\C\bP^n} \C\bP^{\infty} $, which, by Mayer-Vietoris
sequence, yields the isomorphism \eqref{Tcoh}.
\end{proof}
As a consequence of Lemma~\ref{Teq} and Lemma~\ref{TLem}, we get
\begin{prop}
\la{PTcoh}
For all multiplicities $ m \ge 0 $, there are natural algebra isomorphisms 
\begin{equation}
\la{GKM}
H_{T}^*(F_m(G,T),\,\Q)\,\cong\, \Q[x]\,\times_{\Q[x]/(x^{2m+1})}\,\Q[x]\,,
\end{equation}
where $ x \in H^2(BT, \Q) $ is the universal $($rational$)$ Euler class.
\end{prop}
\begin{remark}
\la{ReGKM}
For $m=0$, formula~\eqref{GKM} is well known: it follows, for example, from a general combinatorial description of $T$-equivariant cohomology of equivariantly formal spaces in terms of moment graphs 
(see \cite{GKM98}). In our subsequent paper, we will generalize the main localization 
theorem of \cite{GKM98} to moment graphs with multiplicities, and as an application,
extend the result of Proposition~\ref{PTcoh} to quasi-flag manifolds for an arbitrary compact connected Lie group. 
\end{remark}

Next, we recall the modules of $\C W$-valued quasi-invariants, $ \mathbf{Q}_{k}(W) $,  introduced in \cite{BC11}. 
In \cite[Section~3.2]{BC11}, these modules are considered only for integral multiplicities $ k \in \Z_{+}$; however, their definition makes sense --- in the Coxeter case --- 
for all $ k \in \frac{1}{2}\, \Z_+ \,$ ({\it cf.} \cite[(3.8)]{BC11}). We provide a natural topological interpretation of these modules.
\begin{cor}
\la{RelBC}
For all $ n \ge 0 $, there are natural isomorphisms of $\Q[x] \rtimes W$-modules
\begin{equation}
\la{TbfQ}
H_{T}^*(\Sigma E_n(T)\,,\,\C)\,\cong\,  \mathbf{Q}_{\frac{n+1}{2}}(W)\,.    
\end{equation}
In particular, $\,H_{T}^*(F_m(G,T),\,\C)\,\cong\,  \mathbf{Q}_{m+\frac{1}{2}}(W)\,$ for all  $\, m \ge 0 $.
\end{cor}
\begin{proof} Under the isomorphism \eqref{Tcoh}, the geometric action of $W = \Z/2\Z $ on $ H_{T}^*(\Sigma E_n(T)\,,\,\Q) $ corresponds to the action $ (p,\,q) \mapsto (s(q),\,s(p)) $ on the fiber product. Relative to this action, we can then define the $W$-equivariant map 
$$
f:\,\Q[x]\,\times_{\Q[x]/(x^{n+1})} \Q[x]\,\to \,\Q[x] \otimes \Q W \ ,\quad
(p,\,q) \mapsto \frac{1}{2}(p+qs)
$$
This map is obviously injective, and it is easy to see that its image is
$\, \Q[x] e_0 + \Q[x]\, x^{n+1} e_1 \,$, where $e_0 = (1+s)/2 $ and $e_1 = (1-s)/2$ are the idempotents in $\Q W $ corresponding to the trivial and sign representations of $W$. 
Example~3.9 of \cite{BC11} shows that $ \im(f) $ is precisely $ \mathbf{Q}_{\frac{n+1}{2}}(W)$; thus, combining $f$ with the isomorphism of Lemma~\ref{TLem} 
gives the required isomorphism \eqref{TbfQ}. The last statement 
then follows from Proposition~\ref{PTcoh}.
\end{proof}
\begin{remark}\la{TGW}
Recall that, for any compact connected Lie group $G$, there is a natural isomorphism
\begin{equation}
\la{TGrel}
H^*_G(X, \Q) \,\cong\, H^*_T(X, \Q)^W 
\end{equation}
that extends the result of Borel's Theorem~\ref{tborel} to an arbitrary $G$-space $X$
(see, e.g., \cite[Chap. III, Prop. 1]{Hs75}). For $ X = F_m(G,T) $, it follows
from Corollary~\ref{RelBC} that 
$$
H_{T}^*(F_m(G,T),\,\C)^W \,\cong\, e_0 \mathbf{Q}_{m+\frac{1}{2}}(W) \,\cong\,
Q_m(W) \,.
$$
Thus the isomorphism \eqref{cohalg1} of Theorem~\ref{MTh1} can be deduced from \eqref{GKM} by \eqref{TGrel}.
\end{remark}

\begin{remark} \la{fmvsArnld}
By \eqref{FmT}, there is a $T$-equivariant homeomorphism
$\,F_m(G,T) \cong \bS^0 \ast T^{\ast (2m+1)}$, where the $T$-action on the right is the join of the trivial action on $\bS^0 $ with the left multiplication-action on each of the $T$-factors. Identifying $\bS^0 = \{ \pm 1\} \subset \mathbb{R} \subset \C $ and $T =  \{|z|=1\} \subset \c$, we get the explicit homeomorphism $
\bS^0 \ast T^{\ast (2m+1)} \xrightarrow{\sim} \bS^{4m+2} \subset \mathbb{R} \times \c^{2m+1}\, $: 
\begin{equation} 
\la{fmbract1} 
t_0x_0 + t_1 z_1 +\ldots +t_{2m+1}z_{2m+1} \mapsto  (\sqrt{t_0}x_0, \sqrt{t_1}z_1,\ldots, \sqrt{t_{2m+1}}z_{2m+1}) \,,
\end{equation}
where $t_0,\ldots,t_n \geqslant 0$ and $\sum_i t_i=1$. Under \eqref{fmbract1}, the $T$-action on $\bS^{4m+2} $ agrees with (the restriction
of) the $T$-action on $ \mathbb{R} \times \c^{2m+1} $ defined by
\begin{equation} 
\la{fmbract2} 
\lambda \cdot (x,z_1,\ldots, z_{2m+1}) \,=\,(x,\lambda z_1,\ldots, \lambda z_{2m+1})\,.\end{equation}
This $T$-action is obviously smooth: the induced representation on the tangent spaces ${\rm T}_p\bS^{4m+2} \cong \c^{2m+1}$ at each of its two fixed points is:
$\,\lambda \cdot (z_1,\ldots, z_{2m+1}) \,=\,(\lambda z_1,\ldots, \lambda z_{2m+1})\,$.

The above considerations show that there is {\it no} $T$-equivariant diffeomorphism between our spaces $F_m(G,T)$ and the $T$-spheres $\bS^{4m+2}$ constructed by M. Feigin and K. Feldman in Appendix~\ref{ArMapCon}. (Recall that the Feigin-Feldman $T$-action on $\bS^{4m+2}$  arises from the so-called Arnold-Maxwell Theorem (see Theorem~\ref{first}), 
and the weights of its induced representations on the tangent spaces at fixed points are computed in Proposition~\ref{third}. These weights ($2,4,\ldots, 2(2m+1)$) differ from those of \eqref{fmbract2}.) On the other hand, we can define a (smooth) map 
$\varphi:\, \mathbb{R} \times \c^{2m+1} \to  \c^{2m+1} \times \mathbb{R}$ by 
\begin{equation} 
\la{fmbract4}
(x,z_1,\ldots,z_{2m+1}) \,\mapsto\, (z_{2m+1}^{(4m+2)},z_{2m}^{(4m)},\ldots, z_1^{(2)},x)\,,  \end{equation}
where we write $ z^{(2k)} := |z|(z/|z|)^{2k} $ for $ 1 \leqslant k \leqslant 2m+1$. 
If we identify $ F_m(G,T) $ with a submanifold of $ \mathbb{R} \times \c^{2m+1} $ (using \eqref{fmbract1}) and $ \bS^{4m+2} \subset \c^{2m+1} \times \mathbb{R} $ as in Appendix~\ref{ArMapCon}, then $\varphi$ restricts to a map $ F_m(G,T) \to \bS^{4m+2}  $, which is easily seen to be $T$-equivariant for the 
Feigin-Feldman action.  Up to a permutation of coordinates, $\varphi$ is the join of the identity map on $\bS^0$ with the finite covering maps $p_{2k}: T \to T^{(2k)}$,$\, z \mapsto z^{2k}$, where $T^{(r)}$ stands for the torus equipped with the $T$-action $\lambda \cdot z = \lambda^r z$. Since the $p_{2k}$ induce isomorphisms on rational $T$-equivariant cohomology, we see that the map $ \varphi: F_m(G,T) \to \bS^{4m+2}  $ --- while not being a global diffeomorphism --- induces an isomorphism on rational $T$-equivariant (and hence, by \eqref{TGW}, $G$-equivariant) cohomology.     
\end{remark}

\subsection{Divided difference operators}
\la{S4.4}
As an application of Theorem~\ref{MTh1}, we give a topological construction of generalized
divided difference operators associated with quasi-invariants. Recall that the classical divided difference operators $\,\Delta_{\alpha}: \Q[V] \to \Q[V] \,$ are attached to reflections
$ s_\alpha \in W $ of a Coxeter group $W$  by the rule ({\it cf.} \cite{D73, D74}):
\begin{equation}
\la{demop}
(1 - s_{\alpha})p \,= \,\Delta_\alpha(p) \cdot \alpha_H
\end{equation}
where $ \alpha_H \subset V^* $ is a linear form vanishing on the reflection hyperplane $ H = H_\alpha $.
Note that \eqref{demop} defines $ \Delta_\alpha $ uniquely up to a nonzero constant factor.
The definition of quasi-invariants of Coxeter groups suggests the following natural generalization of \eqref{demop}:
\begin{equation}
\la{demopgen}
(1 - s_{\alpha})p \,= \,\Delta^{(m_{\alpha})}_\alpha(p) \cdot \alpha^{2 m_\alpha + 1}_H
\end{equation}
To be precise, given a $W$-invariant multiplicity function $ m: \A \to \Z_+ $, $\, \alpha \mapsto m_{\alpha} $, formulas \eqref{demopgen} define unique (up to nonzero constants) linear maps 
\begin{equation}
\la{mdelta}
\Delta^{(m_{\alpha})}_\alpha: Q_m(W) \to Q_0(W) 
\end{equation}
one for each reflection $ s_{\alpha} \in W $. Note that $ Q_0(W) = \Q[V] $, and for $ m = 0 $, the 
maps \eqref{mdelta} coincide  with the classical divided difference operators: $\,\Delta^{(0)}_{\alpha} = \Delta_{\alpha} $. 
\begin{defi}
We call \eqref{mdelta} the {\it divided difference operators of $W$ of multiplicity $m$}. 
\end{defi}

When $W$ has rank one, i.e. $W$ is generated by a single reflection $s $, the corresponding map $ \Delta_s^{(m)} $ takes values in $ \Q[V]^W $ thus defining a linear operator on $W$-quasi-invariants:
\begin{equation}
\la{mdelta1}
\Delta^{(m)}_s:\, Q_m(W) \to Q_m(W)\ . 
\end{equation}
The operator \eqref{mdelta1} has a natural topological interpretation in terms of our spaces of quasi-invariants. The proof of Theorem~\ref{MTh1} shows that the basic fibration
\eqref{borfib} is equivalent to a sphere fibration with fibre $ F_m \simeq \bS^{4m+2} $. Hence,
associated to \eqref{borfib} there is a Gysin long exact sequence of the form 
(see, e.g., \cite[Example~II.5.C]{McL01}):
\begin{equation}
\la{Gysin}
\ldots \,\to H^n(BG, \Q) \xrightarrow{p_m^*} H^n(X_m, \Q) \xrightarrow{(p_m)_*} H^{n-4m-2}(BG, \Q) \to H^{n+1}(BG, \Q) \to\, \ldots 
\end{equation}
where $ p_m^*$ is the natural pullback map induced on cohomology by the $m$-th whisker map
$p_m: X_m \to BG $ and $ (p_m)_* $ is a `wrong way' pushforward map called the Gysin homomorphism.
Combining these last two maps, we get the graded linear endomorphism on $ H^*(X_m, \Q)$ of  degree $ - (4m +2) \,$:
\begin{equation}
\la{pmpm}
p_m^* \circ (p_m)_*:\ H^*(X_m, \Q)\,\to \, H^*(X_m,\,\Q)
\end{equation}
The next proposition generalizes a well-known formula for the classical divided difference operators $ \Delta_{\alpha} $ (proven, for example, in \cite{BE90}).
\begin{prop}
\la{P22}
Under the isomorphism of Theorem~\ref{MTh1}, the operator \eqref{pmpm} coincides with the
divided difference operator \eqref{mdelta1} of multiplicity $m$: i.e.,
\begin{equation}
\la{dpmpm}
\Delta_s^{(m)} = p_m^* \circ (p_m)_*
\end{equation}
\end{prop}
\begin{proof}
Since the algebra homomorphism $ p_m^*: H^*(BG, \Q) \to H^*(X_m, \Q) $ is injective (for all $m$),
the Gysin sequence \eqref{Gysin} breaks up into short exact sequences
\begin{equation}
\la{Gyssh}
0 \to H^*(BG, \Q) \xrightarrow{p_m^*} H^*(X_m, \Q) \xrightarrow{(p_m)_*} H^{*-4m-2}(BG, \Q) \to 0
\end{equation}
Now, if we identify $\,H^*(BG, \Q) = \Q[c_2]\,$ and $ H^*(X_m, \Q) = \Q[x^2, x^{2m+1}] $
as in (the proof of) Theorem~\ref{MTh1}, the map $ p_m^*$ takes $ c_2 $ to $ x^2 $ and hence
$ c_2^k $ to $ x^{2k} $ for all $ k \ge 0 $. By exactness of \eqref{Gyssh}, we then conclude that
$\,(p_m)_*(x^{2k}) = 0 \,$, while $ (p_m)_*(x^{2m+1}) = \kappa_m $, where $ \kappa_m \in \Q^{\times}$
is a nonzero constant. Hence, $\, p_m^* (p_m)_*(x^{2k}) = 0 \,$ for all $\, k \ge 0 $; on the other hand, by projection formula,
\begin{eqnarray*}
p_m^* (p_m)_*(x^{2k+2m+1}) &=& p_m^* (p_m)_*(x^{2k} \cdot x^{2m+1}) \\
&=& p_m^* (p_m)_*(p_m^*(c_2^k) \cdot x^{2m+1}) \\
&=& p_m^*(c_2^k) \cdot (p_m)_*(x^{2m+1})\\
&=& \kappa_m \,x^{2k}
\end{eqnarray*}
Thus, up to a nonzero constant factor, we have
$$
p_m^* (p_m)_*(x^N) =
\left\{
\begin{array}{ll}
0\ ,& \mbox{if}\ \ N = 2k  \\*[1ex]
x^{2k}\ , & \mbox{if}\ \ N = 2k+2m+1
\end{array}
\right.
$$
which agrees with the action of $\, \Delta_s^{(m)} = \frac{1}{x^{2m+1}}(1-s)\,$ on $\, Q_m(W) = \Q[x^2, x^{2m+1}]\,$.
\end{proof}
%

\section{`Fake' spaces of quasi-invariants}
\la{S5}
By Theorem~\ref{MTh1}, the spaces $ X_m(G, T) $ provide topological realizations for the algebras
$Q_m(W) $ that are unique up to {\it rational} equivalence. This raises the question of whether the $ X_m(G, T)$'s are actually unique up to homotopy equivalence. In this section, we answer it in the negative by constructing a natural class of counterexamples related to finite loop spaces. These remarkable loop spaces --- sometimes referred to as {\it fake Lie groups} --- were discovered by D. L. Rector \cite{Rec71}  as examples of nonstandard (`exotic') deloopings of $ \bS^3 $. We will show that the rational cohomology rings of the `fake' spaces of quasi-invariants associated to the Rector spaces are isomorphic to the `genuine' ones $X_m(G,T)$; however, the spaces themselves are {\it not}\, homotopy equivalent (in fact, as we will see in Section~\ref{S6}, they can be distinguished from each other $K$-theoretically). 
Thus, we get many different topological realizations of $ Q_m(W) $, but among these only the `genuine' spaces of quasi-invariants $X_m(G,T)$ satisfy all properties (QI$_{1}$)-(QI$_{5}$).

\subsection{Finite loop spaces}
The finite loop spaces are natural homotopy-theoretic generalizations of compact Lie groups. Their basic properties as well as many examples can be found in R. Kane's monograph \cite{K88}; for later developments, we refer to the survey papers \cite{Not95}, \cite{D98}, and \cite{Gr10}. Here we only recall the main definition.
\begin{defi}
\la{FLS}
A {\it finite loop space} is a pointed connected space $B$ such that $ \Omega B $ is homotopy equivalent to a finite CW-complex. 
\end{defi}

It is convenient to represent a finite loop space as a triple $(X, B, e)$, where $X$ is a finite CW-complex, $ B $ is a pointed connected space, and $ e: X \stackrel{\sim}{\to} \Omega B $ is a homotopy equivalence. A prototypical example is $(G, BG, e) $, where $G$ is a compact Lie group, $BG$ its classifying space, and $ e: G \stackrel{\sim}{\to} \Omega BG $ is a canonical  equivalence. In general, finite loop spaces have many properties in common with compact Lie groups; however, the class of such spaces is much larger. In fact, if $G$ is a compact connected non-abelian Lie group, there exist uncountably many homotopically distinct spaces $B$ such that $ \Omega B \simeq G $; thus the underlying topological space of $G$ carries uncountably many finite loop structures  (see \cite{Mol92}). In the case $G = SU(2)$, this striking phenomenon was originally discovered by Rector \cite{Rec71} (see Theorem~\ref{RTh} below).

\subsection{Fake Lie groups of type $SU(2)$}
We will work with localizations of topological spaces in the sense of D. Sullivan (for basic properties of this classical construction we refer the reader to \cite[Part 2]{MP12}.) Given a space $X$ and a prime  number $p$, we denote the localization of $X$ at $p$ by $X_{(p)}$. Following \cite[8.5.1]{MP12}, we say that two (nilpotent, finite type) spaces $X$ and $Y$ are {\it in the same genus} if $X_{(p)} \simeq Y_{(p)} $ for every prime $p$. We will be interested in finite loop spaces that are in the same genus as $BG$ for some compact connected Lie group $G$. Such spaces (sometimes called the fake Lie groups of type $G$) have been studied extensively in the literature (see, e.g., \cite{NS90}) since their original discovery in \cite{Rec71}.  This last paper provided a complete classification of spaces in the genus of $ BSU(2)$, including a homotopy-theoretic characterization of the $ BSU(2) $ itself among these spaces. The main results of \cite{Rec71} can be encapsulated in the following
\begin{theorem}[Rector] 
\la{RTh}
Let $G = SU(2)$, and  let $B$ be a space in the genus of $ BG $. 
Then, for each prime $p$, there is a homotopy invariant $\,(B/p) \in \{\pm 1\}\,$ called the {\rm Rector invariant} of $B$ at $p$, such that 

 $(1)$ The set $\{(B/p)\}$, where $p$ runs over all primes, is a complete set of  invariants
  of $B$ in the genus of $ BG$.
 
 $(2)$ Every combination of values of $ (B/p) $ can occur for some $B$. In particular, the genus of $BG$ consists of uncountably many distinct homotopy types.
 
 $(3)$ The Rector invariant of $ B = BG $ equals  $1$ at all primes $p$.
 
 $(4)$ The space $B$ admits a maximal torus\footnote{We say that a finite loop space $B$ admits a maximal torus if there is a map $ p: BT_n \to B $ from the classifying space of a finite-dimensional torus with homotopy fibre being a finite CW-complex (see \cite{Rec71b}).} if and only if $B$ is homotopy equivalent to $BG$.
\end{theorem}
\begin{remark}
Each space $B$ in the genus of $ BSU(2)$ defines a loop structure on $\bS^3 $, i.e. $ \Omega B \simeq \bS^3 $. Conversely, a uniqueness theorem of Dwyer, Miller and Wilkerson \cite{DMW87} implies that every loop structure on $ \bS^3 $ belongs to the genus of $ BSU(2) $. Thus, Theorem~\ref{RTh} combined with results of \cite{DMW87} provides a complete classification of finite loop spaces of type $SU(2)$.
\end{remark}
\begin{remark}
It was a long-standing conjecture in homotopy theory (motivated in part by Theorem~\ref{RTh}$(4)$, {\it cf.} \cite{Wil74}) that a finite loop space with a maximal torus is homotopy equivalent to the classifying space of a compact Lie group. This conjecture was eventually proved in \cite{AG09}, using the Classification Theorem of $p$-compact groups. Thus, the existence of maximal tori provides a purely homotopy-theoretic characterization of compact Lie groups among finite loop spaces.
\end{remark}
Even though the spaces $B \not\simeq BG$ do not admit maximal tori, this does not rule out the possibility that there could exist interesting maps $ f: BT \to B $ whose homotopy fibres are {\it not} finite CW complexes. In his thesis (see \cite{DY01}), D. Yau refined Rector's classification by describing the spaces $B$ in the genus of $BSU(2)$ that can occur as targets of essential (i.e., non-nullhomotopic) maps from $BT$. Such spaces admit a beautiful arithmetic characterization: 
\begin{theorem}[Yau]
\la{esmap}
Let $G = SU(2)$, and  let $B$ be a space in the genus of $ BG $. Then

$(1)$ $B$ admits an essential map $ f: BT \to B $ if and only if there is an integer $k \neq 0$ such that $ (B/p)=(k/p) $ for all but finitely many primes $p$, where $(k/p)$ denotes the Legendre symbol\footnote{Recall that, for a prime $p$, the {\it Legendre symbol} $(k/p)$ of an integer $k$ is defined whenever $\,p \nmid k\,$: for $p$ odd, we have $ (k/p) = 1 $ (resp., $-1$) if $k$ is a quadratic residue (resp., nonresidue) mod $p$, while for $p=2$,  $\,(k/2) = 1 $ (resp., $-1$) if $k$ is  quadratic residue (resp. nonresidue) mod $8$.} of $k$.

$(2)$ If $B$ satisfies condition $(1)$, then there exists a unique $($up to homotopy$)$ map $p_B: BT \to B$ such that 
every essential map $\,f: BT \to B \,$ is homotopic to $ g \circ p_B $ for some self-map $g$ of $B$. 

$(3)$ For $ B = BG $, the map $ p_{BG}: BT \to BG $ is induced by the maximal torus inclusion.
\end{theorem}

\subsection{`Fake' spaces of quasi-invariants}
\la{S5.3}
Let $B$ be a space in the genus of $BG$ (for $G = SU(2)$) that admits an essential map from $BT$. Theorem~\ref{esmap} shows that, for such a space, there is a natural generalization of the maximal torus: namely, the  `maximal' essential map $ p_B: BT \to B $. We let $F(\Omega B, T)$ denote the homotopy fibre of this map and apply the Ganea construction to the associated fibration sequence:
\begin{equation}
\la{RectorTow}
\begin{diagram}[small]
F(\Omega B, T)               & \rTo  &  F_1(\Omega B, T)      &  \rTo  & F_2(\Omega B, T) & \rTo& \ldots \\
\dTo^{j_B}        &       & \dTo^{j_{1,B}} &        & \dTo^{j_{2,B}} \\
BT               & \rTo^{\pi_0\quad }  &  X_1(\Omega B, T)      &  \rTo^{\pi_1}  & X_2(\Omega B, T) & \rTo^{\pi_2}& \ldots \\
\dTo^{p_B}        &       & \dTo^{p_{1,B}} &        & \dTo^{p_{2,B}} \\
B               & \rEq  &  B     &  \rEq  & B& \rEq & \ldots \\
\end{diagram}
\end{equation}
As a result, we construct a tower of spaces $ X_m(\Omega B, T) $ which we will refer to as the {\it `fake' spaces of quasi-invariants} associated to the Rector space $B$. Note, if $ B = BG $, then $\, \Omega B \simeq G \,$,
and by Theorem~\ref{esmap}$(3)$, the map $\, p_B: BT \to BG \,$ is the maximal torus inclusion; hence, in this case, 
$\, X_m(\Omega B, T) $ are equivalent to the `genuine' spaces $ X_m(G, T) $ of quasi-invariants (see Definition~\ref{qspaces}). 

To compute the cohomology of $ X_m(\Omega B, T) $ we recall ({\it cf.} \cite{Rec71}) that any space $B$ in the genus of $BG$ can be represented as a (generalized) homotopy pullback:
\begin{equation}
\la{hopull}
B = {\rm holim}_{\{p\}}\{ \begin{diagram}[small] BG_{(p)} & \rTo^{r_p} & BG_{(0)} & \rTo^{n_p} & BG_{(0)}\} \end{diagram} \ ,
\end{equation}
where the indexing set $\{p\}$ runs over all primes, $\,r_p $ denotes the natural map from the 
$p$-localization to the rationalization of $BG$, and the map $ n_p$ is induced by multiplication by an integer $ n_p $ which is relatively prime to $p$ and such that $(n_p/p)=(B/p)$ for every $p$ (for $p=2$, one requires, in addition, that $\,n_p \equiv 1 (\mathrm{mod}\, 4)$). 

Now, if a space $B$ admits an essential map from $BT$, part $(1)$ of Theorem \ref{esmap} implies that the set of integers $ \{n_p \in \Z \,:\,p \ \text{prime}\} $ appearing in \eqref{hopull} can be chosen to be finite. Hence, for such spaces, we can define the natural number
\begin{equation} 
\la{fdegree} N_B\,:=\,\mathrm{min}\{ \mathrm{lcm}(n_p) \in \N\,:\, B=\mathrm{holim}_{\{p\}}(n_p \circ r_p) \}\,,
\end{equation}
which is clearly a homotopy invariant of $B$.
Note that $ N_B = 1 $ iff $ B = BG $; however, in general, $N_B $ does not determine
the homotopy type of $B$ (see \cite[(1.8)]{DY01} for a counterexample).
\begin{lemma} 
\la{CohB}
For any space $B$ in the genus of $BG$, $\, H^{\ast}(B, \Z) \cong \Z[u]\,$, where $|u|=4$. If $B$ admits an essential map from $BT$, then, with natural identification $ H^{\ast}(BT,\Z) \cong \Z[x] $ as in Theorem~\ref{MTh1},  
the map $p_B^{\ast}: H^{\ast}(B, \Z) \to H^{\ast}(BT,\Z) $ is given by $\,p_B^{\ast}(u)= N_B\,x^2\,$, where $N_B$ is defined by \eqref{fdegree}.
\end{lemma} 
\begin{proof}
The first claim can be deduced easily from the fact that $ \Omega B \simeq \bS^3 $ by looking at the Serre spectral sequence of the path fibration
$\,\Omega B \to P_*B \to B \,$ ({\it cf.} \cite[\S 4]{Rec71b}). The second claim is
a consequence of the last part of \cite[Theorem 1.7]{DY01}, which shows that \eqref{fdegree} equals (up to sign) the degree of the map $\,p_B^{\ast}\,$ on $K$-theory with coefficients in $\Z$ and hence on cohomology.
\end{proof}
\begin{theorem} 
\la{CohXMR}
Let $B$ be a space in the genus of $BG$ admitting an essential map from $BT$. 

$(i)$ All maps $ \pi_m $ in \eqref{RectorTow} are injective on rational cohomology. For each $m\ge 0$, the composite map 
$ \tilde{\pi}_m = \pi_{m-1} \ldots \pi_1 \pi_0
$ induces an embedding $ H^{\ast}(X_m(\Omega B, T), \Q) \into H^{\ast}(BT, \Q) = \Q[x] $ with image $ Q_m(W) \subseteq \Q[x] $. Thus, $\,
H^{\ast}(X_m(\Omega B, T), \Q) \cong Q_m(W)\,$ for all $m \ge 0 $.

$(ii)$ For each $ m\ge 0 $, there is an algebra isomorphism 
$$ 
H^{\ast}(X_m(\Omega B, T), \Q) \xrightarrow{\sim} H^{\ast}(X_m(G, T), \Q) 
$$
making commutative the diagram
$$
\begin{diagram}[small]
H^{\ast}(B, \Q) & \rInto^{p_{m,B}^{\ast}} & {H}^{\ast}(X_m(\Omega B, T), \Q) & \rInto^{\tilde{\pi}_m^{\ast}} & {H}^{\ast}(BT, \Q)\\
\dTo^{(p_{BG}^*)^{-1} p_B^*} & & \dTo^{\wr} & & \Big\|\\
{H}^{\ast}(BG, \Q) & \rInto^{p_m^{\ast}} & {H}^{\ast}(X_m(G,T), \Q) & \rInto^{\tilde{\pi}_m^{\ast}} & {H}^\ast(BT, \Q)
\end{diagram}
$$
where the map $ (p_{BG}^*)^{-1} p_B^*$ is given explicitly by $\, u \mapsto N_Bx^2 $
$($see Lemma~\ref{CohB}$)$.
\end{theorem}
\begin{proof}
We prove part $(i)$ by induction on $m$. First, note that for $m=0$, $(i)$ as well as $(ii)$ follow from Lemma \ref{CohB}. To perform the induction we define the subalgebras $\,Q'_m \subseteq \Q[x]\,$ for $m > 0$ by 
$$
Q'_0 := Q[x]\,,\qquad Q'_m := \Q+ N_Bx^2 \cdot Q'_{m-1}\,, \ m>0\ . 
$$
Clearly, 
$$Q'_m = \Q + \Q \cdot N_Bx^2 + \ldots +\Q \cdot (N_Bx^2)^{m-1} + (N_Bx^2)^m \Q[x] \ . $$
It follows that $Q'_m=Q_m$ as subrings of $\Q[x]$ for all $m$. Now assume that 
$$
H^{\ast}(X_m(\Omega B, T), \Q) \cong Q'_m\\,,
$$
and that $\tilde{\pi}_m^{\ast}$ is the inclusion $Q'_m \hookrightarrow \Q[x]$. To compute the cohomology of the fibre $F_m(\Omega B, T)$, we use the Eilenberg-Moore spectral sequence for the fibration sequence $F_m(\Omega B, T) \to X_m(\Omega B, T) \to B$, whose $E_2$-term is 
$$ E_2^{\ast, \ast}\,=\, \mathrm{Tor}_{\ast,\ast}^{{H}^{\ast}(B)}({H}^{\ast}(\mathrm{pt}), {H}^{\ast}(X_m(\Omega B, T))) \,\cong\,  \mathrm{Tor}^{\ast,\ast}_{\Q[u]}(\Q, Q'_m)$$ 
By Lemma \ref{CohB}, $\mathrm{Tor}_{\ast,\ast}^{\Q[u]}(\Q, Q'_m)$ is the (co)homology of the complex %
$$\begin{diagram}[small] 0 & \rTo & Q'_m & \rTo^{\cdot N_B x^2} & Q'_m & \rTo & 0\end{diagram} \ .$$
Since $Q'_m \subseteq \Q[x]$ is an integral domain, $\mathrm{Tor}^{\Q[x]}_i(\Q, Q'_m) = 0$ for $i>0$. The Eilenberg-Moore spectral sequence for the fibration sequence $F_m(\Omega B, T) \to X_m(\Omega B, T) \to B$ therefore collapses to give 
$${H}^{\ast}(F_m(\Omega B, T), \Q) \cong\, Q'_m/(N_B x^2) \ . $$ 
Further, since the Eilenberg-Moore spectral sequence is multiplicative, $j_{m,B}^{\ast}$ is the canonical quotient map. In particular, note that the cohomology of $F_m(\Omega B, T)$  is concentrated in even degree. The long exact sequence of cohomologies associated with the cofibration sequence $F_m(\Omega B, T) \to X_m(\Omega B, T) \to X_{m+1}(\Omega B, T)$ yields (for $n$ even)
$$\begin{diagram}[small] \tilde{{H}}^{n}(X_{m+1}[\Omega B, T)] & \rInto^{\pi_m^{\ast}}&  \tilde{{H}}^{n}[X_m(\Omega B, T)] & \rTo^{j_{m,B}^{\ast}} & \tilde{{H}}^{n}[F_m(\Omega B, T)] & \rOnto^{\partial} & \tilde{{H}}^{n+1}[X_{m+1}(\Omega B, T)] \end{diagram} $$
Since $j_{m,B}^{\ast}$ is surjective, we have
$$\tilde{{H}}^{n+1}(X_{m+1}(\Omega B, T), \Q) = 0 \quad \mbox{for}\ n \ \mbox{even}\,. 
$$
Hence, 
$$
{H}^{\ast}(X_{m+1}(\Omega B, T), \Q) \cong \Q + \Ker(j_{m,B}^{\ast}) \,=\, \Q+ (N_B x^2) \cdot Q'_{m} \,=\, Q'_{m+1}\,, $$
with $\pi_m^{\ast}$ being the inclusion $Q'_{m+1} \hookrightarrow Q'_m$. This completes the induction step, proving part $(i)$. 

Part $(ii)$ follows immediately from $(i)$ combined with Lemma \ref{CohB} (since $p_B = p_{m,B} \circ \tilde{\pi}_m$). 
\end{proof}
\begin{cor}
\la{Cor5.8}
For a fixed $m\ge0$, all spaces $ X_m(\Omega B, T) $ are rationally  equivalent to $X_m(G,T)$ $($and hence to each other$)$.
\end{cor}
\begin{proof}
This follows from Theorem~\ref{CohXMR} and the uniqueness part of Theorem~\ref{MTh1}.
\end{proof}
In Section~\ref{S6.3} (see Corollary~\ref{rsdistinct}), 
we will show that $ X_m(\Omega B, T) \not\simeq X_m(\Omega B', T) $ whenever $\, N_B \not= N_{B'} \,$. Thus Theorem~\ref{CohXMR} provides many different topological realizations\footnote{It is tempting to conjecture that the (homotopy types of the) spaces $ X_m(\Omega B, T) $ associated with the Rector spaces $B$ admiting an essential map from $BT$ constitute the set of {\it all} such realizations. Unfortunately, besides Theorem~\ref{MTh1}$(2)$, we do not have much evidence for this conjecture.} for the algebras $Q_m(W)$. 
However, these do not give us different solutions to our realization problem (see Section~\ref{realizationprob}), since none of the spaces $ B $ in the genus of $BG$ (except for $BG$ itself) admits a maximal torus and hence none carries a natural $W$-action. In addition, by Ganea's Theorem~\ref{GThm}, 
$\, \hocolim_{m}[X_m(\Omega B, T)] \simeq B \,$, which shows that property (QI$_{2}$)
  fails for $ X_m(\Omega B, T)$  
when $ B \not\simeq BG\,$.

\section{Equivariant $K$-theory}
\la{S6}
In this section, we compute the $G$-equivariant $K$-theory $ K_G(F_m) $ of the $m$-quasi-flag manifold $F_m = F_m(G,T) $ associated to $ G = SU(2)$. We find that $ K_G(F_m) $  is isomorphic to the ring $ \cQ_m(W) $ of {\it exponential} quasi-invariants of $W$. By the Atiyah-Segal Theorem, the (ordinary) $K$-theory of $X_m(G,T) $ is then isomorphic to the completion $\widehat{\cQ}_m(W)$ of $ \cQ_m(W) $ with respect to the canonical augmentation ideal of $R(G)$. For the `fake' spaces of quasi-invariants, $ X_m(\Omega B, T)$, associated to Rector spaces, the $K$-theory rings $ K[X_m(\Omega B, T)] $ are new invariants that are not isomorphic to $\widehat{\cQ}_m(W) $ in general and are strong enough to distinguish the $ X_m(\Omega B, T)$ up to homotopy equivalence.
\subsection{Equivariant $K$-theory}
\la{S6.1}
Recall that, for a compact Lie group $G$ acting continuously on a compact topological space $X$, the $K_G(X)$ is defined to be the Grothendieck group of $G$-equivariant (complex topological) vector bundles on $X$. As shown in  \cite{Seg68}, this construction extends to a
$\Z/2$-graded multiplicative generalized cohomology theory $ K_G^* $ on the category of (locally compact) $G$-spaces that is called the {\it $G$-equivariant $K$-theory}. We write $ K^*_G(X) := K_G^0(X) \oplus K_G^1(X)$, with understanding that $ K_G^0(X) \cong K_G^{2n}(X) $ and  $ K_G^1(X) \cong K_G^{2n+1}(X) $ for all $ n \in \Z $.
When $G$ is trivial, $K^{\ast}_G(X)$ coincides with the ordinary complex $K$-theory $K^*(X) $, while for $\,X = \mathrm{pt}\,$, $ K^{\ast}_G(\mathrm{pt}) $ is the representation ring $R(G)$ of $G$ (in particular, we have $\,K^{1}_G(\mathrm{pt}) = 0$). In general, 
by functoriality of $K^{\ast}_G $, the trivial map $\,X \to \mathrm{pt}\,$ gives a canonical $R(G)$-module structure on the ring $K^{\ast}_G(X)$ for any $G$-space $X$. The ring $ K^*_G(X) $ has nice properties for which we refer the reader to \cite{Seg68}. Here we only mention two technical results needed for our computations.

The first result is a well-known K\"unneth type formula for equivariant $K$-theory first studied by Hodgkin (see, e.g., \cite[Theorem 2.3]{BrZ00}).
\begin{theorem}[Hodgkin] \la{HSS}
Let $G$ be a compact connected Lie group, such that $\pi_1(G)$ is torsion-free. Then, for any two $G$-spaces
$X$ and $Y$, there is a spectral sequence with $E^2$-term
$$ 
E^2_{\ast,\ast} = \mathrm{Tor}_{\ast,\ast}^{R(G)}(K^{\ast}_G(X), K^{\ast}_G(Y))
$$
that converges to $\, K^{\ast}_G(X \times Y) \,$, where $X \times Y$
is viewed as a $G$-space with the diagonal action.
\end{theorem}

The second result is the following Mayer-Vietoris type formula, which is also --- in one form or another --- well known to experts. 
\begin{lemma} 
\la{KHoPush}
Let $f:U \to X$ and $g:U \to Y$ be proper equivariant maps of $G$-spaces. Let $\,Z = \mathrm{hocolim}(X \stackrel{f}{\leftarrow} U \stackrel{g}{\to} Y)$, where `$\hocolim$' is taken in the category of $G$-spaces.
Then, the abelian groups $ K_G^*(X) $, $K_G^*(Y)$ and $ K_G^*(Z)$ are related by the six-term exact sequence
$$
\begin{diagram}[small]
K^0_G(Z) & \rTo & K^0_G(X) \oplus K^0_G(Y) & \rTo^{f^{\ast}-g^{\ast}} & K^0_G(U)\\
\uTo^{\partial} & &  & & \dTo_{\partial}\\
K^1_G(U) & \lTo^{f^{\ast}-g^{\ast}} & K^1_G(X) \oplus K^1_G(Y) & \lTo^{} &  K^1_G(Z)\\
\end{diagram}
$$
\end{lemma}
\noindent
The proof of Lemma~\ref{KHoPush} can be found, for example, in \cite{JO99}.

\subsection{$K$-theory of quasi-flag manifolds} 
\la{S6.2}
We first introduce rings $ \cQ_m(W) $ of {\it exponential quasi-invariants} of a Weyl group $W$. Let $G$ be a compact connected Lie group with maximal torus $T$ and associated Weyl group $W$. Let $ \hat{T} := \Hom(T, U(1)) $ denote the character lattice and $R(T)$ the representation ring of $T$. It is well known that $ R(T) \cong \Z[\hat{T}] $ via the canonical map induced by taking characters of representations, and $ R(T)^W \cong R(G) $ via the restriction map $ i^*: R(G) \to R(T) $ induced by the inclusion $ i: T \into G $ (see, e.g., \cite[Chap. IX, Sect. 3]{Bou82}). Using the first isomorphism we identify $ R(T) = \Z[\hat{T}] $ and write $ e^{\lambda} $ for the elements of $ R(T) $ corresponding to characters $ \lambda \in \hat{T} $. Next, we let $ \R \subseteq \hat{T} $ denote the root system of $W$ determined by $(G,T)$ and choose a subset $ \R_+ \subset \R $ of positive roots in $ \R $. If $ s_\alpha \in W $ is the reflection in $W$ corresponding to $ \alpha \in \R_+ $, then the difference $ e^{\lambda} - e^{s_{\alpha}(\lambda)} $ in $ R(T)$ is uniquely divisible by $ 1 - e^{\alpha} $ for any $\lambda \in \hat{T}$. Following \cite{D74},  we define a linear endomorphism $\, \Lambda_{\alpha}: R(T) \to R(T)\,$ for each $\alpha \in \R_+ $, such that
\begin{equation}
\la{extdem}
(1 - s_\alpha) f = \Lambda_{\alpha}(f) \cdot (1 - e^{\alpha})\ .
\end{equation}
The operator $ \Lambda_\alpha $ is an exponential analogue of the divided difference operator $ \Delta_\alpha $ introduced
in Section~\ref{S4.4}(see \eqref{demop}). Note that the conditions \eqref{Qinv} defining the usual quasi-invariant polynomials 
can be written in terms of the divided difference operators as $ \Delta_\alpha(p) \equiv 0 \ \mbox{mod}\, (\alpha)^{2 m_\alpha} $.
This motivates the following definition  of quasi-invariants in the exponential case.
\begin{defi}
\la{Expqi}
An element $ f \in R(T)$ is called an {\it exponential quasi-invariant of $W$ of multiplicity} $ m \in \M(W) $ if
\begin{equation}
\la{EQinv}
\Lambda_{\alpha}(f) \,\equiv\, 0 \ \,\mbox{mod}\ (1 - e^{\frac{\alpha}{2}})^{2 m_\alpha}\ ,\quad 
\forall\,  \alpha \in \R_+\ .
\end{equation}
\end{defi}
\begin{remark}
\la{R22}
In general, it may happen that $ \frac{\alpha}{2} \not\in \hat{T} $ for some $ \alpha \in \R_+ $, so that 
$\, e^{\frac{\alpha}{2}} \not\in R(T) $. We view
\eqref{EQinv} as a congruence in the extended group ring $ \Z[\frac{1}{2} \hat{T}] $ that naturally contains $ R(T) $.
\end{remark}
We write  $\cQ_m(W)$ for the set of all  $ f \in R(T) $ satisfying \eqref{Expqi} for a fixed multiplicity $ m $. This 
set is closed under addition and multiplication in $R(T)$, i.e. $\cQ_m(W)$ is a commutative subring of $ R(T) $. 
(The latter can be easily seen from the twisted derivation property of
Demazure operators: $ \Lambda_{\alpha}(f_1 f_2) = \Lambda_\alpha(f_1)\cdot f_2 + s_{\alpha}(f_1) \cdot \Lambda_{\alpha}(f_2) $ that holds for all $  \alpha \in \R $, see \cite[Sect. 5.5]{D74}.) 
\begin{example}
\la{ExSU2}
We describe $\cQ_m(W)$ explicitly in the case of $ G = SU(2) $ and $ T = U(1) $ the diagonal torus. In this case 
$ \hat{T} $ coincides with the weight lattice $P(\R)$  which is generated by the fundamental weight $ \varpi: T \to U(1) $ defined by $ \varpi \left(\begin{array}{cc} t & 0\\ 0 & t^{-1}\end{array}\right) = t $. The corresponding (simple) root is $ \alpha = 2 \varpi $, and the Weyl group $W = \langle s_{\alpha}\rangle \cong \Z/2\Z$ acts on $ \hat{T} $ by $ s_{\alpha}(\varpi) = - \varpi $. We have
\begin{equation}
\la{RTG}
R(T) \cong \Z[z, z^{-1}]\ ,\quad R(G) = R(T)^W \cong \Z[z+z^{-1}]
\end{equation}
where $ z = e^{\varpi} = e^{\frac{\alpha}{2}}\,$. Now, with these identifications, we claim that
\begin{equation}
\la{mulquasi}
\cQ_m(W) = \Z \oplus \Z \cdot (z^{1/2}-z^{-1/2})^2 \oplus \Z \cdot (z^{1/2}-z^{-1/2})^{4} \oplus \ldots \oplus (z^{1/2}-z^{-1/2})^{2m} \cdot \Z[z,z^{-1}]\ .   
\end{equation}
Indeed, if $ f \in \Z[z,z^{-1}] $ can be written in the form \eqref{mulquasi}, then 
$$ 
f - s_\alpha(f) \in  (z^{1/2}-z^{-1/2})^{2m} (1-s_\alpha)\,\Z[z,z^{-1}] = (z^{1/2}-z^{-1/2})^{2m} \,(z-z^{-1})\, \Z[z,z^{-1}]\ ,
$$
which shows that $ \Lambda_{\alpha}(f) = (1-z^2)^{-1}(f - s_\alpha f) $ is divisible by $\,(1-z)^{2m} = (1 - e^{\frac{\alpha}{2}})^{2m} $ in $ \Z[z,z^{-1}] $. Thus $ f \in \cQ_m $. 
To see the converse denote the right-hand side of \eqref{mulquasi} by $ \tilde{\cQ}_m$. Note that there is a natural  $\Q[z+z^{-1}]$-module decomposition
$$
\Q[z,z^{-1}] \cong \Q[z+z^{-1}] \oplus \Q[z+z^{-1}]\cdot \delta\,,
$$
where $\delta := z-z^{-1}$.  Writing $f = p + q \cdot \delta$ with $p, q \in \Q[z+z^{-1}]$, we find that $f-s_\alpha(f) = 2q \delta$. Thus, if  $ f \in \cQ_m $ then $\,f - s_\alpha(f) \in (z^{1/2}-z^{-1/2})^{2m}\,(z-z^{-1})\,\Z[z,z^{-1}]\,$ and hence $q \in (z^{1/2}-z^{-1/2})^{2m}\,\Q[z,z^{-1}]$. It follows that $ f \in \tilde{\cQ}_m \otimes \Q $.  On the other hand, $ (\tilde{\cQ}_m \otimes \Q) \cap \Z[z, z^{-1}] = \tilde{\cQ}_m$ which implies that $ \cQ_m \subseteq \tilde{\cQ}_m$. 
\end{example}

Let $ F_m = F_m(G,T) $ be the $m$-quasi-flag manifold of $G = SU(2) $ introduced in Section \ref{S4.3} (see Definition~\ref{qspaces}). Recall that $ F_m $ is a $G$-space of homotopy type of a finite CW-complex. 
The next theorem computes the $G$-equivariant $K$-theory of $F_m$, which is the main result of this section.

\begin{theorem} 
\la{ektmq} 
This is a natural isomorphism of $\Z/2$-graded commutative rings
$$
K^{\ast}_G(F_m) \cong \cQ_m(W)
$$
Thus $\,K^0_G(F_m) \cong \cQ_m (W)\,$ and $\,K^1_G(F_m) = 0\,$ for all $ m \in \Z_+ $.
\end{theorem}
\begin{proof}
Recall that $K^{\ast}_G({\rm pt}) = R(G) \cong \Z[t]$, where $t$ corresponds to the $2$-dimensional irreducible representation of $G = SU(2)$. The natural map $\,K^{\ast}_G({\rm pt}) \to K^{\ast}_G(G) \cong K^{\ast}({\rm pt})\,$ is then identified with the projection $\,\Z[t] \to \Z\,$ taking $\, t \mapsto 2\,$. For $m=0$, by definition, we have $\,F_0=G/T\,$, and hence  ({\it cf.} Example~\ref{ExSU2})
\begin{equation} \la{kggt}
K^{\ast}_G(G/T) \cong 
K^{\ast}_T({\rm pt}) = 
R(T) \cong \Z[z,z^{-1}]\ . 
\end{equation}
Thus $K^0_G(F_0) \cong \Z[z,z^{-1}] = \cQ_0(W)$ and $K^1_G(F_0)=0$ as is well known. Further, the map $\,R(G) \to R(T)\,$ induced on $G$-equivariant $K$-theory by $\, G/T \to {\rm pt} \,$ is identified with $\,\Z[t] \to \Z[z,z^{-1}]\,,\ t \mapsto z+z^{-1}$.

Now, recall that $ F_{m+1} = F_m \ast G $, which means
\begin{equation} \la{fiberjoin} F_{m+1} \simeq {\rm hocolim}[F_m \leftarrow F_m \times G \to G]\ . \end{equation}
There is a canonical $G$-equivariant map $F_m \to F_{m+1}$ which we denote by $i_{m,m+1}$, which is nontrivial (not null-homotopic) in the homotopy category of $G$-spaces (see Remark \ref{rem11}). Let $ i_{m,n}: F_m \to F_n $ denote the composite map
$ i_{m,n} := i_{n-1,n} \circ \ldots \circ i_{m,m+1}$ for $n>m$.  We claim that the map $\,i_{0,m}^{\ast}: K_G^{\ast}(F_m) \to K^{\ast}_G(G/T)\,$ induced by $ i_{0,m}: G/T \to F_m $ is injective, and under the isomorphism \eqref{kggt}, it is identified with the inclusion of $ \cQ_m(W)$ in $\Z[z,z^{-1}]$. We prove our claim by induction on $m$. For $m=0$, this is just \eqref{kggt}.

Assume, for some $m \ge 0$, that  $\,K^{\ast}_G(F_m) \cong \cQ_m(W)\,$ and that the map $ i^*_{0,m}: K^{\ast}_G(F_m) \to K^{\ast}_G(G/T)$ is identified with the inclusion of $\cQ_m(W)$ in $\Z[z,z^{-1}]$ as a subring. Then the image of $t \in K^{\ast}_G({\rm pt})$ in $\,K^{\ast}_G(F_m) \cong \cQ_m(W)\,$ is $z+z^{-1}$.
 Since $K^{\ast}_G(G) \cong \Z$ has the free $K^{\ast}_G({\rm pt}) \cong \Z[t]$-module resolution 
$\, 0 \to \Z[t] \to  \Z[t]  \stackrel{\cdot (t-2)}{\longrightarrow} \Z \to 0\,$, the Tor-group
$$
{\rm Tor}_{*}^{R(G)}(K^{\ast}_G(F_m), K^{\ast}_G(G))\ \cong\  {\rm Tor}_*^{\Z[t]}(\cQ_m, \Z)
$$ 
is identified with the homology of the two-term complex 
$\, 0 \to  \cQ_m(W) \stackrel{\cdot (z+z^{-1}-2)}{\longrightarrow}  \cQ_m(W) \to 0 \,$, whose first homology vanishes since $ \cQ_m(W)$ is an integral domain. It follows that Hodgkin's spectral sequence (see Theorem \ref{HSS}) that
$$ 
K^{\ast}_G(F_m \times G) \cong \cQ_m(W)/(z+z^{-1}-2)\, ,
$$
and that the map $K^{\ast}_G(F_m) \to K^{\ast}_G(F_m \times G)$ induced by the projection $F_m \times G \to F_m$ is the canonical quotient map $\pi: \cQ_m(W) \to \cQ_m(W)/(z+z^{-1}-2)$. Next, applying Lemma \ref{KHoPush} to the homotopy pushout \eqref{fiberjoin}, we obtain the 
four-term exact sequence 
\begin{equation} \la{6tes}
0 \to K^0_G(F_{m+1}) \xrightarrow{(i_{m, m+1}, f)^{\ast}} \cQ_m(W) \oplus \Z 
\xrightarrow{ i^*-\pi^*}  
\cQ_m(W)/(z+z^{-1}-2) \xrightarrow{\partial}
 K^1_G(F_{m+1}) \to 0\ ,
\end{equation}
where $\,i:\Z \to \cQ_m(W)\,$ is the structure map of the ring $ \cQ_m(W)) $ and $\,f: G \to F_{m+1}\,$ is the natural map associated to  \eqref{fiberjoin}. It follows from \eqref{6tes} that $K^1_G(F_{m+1})=0$, and 
$$ K^0_G(F_{m+1}) \cong \Ker(i^*-\pi^*) = \Z + (z+z^{-1}-2) \cdot \cQ_m(W) = \cQ_{m+1}(W) \ .
$$
Furthermore, the inclusion $ \cQ_{m+1}(W) \into \cQ_{m}(W)$ is identified with the map $i_{m,m+1}^{\ast}$. This completes the induction step, completing the proof of the theorem.  
\end{proof}

\subsection{The equivariant Chern character} 
Recall that the space $ X_m = X_m(G,T) $ of $m$-quasi-invariants is defined as the homotopy $G$-quotient $ X_m := EG \times _G F_m $. The Borel construction yields a natural map
\begin{equation}\la{alpha}
\alpha: \, K_G^*(F_m) \to K^*(X_m)
\end{equation}
where $ K^*(X) = K^0(X) \oplus K^1(X) $ is the (complex) topological $K$-theory defined by
$\,K^0(X) = [X, \,BU]\,$ and $ K^1(X) = [X,\,U]\,$. Theorem~\ref{ektmq} shows that 
$K_G^*(F_m)$ is a finitely generated $R(G)$-module for all $ m \in \Z_+ $. Hence, by Atiyah-Segal 
Completion Theorem \cite{AS69}, the map \eqref{alpha} extends to an isomorphism
\begin{equation}\la{alpha1}
\widehat{K}_G^*(F_m)_{I_G} \cong K^*(X_m)
\end{equation}
where $ \widehat{K}_G^*(F)_{I_G}  $ denotes the (adic) completion of $ K_G^*(F) $ (as an $R(G)$-module) with respect to the augmentation ideal of $ R(G) $ defined as the kernel of the dimension function $\, I_G := \Ker[\dim:\, R(G) \to \Z]\,$. 
If we identify $ R(G) \cong \Z[z + z^{-1}] $ as the invariant subring of $R(T) \cong \Z[z,z^{-1}] $ as in the proof of
Theorem~\ref{ektmq}, then $ I_G = (z + z^{-1} - 2) $. Thus, as a consequence of \eqref{alpha1}, we get
\begin{cor}
\la{Cor111}
For all $ m \ge 0 $, there is an isomorphism
$$ 
K^*(X_m) \cong \widehat{\cQ}_m(W)_I 
$$ 
where $ (\widehat{\cQ}_m)_I $ denotes the completion of \eqref{mulquasi} with respect to the ideal $ I = (z + z^{-1} - 2) \subset \Z[z+z^{-1}]$.
\end{cor}

Next, we compute a Chern character map relating
equivariant $K$-theory to equivariant cohomology.
Recall that the Chern character of an equivariant vector bundle on
a $G$-space $F$ is defined as the (non-equivariant) Chern character of the associated vector bundle on $ EG \times_G F $. This gives a natural map
\begin{equation}
\la{chGF}
\ch_G(F):\ K^*_G(F) \,\to\, \widehat{H}^{\ast}_G(F, \Q)
\end{equation}
where $\,\widehat{H}^{\ast}_G(F, \Q) := \prod_{k=0}^\infty H^k_G(F, \Q)\,$. The following proposition describes the map \eqref{chGF} for $ F = F_m(G,T) $ explicitly, using the identifications of Theorem~\ref{MTh1} and Theorem~\ref{ektmq}.

\begin{prop}\la{ChMap}
$(1)$ The Chern character map $ \ch_G(F_m): K^*_G(F_m) \to \widehat{H}^{\ast}(X_m, \Q) $ is given by 
\begin{equation}
\la{eqchern}
\exp:\ \cQ_m(W)\,\to \, \widehat{Q}_m(W)\ , \quad 
z \,\mapsto \, \sum_{n=0}^{\infty} \frac{x^n}{n!}\ ,
\end{equation}
where $\, \widehat{Q}_m(W) := Q_m(W) \otimes_{\Q[x^2]} \Q[[x^2]] \,$ is the completed ring of quasi-invariants of $ W = \Z/2\Z$. 

$(2)$ The map $\, \ch_G(F_m)\,$ factors through \eqref{alpha} inducing an isomorphism on rational $K$-theory
$$
K(X_m)_{\Q} \,\cong\, \widehat{H}^{\ast}(X_m, \Q)\, \cong \,\widehat{Q}_m(W)
$$
\end{prop}
\begin{proof}
For $ F_0 = G/T $, we can identify $ K^*_G(G/T) \cong R(T) \cong \Z[z,z^{-1}] $ and $ \widehat{H}^{\ast}_G(G/T, \Q) = \widehat{H}^{\ast}(BT, \Q) \cong \Q[[x]]$ as in (the proofs of) Theorem \ref{MTh1} and Theorem \ref{ektmq}. With these identifications, it is well known that the equivariant Chern character is given by exponentiation
 (see, e.g., \cite[Example A.5]{FRW}): 
\begin{equation}
\la{exp1}
\ch_G(G/T):\,K^*_G(G/T)\, \to\,\widehat{H}^{\ast}(BT, \Q) \,,\qquad z \mapsto \exp(x) \ .
\end{equation}
Now, by functoriality of the Chern character, the maps $\, G/T \xrightarrow{i_{0,m}} F_m \to {\rm pt}\,$ give a commutative diagram of ring homomorphisms
\begin{equation} 
\la{Chernc2} 
\begin{diagram}[small]
K^*_G(\mathrm{pt}) & \rInto^{\ch_G({\rm pt})} & \widehat{H}^{*}(\mathrm{pt}, \Q)\\
\dInto & & \dInto\\
K^*_G(F_m) & \rTo^{\ch_G(F_m)} & \widehat{H}^{*}(F_m, \Q)\\
\dInto^{i^*_{0,m}} & & \dInto_{i^*_{0,m}}\\
K^*_G(G/T) & \rInto^{\ \ch_G(G/T)} & \widehat{H}^{*}(G/T, \Q)\\
\end{diagram} \ .
\end{equation}
where the vertical maps as well as the top and the bottom horizontal maps are injective. Hence, the
map in the middle, $ \ch_G(F_m): K^*_G(F_m) \to \widehat{H}^{*}(F_m, \Q) $, is also injective, and it is 
given by restriction of the exponential map \eqref{exp1}. This proves the first claim of the proposition.
The second claim follows from the first and Corollary~\ref{Cor111}.
\end{proof}

\subsection{$K$-theory of `fake' spaces of quasi-invariants}
\la{S6.3}
In this section, we compute the $K$-theory of `fake' spaces of quasi-invairants $X_m(\Omega B, T)$ constructed in Section \ref{S5.3}. We will keep the notation $ G = SU(2) $ and $\, T = U(1) $ and use the identification $ K^*(BT) \cong \Z[[t]] $ as in the previous section. Let $B$ be a space in the genus of $BG$ that admits an essential map from $BT$. By \cite[Proposition 2.1]{DY01}, there is an isomorphism of rings $ K^{\ast}(B) \cong \Z[[u]]\,$, such that 
for any essential map $f: BT \to B$, the induced map $\,f^{\ast}:\,K^{\ast}(B) \to K^{\ast}(BT)\,$ is given by
$$
f^{\ast}(u) = \mathrm{deg}(f) t^2 + \text{higher order terms in } t\ ,
$$
where the integer $ \deg(f)$ coincides (up to sign) with the degree of $f$ in integral (co)homology in dimension $4$ ({\it cf.} Lemma~\ref{CohB}). In fact, by a general result of Notbohm and Smith (see \cite[Theorem 5.2]{NS90}), the assignment $ f \mapsto f^* $ 
gives a bijection between the homotopy classes of maps from $BT$ to $B$ an the 
$\lambda$-ring homomorphisms from $K^{\ast}(B)$ to $K^{\ast}(BT)$:
$$
[BT, B]_* \cong \Hom_{\lambda}(K^{\ast}(B), K^{\ast}(BT))\,.
$$
Next, recall that, by Theorem~\ref{esmap}, among all essential maps $ BT \to B $,
there is a `maximal' one $ p_B: BT \to B $, for which $\deg(p_B) = N_B$, where $N_B$ is
the integer defined by \eqref{fdegree}: the corresponding power series
\begin{equation} 
\la{bgen} 
p_B^{\ast}(u) = N_B t^2+ \text{higher order terms in } t\ . 
\end{equation}
is a useful $K$-theoretic invariant of $B$ that depends on the Rector invariants $(B/p)$ (see \cite{DY01}).
Using \eqref{bgen}, we define a sequence of subrings $ \cQ_m(B) $ in $ \Z[[t]]$ inductively by the rule:
\begin{equation}\la{ringB}
\cQ_{0}(B) := \Z[[t]]\ ,\qquad \cQ_{m}(B) := \Z + p_B^{\ast}(u) \cQ_{m-1}(B)
\ ,\quad m \ge 1 \ .
\end{equation}
Note that there are natural inclusions 
\begin{equation*}
\la{natin}
\cQ_{0}(B) \supseteq \cQ_1(B) \supseteq \ldots \supseteq \cQ_m(B) \supseteq \cQ_{m+1}(B) \supseteq \ldots
\end{equation*}
which are all ring homomorphisms. 

\begin{example}
\la{ExBBG}
For $ B = BG $, one can easily compute the power series $ p_B(u)$ in an explicit form.
Recall that the Atiyah-Segal completion theorem gives an isomorphism $K^{\ast}(BG) \cong \widehat{K}^{\ast}_G({\rm pt})_{I}$, where $ I = I_G$ is the ideal of virtual representations 
in $\, K^{\ast}_G({\rm pt})\cong R(G) \,$ of dimension $0$. If we identify $ K^{\ast}_G({\rm pt}) \cong \Z[v]$, where $v$ is the standard $2$-dimensional representation of $G$, then $ \,I = (v-2)\,$, and $K^{\ast}(BG) \cong \Z[[u]]$, where $u=v-2$. Similarly, $K^{\ast}(BT) \cong \Z[[t]]$, where $t=z-1$, with $z$ standing for the generating character of $T$. The naturality of \eqref{alpha} (with respect to the $G$-equivariant map $p: G/T \to {\rm pt}$) yields the commutative diagram 
$$ 
\begin{diagram}[small]
K^{\ast}_G({\rm pt}) \cong \Z[u]  & \rTo^{\alpha} & K^{\ast}(BG) \cong \Z[[u]]\\
\dTo^{p^{\ast}} & & \dTo^{p_B^{\ast}}\\
K^{\ast}_G(G/T) \cong \Z[t] & \rTo^{\alpha} & K^{\ast}(BT) \cong \Z[[t]] \\ 
\end{diagram}\ .
$$
\vspace*{1ex}
Since $p^{\ast}(v)$ is the restriction of $v$ to $T$, we have $p^{\ast}(v)=z+z^{-1}$. Hence, 
$$
p_B^{\ast}(u) \,=\,  p_B^{\ast}(v-2)
              \, = \, z+z^{-1}-2 
              \,= \, (1+t) +\frac{1}{1+t} -2
               \,=\,  \frac{t^2}{1+t}
$$
It follows that $ \cQ_m(BG) \cong \widehat{\cQ}_m(W)$, where the right-hand side is the completion of $\cQ_m(W)$ with respect to the ideal generated by $z+z^{-1}-2$ ({\it cf.} Corollary~\ref{Cor111}).
\end{example}

Now, we state the the main result of this section.
\begin{theorem} 
\la{RectorQI}
There are isomorphisms of rings 
\begin{equation}
\la{isoB}
K^{\ast}[X_m(\Omega B, T)] \cong \cQ_{m}(B) \ , \ \forall\,m \ge 0\,. 
\end{equation}
In particular, $\,K^1[X_m(\Omega B,T)] = 0\,$ for all $m\ge 0$. The maps $\,\pi_m^{\ast}:\,K^{\ast}[X_{m+1}(\Omega B, T)] \to K^{\ast}[X_m(\Omega B, T)]\,$ induced by the Ganea maps $ \pi_m $ in \eqref{RectorTow}  correspond under \eqref{isoB} to the natural inclusions $ \cQ_{m+1}(B) \into \cQ_{m}(B)$, and hence are all injective.
\end{theorem}
To prove Theorem~\ref{RectorQI} we will use an Eilenberg-Moore spectral sequence for $K$-theory  in the following form.
\begin{lemma} 
\la{EMK}
Let $F \to E \to B$ be a $($homotopy$)$ fibration sequence over a base $B$ such that $\,K^{\ast}(\Omega B)
\,$ is an exterior algebra in a finite number of generators of odd degrees. Then there is a multiplicative spectral sequence with 
$\,E_2^{i, \ast} \cong \mathrm{Tor}^{K^{\ast}(B)}_i(\Z, K^{\ast}(E)) \,$
that strongly converges to $K^{\ast}(F)$.
\end{lemma}
The proof of Lemma~\ref{EMK} can be found, for example, in \cite{JO99} (see Main Theorem, Part 3).
\begin{proof}[Proof of Theorem~\ref{RectorQI}]
We further claim that the ring homomorphisms
$$
j_{m,B}^{\ast}:\ K^{\ast}[X_m(\Omega B, T)] \to K^{\ast}[F_m(\Omega B, T)]
$$
induced by the fibre maps $ j_{B,m} $ in
\eqref{RectorTow} are surjective, and with \eqref{isoB}, they induce isomorphisms
$$
K^{\ast}[F_m(\Omega B, T)] \cong  \cQ_{m}(B)/(p_B^{\ast}(u))
$$
We prove these facts together with the claims of Theorem~\ref{RectorQI} by induction on $m$. 

For $m=0$, we need only to compute $K^{\ast}[F_0(\Omega B, T)]$. This can be done using Lemma \ref{EMK}. Note that $ K({\rm pt}) \cong \Z$ has the obvious free resolution over $\,K^{\ast}(B) \cong \Z[[u]]\,$:
\begin{equation} 
\la{zres} 
0 \to \Z[[u]] \xrightarrow{\cdot u}  \Z[[u]] \to \Z \to 0  
\end{equation}
Hence $\mathrm{Tor}^{K^{\ast}(B)}_\ast(\Z, K^{\ast}(BT))$ can be identified with homology of the two-term complex $ 0 \to \Z[[t]] \xrightarrow{p^*_B(u)} \Z[[t]] \to 0\,$, where the map in the middle is given by the power series \eqref{bgen}. Since $\Z[[t]]$ is an integral domain, it follows that $\mathrm{Tor}^{K^{\ast}(B)}_i(\Z, K^{\ast}(BT))=0$ for $i>0$. The Eilenburg-Moore spectral sequence of Lemma \ref{EMK} therefore collapses, giving
an isomorphism
$$
K^{\ast}[F_0(\Omega B, T)] \cong \Z[[t]]/(p_B^{\ast}(u))
$$

Next, assume that $\,K^{\ast}[X_m(\Omega B, T)] \cong \cQ_{m}(B)$ and that $\,K^{\ast}[F_m(\Omega B, T)] \cong \cQ_{m}(B)/(p_B^{\ast}(u))$, with $j_{m,B}^{\ast}$ being the canonical quotient map. Since 
$$
X_{m+1}(\Omega B, T) \,\simeq\, \hocolim \,[\, 
\mathrm{pt} \xleftarrow{i_m} F_m(\Omega B, T) \xrightarrow{j_{m,B}} X_m(\Omega B, T)\,]\,,
$$
and since $\,K^1(\mathrm{pt})= K^1[F_m(\Omega B, T)] = K^1[X_m(\Omega B, T)]= 0\,$, Lemma \ref{KHoPush} (with $G$ trivial group) yields the four-term exact sequence 
$$
0 \to K^0[X_{m+1}(\Omega B, T)] \xrightarrow{(i^{\ast}_m,\, \pi^{\ast}_m)} \Z \oplus \cQ_{m}(B)  \xrightarrow{p^{\ast}_m - j^{\ast}_{m,B}} \cQ_{m}(B)/(p_B^{\ast}(u)) \xrightarrow{\partial}\\
 K^1(X_{m+1}(\Omega B, T))\to 0\,.
$$
Here $p_m$ is the trivial map from $F_m(\Omega B, T)$ to the point. Since $j_{m,B}^{\ast}$ is surjective, $K^1[X_{m+1}(\Omega B, T)] = 0$. The above six-term exact sequence also shows that $\,K^0[X_{m+1}(\Omega B, T)] \cong \Ker(p^{\ast}_m -j^\ast_{m,B}) \subseteq \Z \oplus \cQ_{m}(B)$ (with isomorphism given by the map $\,(i^{\ast}_m,\,\pi^{\ast}_m))\,$. Projection to $ \cQ_{m}(B)$ identifies this kernel with $ \cQ_{m+1}(B) =
\Z + p_B^{\ast}(u) \cQ_{m}(B) \subset \cQ_{m}(B)\,$. It follows that $\,K^{\ast}[X_{m+1}(\Omega B, T)] \cong \cQ_{m+1}(B)\,$, with $\pi^{\ast}_m$ being the inclusion of $ \cQ_{m+1}(B)$ into $\cQ_{m}(B) $. Finally, by taking the (completed) tensor product of the resolution \eqref{zres} with $ \cQ_{m+1}(B) $, we see that $\,\mathrm{Tor}^{K^{\ast}(B)}_i(\Z, \cQ_{m+1}(B))\,$ is the homology of the complex
$$
0 \to \cQ_{m+1}(B) \xrightarrow{\,p_B^{\ast}(u)\,} \cQ_{m+1}(B) \to 0   
$$
where the map in the middle is given by multiplication by
the formal power series \eqref{bgen}.
Since $\cQ_{m+1}(B) \subseteq \Z[[t]]$ is an integral domain, $\,\mathrm{Tor}^{K^{\ast}(B)}_i(\Z, K^{\ast}(X_{m+1})) =0 \,$ for $i>0$. The spectral sequence of Lemma \ref{EMK} associated with the fibration sequence $\,F_{m+1} \to X_{m+1} \to B\,$ therefore collapses, giving 
$$
K^{\ast}[F_{m+1}(\Omega B, T)] \cong \cQ_{m+1}(B)/(p_B^{\ast}(u))\ ,  
$$
with $j^\ast_{m+1,B}$ being the canonical quotient map. This completes the induction step and thus finishes the proof of the theorem.
\end{proof}
Theorem~\ref{RectorQI} allows one to distinguish spaces
of quasi-invariants of the same multiplicity associated to homotopically distinct spaces in the genus of $BG$. First, we recall
that the topological $K$-theory $K^{\ast}(X)$
of any space $X$ of homotopy type of a CW complex carries a natural filtration 
$$ F^0K^{\ast}(X) \supseteq F^1K^*(X) \supseteq \ldots \supseteq F^nK^*(X) \supseteq F^{n+1}K^*(X) \supseteq \ldots 
$$
where $F^n K^{\ast}(X)$ is defined to be the kernel of the restriction map $\,K^{\ast}(X) \to K^{\ast}(X_{n-1})\,$ corresponding to the $(n-1)$-skeleton of $X$. This filtration is functorial: any map $f: X \to X'$  of spaces, each of which has homotopy type of a CW complex, induces a morphism of filtered rings $\,f^\ast:K^{\ast}(X') \to K^{\ast}(X)$. Moreover, by  Cellular Approximation Theorem, it is independent of the CW structure in the sense that using a different CW structure on $X$ will not change the isomorphism type of $K^*(X)$ as a filtered ring.

\begin{cor} 
\la{rsdistinct}
Let $B$ and $B'$ be two spaces in the genus of $BG$ admitting essential maps from $ BT $. Assume that
$\,N_B \neq N_{B'}\,$. Then $\,X_m(\Omega B, T) \not\simeq X_m(\Omega B', T)\,$ for any $m \geq 0$. 
In particular, $ X_m(\Omega B,T)$ is not homotopy equivalent to $ X_m(G, T) $ for any $B \not\simeq BG $. 
\end{cor}
\begin{proof}
Let $\,\tilde{\pi}_m: BT \to X_m(\Omega B, T)$ denote 
the composite map $\pi_{m-1} \circ \ldots \circ \pi_0$
in \eqref{RectorTow}. By Theorem~\ref{RectorQI}, this map
induces an embedding
$$\tilde{\pi}_m^{\ast}: K^{\ast}[X_m(\Omega B, T)] \cong \cQ_{m}(B) \hookrightarrow \Z[[t]] \cong K^{\ast}(BT) 
$$
which is a morphism of filtered rings. Now, recall that
$ BT \simeq \c\bP^\infty $; the generator $t$ in $\,K^{\ast}(BT) \cong K^*(\c\bP^\infty) = \Z[[t]]\,$ can be taken in the form $ t = b \xi $, where $ \xi \in F^2 K^2(BT) $ and $ b \in K^{-2}({\rm pt})$ is the Bott element (see \cite[Sect. 3]{DY01}). Hence $\,t \in F^2K^0(BT) $, and therefore, by \eqref{bgen}, we have 
\begin{equation*}
\la{congB}
p_B^{\ast}(u) \equiv N_B t^2\, (\mathrm{mod}\, F^5K^{\ast}(BT))\quad \mbox{in}\ \Z[[t]] 
\end{equation*}
Now, by Theorem~\ref{RectorQI},
$$
K^{\ast}[X_m(\Omega B, T)] \cong \cQ_{m}(B) = \Z + \Z \cdot p_B^{\ast}(u) + \ldots + \Z \cdot p_B^{\ast}(u)^{m-1} +   p_B^{\ast}(u) \cdot \Z[[t]]\,.$$
Hence
$$
K^{\ast}[X_m(\Omega B, T)]/F^5K^{\ast}[X_m(\Omega B, T)] \cong \Z + \Z \cdot N_B  t^2\ ,
$$
where the generator $\,N_B t^2\,$ is square zero. It follows that if $p$ is a prime then 
$$ K^{\ast}[X_m(\Omega B, T)]/(p,\, F^5K^{\ast}[X_m(\Omega B, T)]) \, \cong \, \begin{cases} 
                   (\Z/p\Z) + (\Z/p\Z) \cdot \bar{N}_B t^2 & \text{ if }\ p \nmid N_B\\
                   (\Z/p\Z) & \text{ if }\ p\, |\,N_B
                   \end{cases} 
                   $$
where $\,(p,\,F^5K^*(X_m))\,$ denotes the ideal in $K^*(X_m)$ generated by $ p \in \Z $ and $F^5K^*(X_m)$.
This shows that $X_m(\Omega B, T)$ is not homotopy equivalent to $X_m(\Omega B', T)$ unless $\,N_B = N_{B'}\,$.
\end{proof}
\begin{remark}
The converse of Corollary~\ref{rsdistinct} also holds true in the following sense: if two spaces $B$ and $B'$ in the genus of $BG$ have homotopy equivalent towers of spaces
of quasi-invariants $\{ 
X_m(\Omega B,T), \pi_m\}_{m\ge 0} $ and $ \{X_m(\Omega B',T), \pi_m'\}_{m \ge 0}$, then 
$ B \simeq B' \,$. This simply follows from the fact that 
$$
\hocolim_{m \in \Z_+} X_m(\Omega B, T) \simeq B\ ,
$$
which is a consequence of Ganea's Theorem~\ref{GThm}.
\end{remark}
\section{Elliptic cohomology}
\la{S7}
In this section, we compute complex analytic $T$- and
$G$-equivariant elliptic cohomology of the quasi-flag manifolds $ F_m(G,T)$. We express the result in two ways: geometrically (in terms of coherent sheaves on a given elliptic curve $E$) and analytically (in terms of $\Theta$-functions and $q$-difference equations). We also compute the spaces of global sections of the elliptic cohomology sheaves of $F_m(G,T)$ with coefficients twisted by tensor powers of the Looijenga line bundle on $E$. This last computation provides a motivation for our definition of {\it elliptic quasi-invariants} 
of $W$.

\subsection{Equivariant elliptic cohomology}
\la{S7.1}
Complex analytic elliptic cohomology was introduced by I. Grojnowski (see \cite{Gr07}). We will follow the approach of \cite{Ga14} that relies on earlier topological results of \cite{An00} and \cite{Ros01}. We begin by briefly recalling the main definitions.

Let $E$ be an elliptic curve defined analytically over $ \c$. Given a compact connected
abelian Lie group $T$ (i.e., $ T \cong U(1)^n $), we write $ \breve{T} := \Hom(U(1),\,T) $ for its cocharacter lattice and set
$$ 
\M_T := \breve{T} \otimes_{\Z} E \,,
$$
which is an abelian variety of rank $ n = {\rm rk}(T) $ over $\c$. The {\it $T$-equivariant elliptic cohomology} is defined as a functor on the (homotopy) category of finite $T$-CW complexes with values in the category of (complex-analytic) coherent sheaves on $ \M_T $:
\begin{equation}
\la{EllT}
\Ell_T^*:\ \Ho(\Top_T^{\rm fin})\,\to\, \Coh(\M_T)\,.
\end{equation}
This functor has the basic property that $\,\Ell_T^*(T/T')\,\cong\, \O_{\M_{T'}} $
for any closed subgroup $ T' \subseteq T $, where  $\O_{\M_{T'}} = i^* \O_{\M_{T}}$ is the restriction of the 
structure sheaf of $ \M_T $ to $ \M_{T'}$ (see \cite[2.1(1)]{Ga14}). In particular, we have 
\begin{equation}
\la{prEll}
\Ell_T^*({\rm pt}) \cong \O_{\M_T}
\end{equation}

Now, if $G$ is a compact connected Lie group with maximal torus $T$ and Weyl
group $W$, we define the {\it $G$-equivariant} elliptic cohomology of a $G$-space $X$ by
\begin{equation}
\la{TWG}
 \Ell_G^*(X) := \Ell_T^*(X)^W \,,  
\end{equation}
where $X$ is viewed as a $T$-space by restricting $G$-action (see \cite[3.4]{Ga14}). To compute the $G$-equivariant elliptic cohomology we thus need to compute the $T$-equivariant elliptic cohomology of a $G$-space $X$ and take its $W$-invariant sections.

The coherent sheaves $\Ell_T^*(X) $ usually do not have many interesting global sections. To remedy this deficiency one considers a twisted version of elliptic cohomology, where the sheaves $\Ell_T^*(X)$ are tensored with
a certain ample line bundle on $\M_T$. Recall that, if $G$ is a simple, simply connected compact Lie group
with a root system $\R$, there is a canonical $W$-equivariant line bundle $ {\mathscr L} $ on $ \M_T $ associated to the invariant symmetric bilinear form $I$ on the coroot lattice of $\R$ such that $ I(\alpha, \alpha) = 2 $
for all roots of smallest length in $\R$; this line bundle has $I$ as its Chern class and 
has degree equal to the order of the center of $G$  (see \cite{Lo77}). Following \cite{An00, Ga14}, we will refer to $\cL$ as the {\it Looijenga bundle} on $\M_T$ and define the $T$- and $G$-equivariant elliptic cohomology of $X$ with coefficients twisted by $ {\mathscr L}$ by
\begin{eqnarray}
\ELL_T^*(X, {\mathscr L}) &:=&
\bigoplus_{n=0}^{\infty}\, H_{\rm an}^0(\M_T,\,\Ell_T^*(X) \otimes {\mathscr L}^{n})\la{ELLL}\\
\ELL_G^*(X, {\mathscr L}) &:=&
\bigoplus_{n=0}^{\infty}\, H_{\rm an}^0(\M_T,\,\Ell_T^*(X) \otimes {\mathscr L}^{n})^W\la{ELLLW}
\end{eqnarray}
where $H_{\rm an}^0$ stands for the global (holomorphic) sections of the coherent sheaf $ \Ell_T^*(X) $  twisted by the tensor powers of ${\mathscr L}$. Note that \eqref{ELLL} and \eqref{ELLLW} are naturally graded modules over the graded commutative rings 
\begin{equation}
\la{SEW}
\ELL_T^*({\rm pt}, {\mathscr L}) = \bigoplus_{n=0}^{\infty}\, H_{\rm an}^0(\M_T,\,{\mathscr L}^{n})\quad \mbox{and}\quad \ELL_G^*({\rm pt}, {\mathscr L}) = \bigoplus_{n=0}^{\infty}\, H_{\rm an}^0(\M_T,\,{\mathscr L}^{n})^W\,
\end{equation}
which we denote by $S(E)$ and $ S(E)^W$, respectively. Following \cite{Lo77}, we also write $ S(E)^{-W}$
for the subspace of $S(E)$ consisting of all $W$-anti-invariant sections. The main theorem of \cite{Lo77} asserts that $ S(E)^W$ is a graded polynomial algebra generated freely by $\,l+1\,$ homogeneous elements, while $ S(E)^{-W} $ is a free module over $S(E)^W$ of rank one (see \cite[(3.4)]{Lo77}).  The generators of $S(E)^W$
are called the {\it Looijenga theta functions} on $ \M_T $.

\subsection{Elliptic cohomology of quasi-flag manifolds}
\la{S7.2}
In the rank one case $(T = U(1))$, we can identify $ \M_T = E $ and take for a model of $E$ 
the {\it Tate curve} $ E_q := \c^*\!/q^{\Z} $ with $ 0 < |q| < 1 $. The latter is defined
as the quotient of the punctured line $ \c^* = \c\setminus\{0\} $ (viewed as a complex
analytic manifold) by the free action of the infinite cyclic group $\Z\,$:
\begin{equation}
\la{anact}
\c^* \to \c^*\ ,\ z \mapsto q^n z\,.  
\end{equation}
We write $ A := \O_{\rm an}(\c^*) $ for the ring of global analytic functions on $ \c^* $ and 
define $ \A_q := A \rtimes_q \Z $ to be the crossed product algebra for the action of $\Z$ on $A$
induced by \eqref{anact}. Letting $ \xi $ be the (multiplicative) generator of $\Z$, we 
can present $ \A_q $ as a skew-polynomial algebra $A[\xi, \xi^{-1}]$ with multiplication determined by the commutation relation $\,\xi \cdot a(z) = a(qz) \cdot \xi \,$ for $ a(z) \in A $.  The left modules over $\A_q$
can be identified with spaces of global sections of $\Z$-equivariant quasi-coherent sheaves on $\c^*$. The natural projection $ \pi: \c^* \to E_q $ induces then the additive functor
\begin{equation}
\la{funeq}
\Coh(E_q) \to \Mod^{\rm f.p.}_{A}(\A_q)\ ,\quad
\cF \mapsto \widetilde{\cF} := H_{\rm an}^0(\c^*,\,\pi^*\cF)\,,    
\end{equation}
that maps the coherent sheaves on the analytic curve $E_q$ to left $\A_q$-modules admitting a finite presentation $\,A^{\oplus m} \to A^{\oplus n} \to M \to 0\,$ over the subalgebra $ A \subset \A_q $. 
The following proposition is a well-known result that provides a convenient algebraic description of the category $ \Coh(E_q) $; its proof can be found in various places (see, for example, \cite[Thm 2.2]{SV03} or \cite[Sect. 2]{vdPR07}).
\begin{prop}
\la{PSV}
The functor \eqref{funeq} is an exact equivalence of abelian tensor categories.
\end{prop}
We remark that the tensor structure on $ \Coh(E_q) $ is the standard geometric one (defined by tensor product of sheaves of $\O_{E_q}$-modules), while the tensor structure on $ \Mod^{\rm f.p.}_{A}(\A_q) $
is defined by tensoring $\A_q$-modules over the subalgebra $A$ with the action of $\A_q$ on $ M_1 \otimes_A M_2 $ given by $ \xi \cdot (m_1 \otimes m_2 ) = (\xi\cdot m_1) \otimes (\xi \cdot m_2) $.
The vector bundles on $E_q$ correspond under \eqref{funeq} to $\A_q$-modules that are free of finite rank over $ A$; such modules form a full subcategory of $ \Mod^{\rm f.p.}_{A}(\A_q) $ closed under the tensor product. The cohomology of a coherent sheaf $\cF$ on $E_q$ can be computed algebraically in terms of $\A_q$-modules as invariants and coinvariants of the induced action of $ \Z $ on the corresponding $A$-module $\widetilde{\cF}$ (see \cite[Lemma 2.1]{vdPR07}):
\begin{equation}\la{H01}
H^0_{\rm an}(E_q,\,\cF) \,\cong\, \Ker\,(\xi- \id\,:\,\widetilde{\cF})
\ ,\quad  
H^1_{\rm an}(E_q,\,\cF) \,\cong\, {\rm Coker}\,(\xi- \id\,:\,\widetilde{\cF})
\,.
\end{equation}
where $\, \xi \,$ is the multiplicative generator of the copy of $\Z $ in $ \A_q$ acting on the $\A_q$-module $\widetilde{\cF}$. 

\begin{example}\la{Ex222}
The structure sheaf $ \O_{E_q}$ of $E_q$ corresponds under \eqref{funeq} to the cyclic module
$\,\widetilde{\O}_{E_q} = \A_q/\A_q(\xi-1)\,$, which can be identified as $ \widetilde{\O}_{E_q} \cong A e $
with generator $e$ satisfying the relation $ \xi e = e $. The line bundle $ \O_{E_{q}}([1])$ corresponds to $\, \A_q/\A_q(\xi+z) \cong A e\,$, with $e$ satisfying $ \xi e = - z e$. 
More generally, any line bundle on $E_q$ of degree $d$
corresponds to a cyclic $\A_q$-module $ A e $, where the generator $e$ satisfies the relation
$ \xi e = c z^d e $ for some $ c\in \c^* $ (see \cite[Example 2.2]{vdPR07}). 
\end{example}
\vspace{.5ex}

We now proceed with computing elliptic cohomology of the spaces $ F_m = F_m(G,T) $.
For a fixed Tate curve $ E_q $, we first describe the $T$-equivariant cohomology,
presenting the answer in two ways: in terms of coherent sheaves on $E_q$ and 
in terms of $\A_q$-modules via the equivalence \eqref{funeq}.
\begin{theorem}
\la{ThTell}
For all $m\ge0$, there are natural isomorphisms of coherent sheaves in $ \Coh(E_q)$
\begin{equation}\la{cohFm}
\Ell_T^*(F_m)\,\cong\, \O_{E_q} \times_{\O_{E_q}/{\mathcal J}^{2m+1}} \O_{E_q}\ ,    
\end{equation}
where $ {\mathcal J} $ is the subsheaf of ideals in the structure sheaf $ \O_{E_q}$ vanishing at
the identity of $ E_q $. Under the equivalence \eqref{funeq}, the coherent
sheaf \eqref{cohFm} corresponds to the $\A_q$-module
\begin{equation}
\la{cohFm1}
\widetilde{\Ell}_T^*(F_m) \cong A \times_{A/\langle\Theta\rangle^{2m+1}} A \,,   
\end{equation}
where the action of $\A_q $ on the fibre product is induced by the natural action of $\A_q$ on $A$
and $ \langle \Theta\rangle $ denotes the $($principal$)$ ideal in $ A = \O_{\rm an}(\c^*) $ generated by the classical theta function
\begin{equation}
\la{Theta}
\Theta(z) :=  (1-z)\,\prod_{n>0}(1-q^n z)(1-q^{n}z^{-1})\,=\, \frac{1}{(q;q)_{\infty}} \ \sum_{n \in \Z}\, q^{\frac{n(n-1)}{2}} (-z)^n 
\,.
\end{equation}
\end{theorem}
\begin{proof} Recall that, by Lemma~\ref{Teq}, there is a $T$-equivariant homeomorphism
$$ 
F_m  \,\cong\, \Sigma\,E_{2m}(T) = \hocolim\,({\rm pt} \leftarrow E_{2m}(T) \rightarrow {\rm pt})\,,
$$ 
where $ E_{2m}(T) = T^{\ast (2m+1)} $ is Milnor's $2m$-universal $T$-bundle.  As equivariant $K$-theory, 
the $T$-equivariant elliptic cohomology is known to satisfy the Mayer-Vietoris property (see, e.g., \cite[Theorem 3.8]{Ros01}). Hence, as in Lemma \ref{KHoPush}, there is a six term long exact sequence of sheaves on $E_q$:
$$ 
\begin{diagram}[small]
\Ell^0_T(F_m) & \rTo & \Ell^0_T({\rm pt}) \times \Ell^0_T({\rm pt}) & \rTo^{p_1^{\ast}-p_2^{\ast}} & \Ell^0_T( E_{2m}(T))\\
\uTo & & & & \dTo \\
\Ell^1_T(E_{2m}T) & \lTo & \Ell^1_T({\rm pt}) \times \Ell^1_T({\rm pt}) & \lTo & \Ell_T^1(F_m)
\end{diagram}\ . 
$$
where the arrow on top of the diagram is given on sections by $(x_1, x_2) \mapsto p_1^\ast(x_1) - p_2^\ast(x_2)$, with $p_1$ and $ p_2 $ representing two copies of the canonical map $ E_{2m}(T) \to {\rm pt} $. By \eqref{prEll},
we know that $\Ell^{\ast}_T({\rm pt}) \cong \O_{E_q}$; on the other hand, by Lemma \ref{ecohtj} (see below),
$$ \Ell^{\ast}_T(E_{2m}T) \cong \O_{E_q}/{\mathcal J}^{2m+1} ,
$$
where $ {\mathcal J} \subset \O_{E_q} $ is the subsheaf of sections vanishing at $ 1 \in E_q $.
Hence, by exactness, the above commutative diagram shows that $\,\Ell_T^1(F_m(G,T))=0\,$ and 
$$
\Ell_T^0(F_m)= {\rm Ker}(p_1^{\ast}-p_2^{\ast}) \cong \O_{E_q} \times_{\O_{E_q}/{\mathcal J}^{2m+1}} \O_{E_q}\,.
$$
This proves the first claim of the theorem. 

Now, to prove the second  claim we observe that the skyscraper sheaf $\cF := i_{1,*}\C $ on $ E_q $ (with stalk $\c$ supported at $ \{1\} $) corresponds under \eqref{funeq} to the quotient module $ \widetilde{\cF} \cong A/\langle \Theta\rangle $, where the action of $ \A_q $ is induced by the natural action of $\A_q$ on $A$. Indeed, $ \widetilde{\cF} $ is isomorphic to the cokernel of the map $ \O_{E_q} \to \O_{E_q}([1]) $, which is given (with identifications of Example~\ref{Ex222}) by $ e \mapsto \Theta e $. This follows from the fact that as a global analytic function on $\c^*$, $\, \Theta = \Theta(z) $ has simple zeroes exactly at the points $ z = q^n\, (n \in \Z) $. Hence the ideal sheaf ${\mathcal J} \subset \O_{E_q}$ corresponds to the ideal $ \langle \Theta \rangle = A \Theta $ in $A$, and more generally,  since \eqref{funeq} is a tensor functor, ${\mathcal J}^{2m+1}$ corresponds to $\langle\Theta\rangle^{2m+1} = A \Theta^{2m+1}$ for all $m \ge 0$. Now, since \eqref{funeq} is an exact additive functor, it takes the fibre product $ \O_{E_q} \times_{\O_{E_q}/{\mathcal J}^{2m+1}} \O_{E_q}$ in $\Coh(E_q)$ to the module $A \times_{A/\langle\Theta\rangle^{2m+1}} A$ in $\Mod^{\rm f.p.}_{A}(\A_q)$, thus completing the proof of the theorem.
\end{proof}

\begin{lemma} \la{ecohtj} There are isomorphisms of sheaves
$\,\Ell^{\ast}_T(E_n T) \cong \O_{E_q}/{\mathcal J}^{n+1}$ for all $\,n\geq 0\,$.
\end{lemma}
\begin{proof}
 Note that $T$ acts freely on $E_n(T) := T^{\ast (n+1)}$. Recall (see \cite[Sect. 3.2]{Ros01}) that if $X$ is a finite $T$-space, the stalk at $a \in E$ of $\Ell^{\ast}_T(X)$ is isomorphic to $H^{\ast}_T(X^a;\c) \otimes_{\c[z]} \O_{\c^{\ast},1}$, where $\O_{E_q,1}$ stands for the ring of germs of analytic functions at $1 \in E_q$. Here, $X^a$ stands for the fixed point space $X^{T_a}$, where $T_a=\Z/k\Z \subset T$ if $a$ is of finite order $k$ in $E$, and $T_a=T$ if $a$ is not of finite order in $E$. It follows that of $T$ acts freely on $X$, the stalk $\Ell^{\ast}_T(X)_a$ of $\Ell^{\ast}_T(X)$ at $a$ vanishes for $a \neq 1$. Hence, $\Ell_T(E_n T)_a=0$ for $a \neq 1$, and for $U$ a small neighborhood of $1$ in $E_q$, $\Ell^{\ast}_T(X)|_U \cong H^{\ast}_T(E_n T;\c) \otimes_{\c[x]} \O_{E_q}|_U$, where $\O_{E_q}|_U$ acquires the structure of a sheaf of $\c[x]$-modules via the map $\c[x] \to \O_{E_q}(U), x \mapsto \theta$, where $\theta$ is a generator of the maximal ideal of the local ring $\O_{\c^{\ast},1}$. The desired lemma therefore, follows from the fact that $H^{\ast}_T(E_n T;\c) \cong H^\ast(B_n T;\c) \cong \c[x]/x^{n+1}$ (see the proof of Lemma \ref{TLem} above).
\end{proof}
To compute the $G$-equivariant elliptic cohomology of $ F_m $ we need to refine the result of
Theorem~\ref{ThTell} by taking into account the action of $W = \Z/2\Z\,$ on $\,\Ell_T^*(F_m)$. To this end we first refine the result of Proposition~\ref{PSV}. Observe that the equivalence \eqref{funeq} extends to the category  of $W$-equivariant coherent sheaves on $E_q$: 
\begin{equation}
\la{funeqW}
 \Coh_W(E_q) \,\xrightarrow{\sim}\, \Mod^{\rm f.p.}_{A}(\A_q \rtimes W)\ ,
 \end{equation}
where the category of $\A_q$-modules finitely presented over $A$ is replaced by a similar category of modules over the crossed product algebra $ \A_q \rtimes W $ associated to the geometric action of $W $ on $ \c^* $. The algebra $ \A_q \rtimes W $ has the canonical presentation $  A \langle \xi, \xi^{-1}, s \rangle $, where the generators
$ \xi $, $s$ and $ a(z) \in A $ are subject to the relations
$$
s \cdot a(z) = a(z^{-1}) \cdot s\ ,\quad s \cdot \xi = \xi^{-1} \cdot s\ ,\quad \xi \cdot a(z) = a(qz) \cdot \xi\ ,\quad s^2 = 1
$$
We let $ e_+ := (1+s)/2 $ denote the symmetrizing idempotent 
in $ \A_q \rtimes W $ and consider the  subalgebra $ e_+ (\A_q \rtimes W) e_+  $ of $ \A_q \rtimes W $ (with identity element $e_+$). This subalgebra can be naturally identified with the invariant subalgebra $ \A_q^W $ of $ \A_q$ via the isomorphism:
$\, \A_q^{W} \xrightarrow{\sim} e_+ (\A_q \rtimes W) e_+ $,$\ a \mapsto e_+ a \,e_+\,$.
With this identification, we can define the additive functor
\begin{equation}
\la{ResW}
 \Mod(\A_q \rtimes W)\,\to\, \Mod(\A_q^W)\ ,\quad M \mapsto e_{+} M\ ,
\end{equation}
that assigns to a $W$-equivariant $\A_q$-module its subspace of $W$-invariant elements viewed as a module over $\A_q^W$. 
The next result is well known for the  algebra  $ \A_q^{\rm alg} := \O_{\rm alg}(\c^*) \rtimes_q \Z $ which is an algebraic (polynomial) analogue\footnote{The algebra $\A_q^{\rm alg}  $ is usually referred to as a quantum Weyl algebra.} of $ \A_q = \O_{\rm an}(\c^*)\rtimes_q \Z $. The analytic case easily reduces to the algebraic one as $ \A_q^{\rm alg}  $ is naturally a subalgebra of $\A_q  $.
\begin{lemma}
 \la{weylalg}
The functor \eqref{ResW} is an equivalence of categories, its inverse being given by
$$
\A_q \otimes_{\A_q^{W}}\,(\mbox{---}):\ \Mod(\A_q^W)\, \to\, \Mod(\A_q \rtimes W)
$$
\end{lemma}
\begin{proof}
Lemma can be restated by saying that the algebra $ \A_q\rtimes W$ is Morita equivalent to $\A_q^W $. To prove this, by standard Morita theory (see \cite[3.5.6]{MR01}), it suffices to
check that the idempotent $ e_+ $ generates the whole $ \A_q \rtimes W $ as its two-sided ideal.
This last condition holds for $\A_q^{\rm alg} \rtimes W $, since $\A_q^{\rm alg} \rtimes W $
is a simple algebra (has no proper two-sided ideals), if $q$ is not a root of unity. But then
it also holds for $ \A_q \rtimes W $, since $ \A_q^{\rm alg} \rtimes W $ is a unital
subalgebra of $ \A_q \rtimes W $ containing $e_+$.
\end{proof}
Now, combining \eqref{funeqW} with Morita equivalence \eqref{ResW}, we get the equivalence
\begin{equation}
\la{funeq2}
\Coh_W(E_q)\,\xrightarrow{\sim}\,\Mod^{\rm f.p.}_{A^W}(\A_q^W)\ ,\quad \cF \mapsto 
H^0_{\rm an}(\c^*, \,\pi^* \cF)^{W}\,,
\end{equation}
that allows us to describe the $W$-equivariant coherent
sheaves on $E_q$ in terms of $\A_q^W$-modules. Recall that
$ \Ell_G^*(F_m) $ is defined to be the subsheaf of $W$-invariant sections of the coherent sheaf
$ \Ell_T^*(F_m)$ (see \eqref{TWG}). In the next theorem, we describe $ \Ell_G^*(F_m) $ explicitly as an 
$\A_q^W$-submodule of $A$, where the action of $ \A_q^W $ on $A$ is obtained by restricting the
natural action of $ \A_q $.
\begin{theorem}
\la{ThElT}
Under the equivalence \eqref{funeq2}, the $W$-equivariant  sheaf $\,\Ell_T^*(F_m)\,$ maps to the $ \A_q^W$-module representing the $G$-equivariant elliptic cohomology of $ F_m:$
\begin{equation}
\la{EllG}
\widetilde{\Ell}_G^*(F_m)\,\cong\,
A^W + \,A^W (\Theta(z) - \Theta(z^{-1}))\,\vartheta(z)^{2m}\ \subseteq \ A\,,
\end{equation}
where $A^W$ is the subspace of $W$-invariant functions in $A = \O_{\rm an}(\c^*)$ and
$ \vartheta(z) \in A[z^{\pm 1/2}] $ is the Jacobi theta function
\begin{equation}
\la{jacth}
\vartheta(z) :=  (z^{1/2}-z^{-1/2})\,\prod_{n>0}(1-q^n z)(1-q^{n}z^{-1})
\end{equation}
\end{theorem}
\begin{proof}
Observe that the $T$-space $ F_m $  comes together with a natural $T$-equivariant map
\begin{equation}
\la{FmT1}
 (G/T)^T \,\into\, (G/T)^T \ast E_{2m}(T)\,\cong\, F_m(G,T)\,,   
\end{equation}
where $(G/T)^T \subset G/T $ is the set of $T$-fixed points in $G/T$ (see \eqref{FmT}). On $T$-equivariant elliptic cohomology, the map \eqref{FmT1} induces an injective map $\,\Ell_T^*(F_m) \into \Ell^*_T [(G/T)^T]\,$, which under the isomorphism \eqref{cohFm} of Theorem~\ref{ThTell}, corresponds to the canonical inclusion
\begin{equation}
\la{injFm}
 \O_{E_q} \times_{\O_{E_q}/{\mathcal J}^{2m+1}}  \O_{E_q} \,\into\, \O_{E_q} \times \O_{E_q} 
\end{equation}
Now, the map \eqref{FmT1} is also equivariant under the action of $W$ which is given on $ (G/T)^T = \bS^0 $  simply by transposition of points.
It follows that \eqref{injFm} is a morphism of $W$-equivariant sheaves on $E_q$ that, under equivalence
\eqref{funeqW}, corresponds to the $W$-equivariant inclusion
$\,A \times_{A/\langle \Theta\rangle^{2m+1} } A \into A \times A \,$, where $W$ acts on $A \times A$ by
$s \cdot (f(z), g(z)) =  (g(z^{-1}), f(z^{-1}))$.
As a $ (\A_q \rtimes W)$-module, the product $ A \times A $ is thus isomorphic to 
$ A[W] := A \otimes \c W $, where the action of $\A_q \rtimes W$ is given by
\begin{eqnarray}\la{relAW}
a \cdot (f(z) \otimes w) &=& 
a(z)f(z) \otimes w\ ,\nonumber\\
\xi \cdot (f(z) \otimes w) &=& f(qz) \otimes w\ ,\\
s \cdot (f(z) \otimes w) &=& f(z^{-1}) \otimes sw\,.\nonumber
\end{eqnarray}
Choosing a basis in $\c W$ consisting of the idempotents $ \{e_+,\, e_{-}\}$, we can describe $ \widetilde{\Ell}_T^*(F_m) $ as the $(\A_q \rtimes W)$-submodule of $A[W]$ 
\begin{equation}
\la{EllW}
 \widetilde{\Ell}_T^* (F_m) \,\cong\, A \,e_{+} + A \,\Theta(z)^{2m+1}\, e_{-}\,, 
\end{equation}
where the isomorphism is explicitly given by $\,(f,g) \mapsto (f+g) e_{+} + (f-g) e_{-} \,$. Now, applying to \eqref{EllW} the restriction functor \eqref{ResW} and using the (obvious) algebraic
identities for theta functions $\vartheta(z)= - z^{-1/2} \Theta(z)$ and $\Theta(z^{-1}) = - z^{-1} \Theta(z)$, we get
\begin{eqnarray*}
  \widetilde{\Ell}_T^*(F_m)^W & \cong & e_+ A \,e_{+} + e_+ A \,\Theta(z)^{2m+1}\, e_{-} \\
  & = & e_+ A^W + e_+ A \,\Theta(z) \vartheta(z)^{2m} e_{-}\\
  & = & e_+ A^W + e_+ A \,\Theta(z) e_{-} \vartheta(z)^{2m} \\
  & = &e_+ A^W + e_+ A \,e_+ \,(\Theta(z)-\Theta(z^{-1}))\, \vartheta(z)^{2m}\\
  & = &e_+ \left(A^W + A^W \,(\Theta(z)-\Theta(z^{-1}))\, \vartheta(z)^{2m}\right)\,,
\end{eqnarray*}
which, with our identifications $ \widetilde{\Ell}_G^*(F_m) = \widetilde{\Ell}_T^*(F_m)^W $ (see \eqref{TWG}) 
and $\, e_+(\A_q \rtimes W)e_+ = \A_q^W $, is precisely the isomorphism \eqref{EllG}.
\end{proof}
\subsection{Elliptic cohomology with twisted coefficients}
\la{S7.3}
The coherent sheaves $ \Ell_T^*(F_m) $ (and a fortiori $\, \Ell_G^*(F_m) $) do not have nontrivial global sections. Indeed, by Theorem~\ref{ThTell}, $\,\Ell_T^*(F_m)$ fits in the short exact sequence in $\Coh(E_q)$:
\begin{equation}
\la{seqs}
0 \to \Ell_T^*(F_m) \,\to\, \O_{E_q} \oplus \O_{E_q}\,  \to\, \O_{E_q}/{\mathcal J}^{2m+1} \to 0
\end{equation}
that shows at once that $ H_{\rm an}^0(E_q, \Ell^*_T(F_m)) \cong \c \,$ for all $ m \ge 0 $. 
With a little more effort, using the long exact cohomology sequence associated to \eqref{seqs}
we can also find that $\, H_{\rm an}^1(E_q, \Ell^*_T(F_m)) \cong \c^{2m+2} $, which ---
as a $W$-module --- admits decomposition
\begin{equation}
\la{Hcoh}
H_{\rm an}^1(E_q, \Ell^*_T(F_m))\,\cong\, \c_+^{\oplus (m+1)}\,\oplus\, \c_{-}^{\oplus (m+1)}\,,
\end{equation}
where `$\c_{+}$' and `$\c_{-}$' denote the trivial and the sign representations of $W$, respectively.

A much richer picture emerges if we twist the elliptic cohomology sheaves $ \Ell_T^*(F_m) $ with 
the Looijenga line bundle $ \cL $ on $E_q$ (see definitions \eqref{ELLL} and \eqref{ELLLW}).
Under the equivalence \eqref{funeq}, this line bundle corresponds to the rank one free $A$-module $ \tilde{\cL} = A\,v $,  where the action of $\A_q$ and  $W$ are determined by the relations $\, \xi\cdot v = q\, z^2\, v$ and $ s \cdot v = v $ ({\it cf.} Example~\ref{Ex222}).
Since  \eqref{funeq} preserves tensor products, the tensor powers $\cL^n = \cL^{\otimes n} $ of $\cL$ in $ \Coh(E_q)$ correspond to the $\A_q$-modules $\widetilde{\cL}^n = A \,v_n$ with
$\xi \cdot v_n = q^n \,z^{2n}\,v_n$ and 
$ s \cdot v_n = v_n $.
By \eqref{H01}, we can then identify the spaces of global sections of these line bundles as
\begin{equation}\la{glsn}
H_{\rm an}^0(E_q,\,\cL^n)\,\cong\, 
\{f(z) \in A\ :\ f(qz) = q^{-n}\,z^{-2n}\,f(z)\}\ ,\quad \forall\,n \ge 0\,.
\end{equation}
Following \cite{Lo77}, we set
\begin{equation}\la{SE0}
S(E) := \bigoplus_{n \ge 0}\, H_{\rm an}^0(E_q,\,\cL^n) \,, 
\end{equation}
which, with identifications \eqref{glsn}, is a graded subalgebra of $A$ stable under the action of $W$. 
To describe this subalgebra we decompose it as the direct sum of $W$-invariants and anti-invariants:
\begin{equation}\la{SE}
 S(E) \,=\,S(E)^W \oplus\, S(E)^{-W}   
\end{equation}
Then, by Looijenga Theorem (see\cite[(3.4)]{Lo77}), we know that $ S(E)^W $ is a free polynomial algebra on $ 2 $ generators, while $ S(E)^{-W} $ is a free module over $S(E)^W$ of rank one.
The generators of $ S(E)^W $ and $S(E)^{-W}$ can be explicitly given in terms of the Jacobi theta function \eqref{jacth}: namely, $S(E)^W$  is generated (as an algebra) by  $ \vartheta^2(z) $ and $\vartheta^2(-z) $, which are both invariant functions in $S(E)$ of degree $1$, while $ S(E)^{-W}$ is generated (as a module) by the function $\vartheta(z^2) $ which is an anti-invariant in $ S(E)$ of degree $2$. 

Now to state our last result in this section we recall the definitions  of equivariant
elliptic cohomology with twisted coefficients: see formulas \eqref{ELLL} and \eqref{ELLLW} (with $ \M_T = E_q$). 
For $ X = G/T $, it is well known that  (see, e.g., \cite{Ga14}):
\begin{equation}
\la{Ell0}
\ELL_G^*(G/T, {\mathscr L}) \cong \ELL_T^*({\rm pt}, \cL) = S(E)
\end{equation}
We extend this result to the quasi-flag manifolds $ F_m = F_m(G,T) $.
\begin{theorem}
\la{EllTw}
The natural maps 
$$
G/T = F_0(G,T) \to F_1(G,T) \to \ldots \to F_{m-1}(G,T) \to F_m(G,T) \to \ldots
$$
induce injective homomorphisms on twisted elliptic cohomology:
$$
\ldots \into \ELL_G^*(F_m, {\mathscr L}) \into \ELL_G^*(F_{m-1}, {\mathscr L}) \into \ldots \into
\ELL_G^*(G/T, \cL)\,.
$$
Under the identification \eqref{Ell0}, the composite map $\ELL_G^*(F_m, {\mathscr L})\into \ELL_G^*(G/T, \cL)$ corresponds to the inclusion $\,S(E)^W \,\oplus \, \vartheta^{2m}(z)\,S(E)^{-W} \into S(E) \,$, so that
\begin{equation}
\la{ELLG}
\ELL_G^*(F_m, {\mathscr L}) \,\cong\, 
S(E)^W \,\oplus \, \vartheta^{2m}(z)\,S(E)^{-W}\,,
\end{equation}
where $ \,S(E)^W = \c[\vartheta^2(z), \vartheta^2(-z)] \,$ and $\, S(E)^{-W} = \c[\vartheta^2(z), \vartheta^2(-z)]\,\vartheta(z^2)\,$.
\end{theorem}
\begin{proof}
We use the description of $ \widetilde{\Ell}_T^*(F_m) $ given in the proof of Theorem~\ref{ThElT}: namely, $\, \widetilde{\Ell}_T^*(F_m) = A\,e_+ \,+\,A \,\Theta^{2m+1}\, e_{-} \,$ as an $ (\A_q\rtimes W) $-submodule of  $ A[W] = A\,e_+ + A\,e_{-}  \,$. Under the equivalence \eqref{funeq}, the twisted sheaves $\,\Ell_T^*(F_m) \otimes \cL^n \,$ can then be described by
\begin{equation}\la{Elln}
\widetilde{\Ell}_T^*(F_m) \otimes_A \widetilde{\cL}^n \,=\, A\, v_n \otimes e_+ \,\oplus\, A \,\Theta^{2m+1}\, v_n\otimes e_{-} 
\end{equation}
and we can compute their global sections using formula \eqref{H01}: 
\begin{eqnarray*}\la{HH0}
H^0_{\rm an}(E_q,\, \Ell_T^*(F_m) \otimes \cL^n) &\cong& \Ker(\xi-\id:\,\widetilde{\Ell}_T^*(F_m) \otimes_A \widetilde{\cL}^n\,)\\
& \cong & \Ker(\xi- \id:\, A v_n \otimes e_+) \oplus \Ker(\xi-\id:\, A \Theta^{2m+1} v_n \otimes e_-)\\
&\cong& H^0_{\rm an}(E_q, \cL^n)\,e_+ \,\oplus\, 
(H^0_{\rm an}(E_q, \cL^n)\,|\, \Theta^{2m+1}) \, e_{-}\,,
\end{eqnarray*}
where $ (H^0_{\rm an}(E_q, \cL^n)\,|\, \Theta^{2m+1}) $ denotes the subspace of all sections in
$ H^0_{\rm an}(E_q, \cL^n) $ that are divisible by $ \Theta^{2m+1} $  under the identification \eqref{glsn}. Summing up over all $ n\ge 0 $, we find 
$$ 
\ELL_T^*(F_m, \cL) \,\cong\, S(E)\,e_+ \,\oplus\, (S(E)\,|\,\Theta^{2m+1})\, e_{-}\,,
$$
where $S(E)$ is the Looijenga ring \eqref{SE0}. To compute $(S(E)\,|\,\Theta^{2m+1})$ we observe that
an element of $ S(E) $ is divisible by $\Theta^{2m+1}$ in $A$ if and only if its invariant
and anti-invariant parts in $S(E)^W$ and $S(E)^{-W}$ are both divisible by
$\Theta^{2m+1}$. Now, for $ f(z) \in S(E)^W$, $\Theta^{2m+1}(z)$ divides $f(z)$ if and only if $\vartheta^{2m+2}(z)$ divides $f(z)$, while for $ f(z) \in S(E)^{-W}$, $\Theta^{2m+1}(z)$ divides $f(z)$ if and only if $\vartheta^{2m}(z)\,\vartheta(z^2)$ divides $f(z)$. Thus
\begin{equation}
\la{ELLTT}
\ELL_T^*(F_m, \cL) \,\cong\, S(E)\,e_+ \,\oplus\,(\vartheta^{2m+2}(z)\, S(E)^W +  \vartheta^{2m}(z)\,S(E)^{-W})\,e_{-} 
\end{equation}
Now, applying to \eqref{ELLTT} the restriction functor \eqref{ResW}, we get 
\begin{eqnarray*}
\ELL_T^*(F_m, \cL)^W & \cong & 
e_{+}\,S(E)\,e_+ \,\oplus \,e_+\,\left(\vartheta^{2m+2}(z)\,S(E)^W  + \vartheta^{2m}(z)\,S(E)^{-W}\right) e_{-}\\
& \cong & 
S(E)^W \, \oplus \, \vartheta^{2m}(z)\,S(E)^{-W}
\end{eqnarray*}
which gives  \eqref{ELLG} since $ \ELL_T^*(F_m, \cL)^W = \ELL_G^*(F_m, \cL)$. To complete the proof it suffices to note that the map of spaces $ G/T  \to F_m $ induces the natural inclusion
$$
S(E)\,e_+ \,\oplus\,(\vartheta^{2m+2}(z)\, S(E)^W +  \vartheta^{2m}(z)\,S(E)^{-W})\,e_{-} 
 \,\into\,
S(E)\,e_+\,\oplus\, (\vartheta^{2}(z)\,S(E)^W  + S(E)^{-W})\,e_{-} 
$$
as a map representing  $\ELL_T^*(F_m, {\mathscr L}) \to \ELL_T^*(G/T, \cL)$ under the isomorphism \eqref{ELLTT}.
When restricted to $W$-invariants this yields the inclusion
$$
S(E)^W  \oplus \vartheta^{2m}(z)\,S(E)^{-W} \into S(E)^W \oplus S(E)^{-W} = S(E) 
$$
that represents $\ELL_G^*(F_m, {\mathscr L})\into \ELL_G^*(G/T, \cL)$.
\end{proof}
\begin{remark}
\la{RemEll}
The above calculation of elliptic cohomology suggests a natural algebraic definition of
quasi-invariants in the elliptic case ({\it cf.} \eqref{ELLG}). This differs, however, 
from the definition of elliptic quasi-invariants that has already been used in the literature (see, e.g., the beautiful work of O. Chalykh on Macdonald's conjectures \cite{Ch02}). The 
difference seems to be an instance of `elliptic-elliptic' duality studied in the theory of integrable systems (see, e.g., \cite{KS19}).
\end{remark}

\section{Topological Gorenstein duality}
\la{S8}
The realization of algebras of quasi-invariants raises natural questions about homotopy-theoretic analogues (refinements) of basic properties and structures associated with these algebras. In this section, we make first steps in this direction by showing that the spaces of quasi-invariants $ X_m(G,T) $  satisfy Gorenstein duality in the sense of stable homotopy theory. Our main result --- Theorem~\ref{ThGor} --- should be viewed as a 
topological analogue of Theorem~\ref{TGoren} on Gorensteinness of rings of quasi-invariants.
For reader's convenience, we begin this section with a brief introduction to spectral (`higher') algebra, where we recall basic definitions concerning  duality and regularity properties of commutative ring spectra. Our main reference is the beautiful paper \cite{DGI06}, where many concepts that we use were originally introduced (we also recommend the lecture notes \cite{G18} which supplement \cite{DGI06} with many interesting examples). The readers familiar with this material may skip Section~\ref{S7.1} and go directly to Sections~\ref{S7.2} and \ref{S7.3}.

\subsection{Spectral algebra}
\la{S7.1}
Following \cite{DGI06}, we choose the category of {\it symmetric spectra} as our basic model for 
stable homotopy theory. This category can be succinctly described as the (stable model) category $ \Mod_{\bS} $ of modules\footnote{Unfortunately, the term `$\bS$-module' in reference to spectra is very ambiguous: apart from symmetric, other popular types of spectra (e.g., orthogonal and EKMM ones) are also $\bS$-modules. A nice recent survey comparing properties and applications of different types of spectra can be found in \cite{D22}.} over the symmetric sphere spectrum $ \bS = ((S^1)^{\wedge n})_{n \ge 0}$ (see \cite{HHS00}). The category $\Mod_{\bS}$ is equipped with a symmetric monoidal product which is denoted  as a smash $ A \wedge B $ or tensor product $ A \otimes_{\bS} B $ (depending on the context). A {\it ring spectrum} is then, by definition, an $\bS$-algebra, i.e. an $\bS$-module  $ R $ given with two structure maps $ \bS \to R $ and $ R \wedge R \to R $ satisfying the usual unitality and associativity properties.  We denote the category of ring spectra by $\Alg_{\bS}$. There is
a natural (Eilenberg-MacLane) functor $\, H: \Alg_{\Z} \to \Alg_{\bS} $, $\,k \mapsto Hk \,$
that embeds the category $ \Alg_{\Z}$ of usual (discrete) associative rings into $\bS$-algebras by identifying a ring $k$
with its symmetric Eilenberg-MacLane spectrum $ Hk= (K(k,n))_{n \ge 0} $ (see \cite[1.2.5]{HHS00}).
The category $\,\Alg_{\bS}\,$ can be thought of as a homotopical refinement (`thickening') of  $ \Alg_{\Z} $ in the same way as the category $ \Mod_{\bS}$ is a homotopical refinement of 
the category $ \Mod_{\Z} $ of (discrete) abelian groups. 

For a ring spectrum $R \in \Alg_{\bS} $, we let $ \Mod_{R}$ denote the category of left module spectra over $R$. This is a stable model category enriched over $ \Mod_{\bS} $. The latter means that, for two $R$-modules $A$ and $B$, there is a mapping spectrum of $R$-module maps $ A \to B $ that we denote $ \Map_R(A,B) $. Moreover, if $ A $ is a right $R$-module and $B$ is a left $R$-module, there is an
associated smash product $ A \wedge_R B $ defined as the (homotopy) coequalizer $\, A\wedge R \wedge B 
\rightrightarrows A \wedge B \,$ of structure maps $ A \wedge R \to A $ and $ R \wedge B \to B $
in $ \Mod_{\bS}$. Note that both $ \Map_R(A,B) $ and $ A \wedge_R B $ are understood as `derived' objects
in the sense that their first arguments are (replaced by) cofibrant objects in $\Mod_R $. In particular,
if $A$ and $B$ are usual (discrete) modules over a usual (discrete) ring $R$, viewed as symmetric 
spectra via the Eilenberg-MacLane functor, then $\,\pi_i\, \Map_R(A,B) \cong {\rm Ext}_{R}^{-i}(A,B)\,$ and $\, 
\pi_i(A \wedge_R B) \cong {\rm Tor}^R_{i}(A,B)\,$, where $\, \pi_i \,$ stand for the (stable) homotopy groups of spectra. If $R$ is a commutative ring spectrum, then both $ \Map_R(A,B) $ and $ A \wedge_R B $ are naturally $R$-modules, i.e. objects in $ \Mod_R $.

Next, we recall that a subcategory of a (stable) model category $\M$ is called {\it thick} if it is closed under weak equivalences,  cofibration sequences (distinguished triangles) and retracts in $\M$. Further, a subcategory of $ \M $ is called {\it localizing} if it is thick and, in addition, closed under arbitrary coproducts (and hence  homotopy colimits) in $\M$. Given two objects $A$ and $B$ in $ \M$, we say that $B$ is {\it built} from $A$ if $B$ belongs to the localizing subcategory of $ \M $ generated by $A$, and $B$ is {\it finitely built} from $A$ if it belongs to the thick subcategory generated by $A$ (\cite[3.15]{DGI06}). Now, if $ \M = \Mod_R $, an $R$-module $A $ is called {\it small} if it is finitely built from $R$ in $ \Mod_R $. 
This agrees with the usual definition of small (compact) objects in $\Mod_R$: an $R$-module
$A$ is small iff $ \Map_R(A,\,-\,)$ commutes with arbitrary coproducts.

The notion of a localizing subcategory is closely related to that of cellularization. For a fixed object $A \in \Mod_R $, we say that a morphism $ f: M \to N $ in $ \Mod_R $   is an {\it $A$-cellular equivalence} if $f$ induces a (weak) equivalence on mapping spectra:
\begin{equation*}
\la{Keq}
f_*:\, \Map_R(A, M)\,\xrightarrow{\sim}\, \Map_R(A, N)
\end{equation*}
Note that every equivalence in $\Mod_R$ is automatically an $A$-cellular equivalence, but the class of
$A$-cellular equivalences is, in general, larger than that of weak equivalences.
Now, an $R$-module $B$ is called {\it $A$-cellular} if any $A$-cellular equivalence $ f: M \to N $ induces an equivalence $\,\Map_R(B, M) \xrightarrow{\sim} \Map_R(B, N)\,$. 
This terminology is motivated by the fact that the $A$-cellular modules are precisely those objects of $\Mod_R$ that are built from $A$ (see \cite[5.1.15]{Hir03}). Moreover, for any $R$-module $B$, there is a $A$-cellular module $ \Cell^R_A(B) $  together with a $A$-equivalence in $ \Mod_R$:
$$ 
\Cell^R_A(B) \to B 
$$ 
called an {\it $A$-cellular approximation\footnote{Cellularization is an example of a general model-categorical construction called right Bousfield localization (colocalization) with respect to an object $A$. In this language, $A$-cellular equivalences are called $A$-colocal equivalences, $A$-cellular objects are $A$-colocal objects,
and $A$-cellular approximations are functorial cofibrant
replacements in the $A$-colocal model structure on $ \Mod_R $ (see \cite[3.1.19]{Hir03}).}} of $B$. Such an approximation is determined by $B$ uniquely up to  canonical equivalence; we will use the simpler notation $ \Cell_A(B) $ for $ \Cell^R_A(B) $ when the ring spectrum $R$ is understood.

The above categorical notions can be used to impose some finiteness and regularity conditions on commutative ring spectra. First, we say that a morphism of commutative ring spectra $ R \to k $ is called {\it regular} if $k$ is small as an $R$-module. This definition is motivated by the fact that, in classical commutative algebra, a local Noetherian ring $ R $ with residue field $ k = R/\mathfrak{m}$ is regular iff $k$ has a finite length resolution by f.g. free $R$-modules (see \cite{Se00}); for the associated Eilenberg-MacLane spectra, the latter means that $Hk$ is finitely built from $HR$. A more flexible and technically useful condition is obtained by weakening the regularity assumption on $ R \to k$ in the following way.
%
\begin{defi}[\cite{DGI06}, 4.6]
\la{preg}
A morphism of commutative ring spectra $ R \to k $ is called {\it  proxy-regular} if $k$ is a {\it proxy-small} $R$-module via $ R \to k $ in the sense that there is a small $R$-module $K$ that builds $k$ and is finitely built from $k$ in $\Mod_R$. Note that if $K=k$, then $ R\to k $ is {\it regular}. On the other extreme, if $ K = R $ then $ R \to k $ is called {\it cosmall}.
\end{defi}

Let $ E := \Map_R(k,k)$ denote the endomorphism ring spectrum of $k$ viewed as a left $R$-module via the morphism $ R \to k $. There is a standard Quillen adjunction relating right $E$-modules to left $R$-modules:
\begin{equation}\la{TEadj}
(\,\mbox{--}\,) \wedge_E k\, :\, \Mod_{E^{\rm op}}\,\leftrightarrows\,\Mod_R\,:\, \Map_R(k,\,\mbox{--}\,)
\end{equation}

If $R \to k $ is regular, the functors  \eqref{TEadj} induce an 
equivalence between $ \Ho(\Mod_{E^{\rm op}})$ and the full subcategory of $ \Ho(\Mod_R)$ consisting of $k$-cellular $R$-modules (see \cite[Theorem 6.1]{G18}).

If $ R \to k $ is proxy-regular, \eqref{TEadj} does not induce an 
equivalence in general, but the counit of this adjunction still provides a $k$-cellular approximation for modules in $ \Mod_R$ (see \cite[Lemma 6.3]{G18}): 
\begin{equation}
\la{kcell}
\Cell_k(M) \simeq \Map_R(k, M) \wedge_E k    
\end{equation}
Moreover, for all $R$-modules $ M $, there is a natural equivalence (see \cite[Lemma 6.6]{G18})
\begin{equation}
\la{kcell1}
\Cell_k(M) \simeq \Cell_k(R) \wedge_R M    
\end{equation}
Formula \eqref{kcell} shows that when $ R\to k $ is proxy-regular, the $k$-cellular approximation $ \Cell_k(M)$ is functorial and effectively constructible in $  \Mod_R $ ({\it cf.} \cite[Definition 4.3]{DGI06}).

Now, we come to the key definition of a Gorenstein ring spectrum that we state under
the regularity assumptions of Definition~\ref{preg} (which is a slightly less general form than in \cite{DGI06}):
\begin{defi}[{\it cf.} \cite{DGI06}, 8.1 and 8.4]
\la{defGor}
A morphism of commutative ring spectra $R \to k $ is called {\it Gorenstein of shift} $a \in \Z$, if $R \to k $ is proxy-regular and there is an equivalence of $k$-modules
\begin{equation}
\la{Gor1}
\Map_{R}(k, R) \,\simeq\, \Sigma^a k    
\end{equation}
where $\Sigma$ denotes the suspension functor on $ \Mod_k $.
\end{defi}

We will be mostly interested in ring spectra $R$ that are {\it augmented $k$-algebras} over a field $k$. For such algebras, we will always assume that $ R \to k $ is the given augmentation morphism on $R$, and we will simply say that $R$ is Gorenstein if so is $ R \to k $. The Gorenstein condition \eqref{Gor1} can be slightly refined in this case. Note that, if $R$ is a $k$-algebra, using the $k$-module structure on $R$, we can rewrite \eqref{Gor1} in the form
\begin{equation}
\la{Gor11}
\Map_{R}(k, R) \,\simeq\, \Sigma^a\, \Map_{R}(k,\, \Map_k(R,k))    
\end{equation}
Both sides of \eqref{Gor11} have natural right module structures over the endomorphism
ring $ E = \Map_R(k,k) $ but, in general, these module structures need not to agree under the equivalence \eqref{Gor11}. Following \cite{DGI06} (see also \cite[Section 18.2]{G18})), we say that an augmented $k$-algebra $R$ is {\it orientable Gorenstein} if \eqref{Gor11} is an equivalence of right $E$-modules. 

If $R$ is a local Noetherian ring of Krull dimension $d$ with residue field $ k = R/\mathfrak{m}$, then $R$ is Gorenstein (in the sense of commutative algebra) iff
\begin{equation}
\la{extd} 
{\rm Ext}^i_R(k,R) \cong \left\{
\begin{array}{cc}
k     & i= d  \\
0     & {\rm otherwise}
\end{array}
\right.
\end{equation}
The isomorphism \eqref{extd} can be written as an equivalence $\, {\rm RHom}_R(k,R) \simeq \Sigma^d k \,$ in the derived category $ \D(R)$ and thus corresponds to the Gorenstein condition \eqref{Gor1} of Definition~\ref{defGor}. In classical commutative algebra, there is another well-known characterization of Gorenstein rings in terms of local cohomology:  
\begin{equation}
\la{lcoh} 
H_{\mathfrak{m}}^i(R) \cong \left\{
\begin{array}{cc}
\Hom_k(R,k)     & i=d  \\
0     & {\rm otherwise}
\end{array}
\right.
\end{equation}
which can be viewed as a special case of Grothendieck's local duality theorem.
The following definition is a topological analogue of \eqref{lcoh}.
\begin{defi}
\la{Gordual}
An augmented $k$-algebra $R$ satisfies {\it Gorenstein duality of shift} $a$ if there is an equivalence of 
$R$-modules
\begin{equation}
\la{Gor2}
\Cell_k(R) \,\simeq\, \Sigma^a\,\Map_k(R,k)
\end{equation}
\end{defi}
While the algebraic conditions \eqref{extd} and \eqref{lcoh} are known to be equivalent, their topological analogues \eqref{Gor1} and \eqref{Gor2} are, in general, not (see, e.g., \cite[Remark 2.11]{BCHV21} for a counterexample). This necessitates two separate definitions for Gorensteinness of commutative ring spectra.

\vspace{1ex}

%
%

The last property of ring spectra that we want to review is concerned with double centralizers. Recall, for a morphism $ R \to k $, the {\it double centralizer of $R$} is defined to be $ \hat{R} := \Map_E(k,k) $, where
$ E = \Map_R(k,k) $ is the endomorphism spectrum of $k$ in $\Mod_R$. The left multiplication on $k$ gives a morphism of ring spectra $\, R \to \hat{R} \,$, and following \cite{DGI06}, we say 
\begin{defi}[\cite{DGI06}, 4.16]
\la{dc}
$ R\to k $ is {\it dc-complete} if $ R \xrightarrow{\sim} \hat{R} $ is an equivalence in $ \Alg_{\bS}$.
\end{defi}
Note that, in algebra, a surjective homomorphism $ R \to k $ from a Noetherian commutative ring $R$ to a field $k$ is dc-complete iff $ R \cong \hat{R}_I $, where $ \hat{R}_I := \varprojlim R/I^n $ is the $I$-adic completion of $R$ with respect to the ideal $ I = \Ker(R \to k)$. This motivates the above terminology. One can show that if 
$ R \to k $ is dc-complete, the regularity properties of the ring spectra $R$ and $ E $ are strongly connected
(see, e.g., \cite[Proposition 4.17]{DGI06}).

\subsection{Spectra of quasi-invariants}
\la{S7.2}
It is well known that, if $X$ is a pointed connected topological space, the singular cochain complex $ C^*(X, \Q) $, computing cohomology of $X$ with coefficients in $\Q$, admits a commutative DG algebra model\footnote{Such a model can be constructed in a functorial way, using, for example, piecewise polynomial differential forms on $X$ defined over $\Q$ (see \cite{BG76}).}. When $\Q$ is replaced by an arbitrary field $k$, this last fact is no longer true:  in general, the cochain complex $ C^*(X,k) $ is not quasi-isomorphic to any commutative DG algebra over $k$ if ${\rm char}(k) \not= 0 $. A natural way to remedy this problem is to use commutative ring spectra -- instead of DGAs -- as models for $C^*(X,k)$. Specifically, for any commutative ring $k$, the {\it cochain  spectrum} of the space $X$ with coefficients in $k$ is defined by
({\it cf.} \cite{Man01})
\begin{equation}
\la{coch}
C^*(X, k) := \Map_{\bS}(\Sigma^{\infty}X_+, \,Hk)    
\end{equation}
where $ \Sigma^\infty X_+$ is the suspension spectrum associated to $X$, $Hk$ is the Eilenberg-MacLane spectrum of $k$, and $\Map_{\bS}$ denotes the mapping spectrum in the category of (symmetric) spectra. By definition, \eqref{coch} is a commutative ring spectrum with multiplication induced by the multiplication map on $Hk$ and the diagonal map on $X$.
In addition, following \cite{DGI06}, we introduce the {\it chain spectrum} of $X$:
\begin{equation}
\la{chain}
C_*(\Omega X, k) := Hk \wedge \Sigma^{\infty}(\Omega X)_+
\end{equation}
which is a noncommutative ring spectrum that models the singular chain complex of the based loop space of $X$. Both $C^*(X,k)$ and $ C_*(\Omega X,k) $ are augmented $k$-algebras, with augmentation on $C^*(X,k) $ induced by the basepoint inclusion $ {\rm pt} \to X $ and on $ C_*(\Omega X, k)  $ by the trivial map $ \Omega X \to {\rm pt} $. For all $ i \in \Z\,$, there are natural isomorphisms
\begin{equation}
\la{isococh}
\pi_i\,[C^*(X,k)] \cong H^{-i}(X,k)\ ,\quad
\pi_i\,[C_*(\Omega X, k)] \cong H_i(\Omega X, k)
\end{equation}
which show that $ C^*(X,k) $ and $C_*(\Omega X, k) $ are coconnective and connective spectra, respectively.

\vspace*{1ex}

We are now in position to state and prove the main theorem of this section.
\begin{theorem}
\la{ThGor}
Let $ X_m = X_m(G,T) $ be the space of $m$-quasi-invariants associated to $ G = SU(2) $.
Let $ R_m := C^*(X_m, k) $ and $ E_m := C_*(\Omega X_m, k) $ denote the cochain and chain spectra
of $X_m$ with coefficients in an arbitrary field $k$.
Then, for any $ m \ge 0 $,

$(1)$ $\, R_m $ and $\, E_m $ are proxy-regular $($Definition~\ref{preg}$)$ and dc-complete $($Definition~\ref{dc}$)$ with
$$
\Map_{R_m}(k,k) \simeq E_m\quad \mbox{and}\quad\Map_{E_m}(k,k) \simeq R_m
$$

$(2)$ $\, R_m $ is orientable Gorenstein of shift $ a= 1- 4 m \,$ $($Definition~\ref{defGor}$)$

$(3)$ $\,R_m$ satisfies Gorenstein duality of shift $ a = 1-4m\,$ $($Definition~\ref{Gordual}$)$ 
\end{theorem}
\begin{proof}
$(1)$ We start with Borel fibration sequence that comes from the Ganea construction of spaces of quasi-invariants (see \eqref{borfib}):
\begin{equation}
\la{Bfib}
F_m(G,T) \to X_m(G,T) \xrightarrow{p_m} BG     
\end{equation}
To simplify the notation we set
$$
Q_m := C^*(F_m, k)\ ,\quad 
R_m := C^*(X_m, k)\ , \quad 
S := C^*(BG, k)\ .
$$
Since $\,F_m\,$ is a finite connected complex (see \eqref{eq6}), by \cite[Prop.~5.3]{DGI06}, the augmentation morphism $ Q_m \to k $ is cosmall (see Definition~\ref{preg}). Since $G$ is connected,
the classifying space $ BG $ is simply-connected; moreover, the cohomology of $BG$ is free, of finite type over $\Z$, and hence, a fortiori, over any field $k$ (see, e.g., \cite[III.3.17]{MT78}). Therefore, in terminology of \cite[Sect. 4.22]{DGI06}, the pair $(BG, k)$ is of Eilenberg-Moore type. Since $ H_*(\Omega BG, k) \cong 
H_*(G, k) $ is  finite-dimensional over $k$, it follows from \cite[5.5(3)]{DGI06} that $ S \to k $ is a 
regular morphism, i.e. $k$ is small as an $S$-module. Next, since $(BG, k)$ is of Eilenberg-Moore type,
the fibration sequence \eqref{Bfib} gives an equivalence of cochain spectra (see, e.g., \cite[Lemma 3.7]{BCHV21})
\begin{equation}
\la{wedgep}
Q_m \simeq k \wedge_S R_m 
\end{equation}
Now, by \cite[Prop.~4.18(1)]{DGI06}, we conclude from \eqref{wedgep} together with our earlier observations 
that $\,S \to k $ is regular and $ Q_m \to k $ is cosmall that $ R_m \to k $ is proxy-regular.
To complete the proof of part $(1) $ it remains to note that the pair $(X_m, k)$ is of Eilenberg-Moore
type for any field $k$. Indeed, from the fibration sequence \eqref{Bfib} it follows that  
$X_m$ is simply-connected (since so are $F_m$ and $BG$); on the other hand, from the homotopy cofibration sequences (see \eqref{cfibm})
$$ 
F_m \to X_m \xrightarrow{\pi_m} X_{m+1} 
$$
it follows (by induction) that 
$X_m$ is of finite type over $k$ for any $m \ge 0$. By construction of the Eilenberg-Moore
spectral sequence, for $ E_m = C_*(\Omega X_m, k)$, we have $ E_m \simeq \Map_{R_m}(k,k) $, while the
equivalence $ R_m \simeq \Map_{E_m}(k,k) $ holds in general (see remarks in \cite[Sect. 4.22]{DGI06}).
It follows that the augmented $k$-algebras $ R_m $ and $ E_m $ are both dc-complete, and then,
by \cite[Prop. 4.17]{DGI06}, $ E_m $ is proxy-regular (since so is $R_m$).

$(2)$ By the proof of Theorem~\ref{MTh1}, we know that \eqref{Bfib} is a sphere fibration with $ F_m \simeq \bS^{4m+2} $. Hence, $F_m$ is a Poincar\'e duality space of dimension $ 4m+2 $, then its cochain
spectrum $ Q_m = C^*(F_m, k) $ satisfies Poincar\'e duality of dimension $ a = - 4m - 2 $ (in the sense
of \cite[8.11]{DGI06}). Since $ Q_m $ is cosmall, by \cite[Prop. 8.12]{DGI06},  we conclude  that
$ Q_m $ is Gorenstein of shift $ a = -4m-2$. Further, by \cite[10.2]{DGI06}, we also know that
$ S = C^*(BG,k)$ is Gorenstein of dimension $ a = \dim(G) = 3 $. 

Now, consider the morphism of cochain spectra $\, p_m^*: \, S \to R_m \,$ induced by the 
whisker map $ p_m: X_m \to BG $ in \eqref{Bfib}. We claim that $R_m$ is finitely built from $S$ via $ p_m^* $.
To see this denote by $ \E := C_*(\Omega BG, k) \cong C_*(G,k) $ the chain spectrum of $BG$. Since
$G$ is simply-connected, $\E$ is a connective $k$-algebra with $ \pi_0(\E) \cong k[\pi_1(G)] = k $
(see \eqref{isococh}). Since $S$ is of Eilenberg-Moore type, there is an equivalence $\, 
S \simeq \Map_{\E}(k,k) \,$. Furthermore, if we set $ M_m := C_*(F_m, k) $, the action of $G$ on $F_m$ 
induces a left $ \E$-module structure on $ M_m $, and by a standard Eilenberg-Moore spectral sequence argument
there is an equivalence $\,R_m \simeq \Map_{\E}(M_m, k)\,$. Since $ \pi_*(M_m) \cong H_*(F_m, k) $ is 
finite-dimensional over $k$, the $ \E$-module $M_m$ is finitely built from $k$. Now, Proposition~3.18 of \cite{DGI06} implies that $R_m \simeq \Map_{\E}(M_m, k) $ is finitely built from $ S \simeq \Map_{\E}(k,k) $ as we claimed. Since $ R_m $ is proxy-regular and both $ S $ and $ Q_m $ are Gorenstein, it follows from
\cite[Prop. 8.10]{DGI06} that $ R_m $ is Gorenstein as well. The Gorenstein shift of $R_m$ can be 
computed from the following equivalence of $k$-modules induced by \eqref{wedgep} (see \cite[Prop. 8.6]{DGI06}):
\begin{eqnarray*}
\Map_{R_m}(k, \,R_m) &\simeq & \Map_{Q_m}(k, \,\Map_{S}(k, S) \wedge_{k} Q_m)\\
&\simeq& \Map_{Q_m}(k,\, (\Sigma^3 k) \wedge_{k} Q_m)\\
&\simeq& \Sigma^3\, \Map_{Q_m}(k, \,Q_m)\\
&\simeq& \Sigma^3 (\Sigma^{-4m-2} k) \\
&\simeq& \Sigma^{1-4m} k
\end{eqnarray*}
To complete part $(2)$ it remains to note that, for a simply-connected space $X$ of finite type
over $k$, the cochain spectrum $ C^*(X, k) $ is automatically orientable Gorenstein whenever it is 
Gorenstein. This follows from the fact that, under the above assumptions, $k$ carries a unique action of $ E = \Map_{C^*(X,k)}(k,k) \simeq C_*(\Omega X, k)$ (see \cite[Sect. 18.3]{G18} and also the proof of \cite[Lemma 3.8]{BCHV21}).

$(3)$ follows from $(2)$ by a standard argument. If an augmented $k$-algebra $R$ is orientable Gorenstein of shift $a$, then
\begin{equation}
\la{Gor21}
\Cell_k(R)\, \simeq \,\Map_R(k, R) \wedge_E k\, \simeq\, \Sigma^a\,\Map_R(k,\,\Map_k(R,k)) \wedge_E k 
\,\simeq \,\Sigma^a\,\Cell_k[\Map_k(R,k)]
\end{equation}
where the first and the last equivalences are given by \eqref{kcell} and the one in the middle is induced by
\eqref{Gor11}. For $\, R = C^*(X, k) \,$ with $ \pi_0(R) \cong H^0(X,k) \cong k $, we have $ \pi_i \Map_k(R,k) = 0 $ for $ i \ll 0 $. By \cite[Remark 3.17]{DGI06}, the $R$-module $ \Map_k(R,k) $ is then 
built from $k$ and therefore $k$-cellular in $ \Mod_R $. Condition~\ref{Gor2} thus follows from \eqref{Gor21}.
This completes the proof of the theorem.
\end{proof}
\subsection{Generalized spaces of quasi-invariants}
\la{S7.3}
It is natural to ask whether the result of Theorem~\ref{ThGor}, i.e. the topological Gorenstein property, 
holds for generalized (`fake') spaces of quasi-invariants introduced in Section~\ref{S5}.  
In view of Corollary~\ref{Cor5.8}, the answer is obviously affirmative when $k$ is a field of 
characteristic $0$. The next theorem shows that this is also true when $ k = \F_p $. We keep the notation $ G = SU(2) $ and $ T = U(1) $; however, as in Section~\ref{S5}, we do not identify $T$ as a maximal torus in $G$.
\begin{theorem}
 \la{ThGor1}
Let $B$ be a space in the genus of $ BG $ that admits an essential map from $BT$, and let $ X_m = X_m(\Omega B,T) $ be the space of $m$-quasi-invariants associated to $B$. Then, for any prime $p$, the morphism
$ C^*(X_m, \F_p) \to \F_p $ is Gorenstein of shift $ a = 1 - 4m $.
\end{theorem}
\begin{proof}
We give the part of the proof that differs from that of Theorem~\ref{ThGor}. First, observe that, for any space $B$ in the genus of $BG$, we have equivalences of cochain spectra
$$
C^*(B,\, \F_p)\, \simeq\, C^*(B^{\wedge}_p,\, \F_p)
\,\simeq\, C^*((BG)^{\wedge}_p,\, \F_p) \,\simeq\,
C^*(BG,\,\F_p)\ ,
$$
where $ (\,-\,)^{\wedge}_p $ denotes the $ \F_p$-completion functor on pointed spaces. This follows from the fact that both $B$ and $BG$ are $\F_p$-good spaces (in the sense of \cite{Bou75}) and $ B^{\wedge}_p \simeq (BG)^{\wedge}_p $ for any prime $p$. The above equivalences are compatible with augmentation; hence, by \cite[10.2]{DGI06}, we conclude that $ C^*(B, \F_p) \to \F_p $ is a regular map, Gorenstein of shift $ \dim(G) = 3 $. 

Now, assume that $B$ satisfies the conditions of Theorem~\ref{esmap}. Let $ F = F(\Omega B, T)$ denote the homotopy fibre of the maximal essential map $ p_B: BT \to B $. Recall that this last space is not equivalent to a finite CW complex (unless $ B \simeq BG $), and hence its cochain spectrum $C^*(F, \F_p) $ need not be cosmall (as in the case of $BG$). Nevertheless, we claim that $C^*(F, \F_p) \to \F_p $ is always proxy-regular and satisfies the Gorenstein property of shift $ (-2)$. To see this consider the homotopy fibration sequence $\,\Omega B \to F \to BT \,$ associated to the map $ p_B: BT \to B $. 
Since $ BT \simeq \C\bP^\infty $ is of Eilenberg-Moore type (see \cite[4.22]{DGI06}), we have
$$
C^*(\Omega B, \F_p) \simeq C^*(F, \F_p) \wedge_{C^*(BT, \,\F_p)} \F_p 
$$
In view of the fact that $ \Omega B \simeq \bS^3 $, the map $C^*(\Omega B, \F_p)\to \F_p$ is cosmall, and
hence, by \cite[Prop. 4.18]{DGI06},   $C^*(F, \F_p) \to \F_p $ is proxy-regular. Furthermore, since  $\F_p $ is small over $ C^*(BT) = C^*(BT, \F_p)$, we have a natural equivalence of $C^*(BT)$-modules
$$
\Map_{C^*(BT)}(\F_p,\,C^*(BT))\,\wedge_{C^*(BT)}\,C^*(F)\,\xrightarrow{\sim}\,
\Map_{C^*(BT)}(\F_p,\,C^*(F))\,,
$$
which, by the proof of \cite[Prop. 8.6]{DGI06}, implies that $ C^*(F, \F_p) \to \F_p$ is Gorenstein of shift
$a = 1 + (-3) = -2$. 

The rest of the proof is parallel to that of Theorem~\ref{ThGor}. In brief, by Theorem~\ref{GThm}, 
the fibre of the $m$-th Ganea fibration $ F_m \to X_m \to B $ defining the space $ X_m = X_m(\Omega B, T)$
has the homotopy type of $\, \Sigma^{4m} F $. Hence its cochain spectrum $ C^*(F_m, \F_p)  $ is
Gorenstain of shift $ a = - 2 - 4m$. By induction, each space $X_m$ is of finite type over $\F_p$.
Since $ C^*(B, \F_p) \to \F_p $ is a regular Gorenstein map of shift $ 3 $, it follows from the above fibration sequence that $C^*(X_m, \F_p) \to \F_p $ is Gorenstein of shift $ a = -2 -4m + 3 = 1-4m$. 
 \end{proof}
\begin{remark}
\la{shiftpar}
We point out that the topological Gorenstein shifts $ a $ of Theorem~\ref{ThGor} and Theorem~\ref{ThGor1} agree with the algebraic one of Theorem~\ref{TGoren}: to see this it suffices to change the standard polynomial grading on $ Q_m(W) $ to the cohomological one (by `doubling' degrees of the generators).
\end{remark}

\appendix

\section{Arnold-Maxwell Theorem and Quasi-invariants}
\la{ArMapCon}
\begin{center}
\author{M. V. Feigin and K. E. Feldman}
\end{center}

\vspace*{2ex}

    In this Appendix, we calculate the equivariant cohomology ring of the sphere $S^{2n}$ with respect to the action of the unitary group $U(2)$ induced by the Arnold diffeomorphism ${\CP}(1)^n\!/W \cong S^{2n}$, where $W$ is the Coxeter group of type $D_n$ of rank $n$. We verify that this cohomology ring is isomorphic to the ring of quasi-invariants in the case of two-element group.
    


\subsection{Setting}
Let $\mathbb{R}[x]= \mathbb{R}[x_1, \dots ,x_N]$, $N \in \mathbb N$. Recall that the ring of symmetric polynomials $\mathbb{R}[x]^{S_N}$ is isomorphic to the cohomology ring of the classifying space $BU(N)$ for the Lie group $U(N)$. Furthermore, the classifying space $BT^N$ of an $N$-dimensional torus $T^N$ is the Cartesian product $BT^N = {\CP}(\infty)^N$. Its cohomology ring is isomorphic to the ring ${\mathbb R}[x]$, and the projection in the bundle $p\colon BT^N \to BU(N)$ induces the canonical inclusion on cohomology $p^*\colon  {\mathbb R}[x]^{S_N} \subset \mathbb{R}[x]$.

We restrict ourselves to $N=2$. In this case we explicitly construct a collection of nice spaces $\mathcal P_n$, where $n \in \mathbb N$.
For $n=1$ the space we consider is $BT^2 = {\CP}(\infty)\times {\CP}(\infty),$ in the limit $n= \infty$ the corresponding space is homotopically equivalent to $BU(2)$, and for any $n\in \mathbb N$ our space is fibred over $BU(2)$ with the spherical fibre.

In our construction we rely on the Arnold--Maxwell Theorem \cite{Arnold1996}. Maxwell's theorem from the theory of spherical functions, also known as spherical harmonics, may be interpreted as a continuous one-to-one correspondence between the symmetric power of the projective plane $Sym^n ({\RP}^2)$ and the projective space ${\RP}^{2n}$ (see \cite{Arnold1996}). Consideration of a covering led Arnold to a theorem that a quotient of $\CP(1)^n$ 
by an action of the Coxeter group $W$  is diffeomorphic to the $2n$-dimensional sphere $S^{2n}$. 

We observe  that the sphere $S^{2n}$ inherits from this diffeomorphism a canonical $U(2)$ action that comes from the diagonal action of $U(2)$ on $\CP(1)^n$.
We calculate the equivariant cohomology ring of $S^{2n}$ with respect to this action. The equivariant cohomology are simply the ordinary cohomology of the space ${\mathcal P}_n$ given by the Borel construction:
\[
{\mathcal P}_n := EU(2) \times_{U(2)}S^{2n} \cong EU(2) \times S^{2n}/\!\sim,
\]
where $EU(2)$ is the total space of the universal $U(2)$-bundle $EU(2) \xrightarrow{U(2)} BU(2),$ and the equivalence relation $\sim$ is given by the $U(2)$-action: $(xg,y)\sim(x,gy)$, for $g \in U(2)$,
$x \in EU(2), y \in S^{2n}$. 
It turns out that the cohomology ring of this space is isomorphic to the ring of $m$-quasi-invariants $\mathcal{Q}_m$ for the two-element group when $n=2m+1$.

We hope that this connection between quasi-invariants and the Arnold--Maxwell topological theorem sheds light on a direction which one could take in order to find the corresponding topological spaces in a more general situation. An interesting class of spaces of possible relevance which may be worthy to investigate are the Bott--Samelson varieties (see e.g. \cite{AF23}). We also note that a general construction of spaces whose even cohomology are isomorphic to quasi-invariants was proposed recently in \cite{BLR} by making use of homotopy theory.

We compose material of this appendix in the following way. In Section \ref{AMT} we remind the construction of the Arnold diffeomorphism $\CP(1)^n/W \cong S^{2n}$. In Section \ref{Actionsec} we show that this diffeomorphism leads to an action of $U(2)$ on the sphere $S^{2n}$, and we calculate the induced representation of the torus $T^2 \subset U(2)$ in the tangents spaces of the fixed points on the sphere. Section \ref{ECsec} is devoted to the description of the rings of $T^2$ and $U(2)-$equivariant cohomologies of $S^{2n}$, which are obtained using the 
localisation property of the Gysin map. 
In Section \ref{rqisec} we discuss the structure of the rings of quasi-invariants for the two-particle case and verify that they are isomorphic to the $U(2)-$equivariant cohomology rings of $S^{2n}$ for odd~$n$. 

\subsection{Arnold--Maxwell Theorem}
\label{AMT}

Let $\alpha$ be a diffeomorphism of the complex projective line $\CP(1)$ given at every point $(z:w)\in \CP(1)$ by the formula 
\[
\alpha \big( (z : w) \big) = (\bar{w} : -\bar{z}).
\]
It is easily seen that $\alpha$ has order 2.

The Coxeter group $W$ of type $D_n$ of rank $n \in \N$ has a faithful representation in the group of diffeomorphisms of the Cartesian product $\CP(1)^n$. 
The image of this representation is generated by the symmetric group $S_n$, which permutes factors, and by the action of $\alpha$ on any two copies of $\CP(1)$ simultaneously. We assume that the group $W$ has one element for $n=1$, and that it is isomorphic to the product $S_2 \times S_2$ of the symmetric group $S_2$ when $n=2$. For $n=3$ the group is isomorphic to the symmetric group $S_4$. 

It is well known that the orbit space $\CP(1)^n/S_n \cong \CP(n)$ and the diffeomorphism associates to a collection of $n$ linear forms 
$$
z_j x + w_j y, \quad (z_j, w_j)\in \mathbb C^2\setminus (0,0),\quad  \, j = 1,\dots,n,
$$ 
which correspond to points $(z_j: w_j) \in \CP(1)$, a form of degree $n$: $\prod_{j=1}^n (z_j x + w_j y)$, defined up to a proportionality.  The next result was established by V. Arnold. 
\begin{theorem}[\cite{Arnold1996}]
\label{first}
The map
$
Ar\colon  
(\mathbb C^2)^n\to \mathbb{R}^{2n+1}
$
given by
\[
Ar\left((z_1,w_1),\dots,(z_n,w_n)\right) = (\mathrm{Re} f_{0}, \mathrm{Im} f_0,\dots, \mathrm{Re} f_{n-1}, \mathrm{Im} f_{n-1},f_n=\mathrm{Re}f_n),
\]
where
\begin{equation}
    \label{formexpanded}
\prod_{j=1}^n(z_j x +w_j y)(\bar{w}_j x - \bar{z}_jy) = f_0x^{2n} + \cdots + f_{2n}y^{2n},
\end{equation}
induces a diffeomorphism
\begin{equation}
\label{isomquotient}
\widehat{Ar}\colon \CP(1)^n/W \cong S^{2n}.
\end{equation}
\end{theorem}
The diffeomorphism \eqref{isomquotient} associates to a representative $\{(z_1,w_1),\dots, (z_n,w_n)\}$ of a point from $\CP(1)^n/W$ the $2n$-form  
\begin{equation}\label{f}
    f(x,y) = \prod_{j=1}^n (z_jx+w_jy)(\bar{w}_j x-\bar{z}_jy),
\end{equation}
which is described by $2n+1$ real parameters and is unique up to proportionality. The change of a representative $(z_j,w_j) \to (\lambda z_j, \lambda w_j)$, $\lambda \in \mathbb C^\times$ of a point $(z_j: w_j)$ in the $j$-th copy of $\CP(1)$ in $\CP(1)^n$ leads to the multiplication of the form \eqref{f} by a positive real number $|\lambda|^2,$ while the action of $\alpha$ on a single copy of the projective line inside the Cartesian product changes the sign of the form. Hence the map $\widehat{Ar}$ is well-defined on the quotient. 

\begin{remark}
    As discussed in \cite{Arnold1996}, the diffeomorphism \eqref{isomquotient} is a higher-dimensional generalisation of the classical diffeomorphism
    \[
    \CP(2)/\mathrm{conj} \cong S^4, \quad \mathrm{conj}(z_1:z_2:z_3)=(\bar{z}_1: \bar{z}_2: \bar{z}_3),
    \]
    written in different linear coordinates.
\end{remark}

Further to that, Arnold explains in \cite{Arnold1996} a relation of Theorem \ref{first} with a theorem of Maxwell. 
    Maxwell's theorem states that every spherical function in $\mathbb{R}^3$ of degree $n$, that is the restriction to the sphere $S^2$ of a homogeneous harmonic polynomial of degree $n$, is proportional to the derivative of the function $1/r$ along $n$ constant directions in $\mathbb{R}^3$, where $r$ 
    is the distance to the origin. The dimension of the space of spherical functions of degree $n$ in $\mathbb{R}^3$ is $2n+1$ and, in particular, Maxwell's theorem establishes a homeomorphism
    \[
    Sym^n(\RP(2)) = \RP(2)^n/S_n \cong \RP(2n).
    \]
    Arnold gives a topological proof of this statememt in \cite{Arnold1996}. Then by considering a two-fold cover $\CP(1) \to \RP(2)$ he concludes that $\CP(1)^n/W$ is a two-fold cover of
    \[
    \CP(1)^n/(S_n \ltimes (\mathbb{Z}_2)^n) \cong \RP(2n),
    \]
    and $\CP(1)^n/W$ is homeomorphic (in fact, diffeomorphic) to the sphere $S^{2n}$.

\subsection{Action of $U(2)$ on $S^{2n}$}
\label{Actionsec}

The standard linear action $\rho$ of the group $U(2)$ on $\CP(1)$ extends to the diagonal action $\Delta$ of $U(2)$ on $\CP(1)^n$. We show now that $\Delta$ commutes with the action of $W$ and, therefore, descends to the action on the quotient $\CP(1)^n/W \cong S^{2n}$.

It is clear that $\Delta$ commutes with the action of the subgroup of permutations $S_n \subset W$. Let us check that the action $\rho$ commutes with the action of $\alpha$ on a single copy of $\CP(1)$. Let $A= (a_{ij}) \in U(2), p= (z:w) \in \CP(1)$. Then 
\begin{equation}
\label{3}
\begin{aligned} 
    &\alpha\circ A(p) = \alpha\big((a_{11}z+a_{12}w:a_{21}z+a_{22}w)\big)   \\
    &=(\bar{a}_{21}\bar{z}+\bar{a}_{22}\bar{w}:-\bar{a}_{11}\bar{z}-\bar{a}_{12}\bar{w}) = \widetilde{A}\big( (\bar{w}:-\bar{z})\big), 
\end{aligned}
\end{equation}
where
\[
\widetilde{A}=
\begin{pmatrix}
    \bar{a}_{22} &-\bar{a}_{21} \\
    -\bar{a}_{12} & \bar{a}_{11}
\end{pmatrix}.
\]
We note that 
\begin{equation}
\label{4}
\widetilde{A}=(\det\bar{A})(\bar{A}^t)^{-1}= (\det\bar{A})A.
\end{equation}
Combining \eqref{3} and \eqref{4} we obtain
\[
\alpha \circ A\big(p\big) = \widetilde{A}\big((\bar{w}:-\bar{z})\big) = A \circ \alpha \big((\det(A)z:\det(A)w)\big) = A \circ \alpha\big(p\big).
\]
Since the action of the group $W$ is generated by $S_n$ and by the action of $\alpha$ simulatenously on (an even number of) copies of $\CP(1)$ in $\CP(1)^n$ we get the following statement.

\begin{theorem}\label{second}
    The diagonal action $\Delta$ of $U(2)$ on $\CP(1)^n$ commutes with the action of $W$ and can be extended to the quotient action of $U(2)$ on $\CP(1)^n/W \cong S^{2n}.$
\end{theorem}

Let $T=T^2 \subset U(2)$ be the real two-dimensional torus given by the diagonal matrices. The action $\Delta$ of $U(2)$ on $S^{2n}$ induces the action of the torus on $S^{2n}$ with two fixed points. The first fixed point $p_1$
is the image under the map $\widehat{Ar}$  of the point 
\[
q_1=\big((0:1), (0:1), \ldots, (0:1)\big)\in \mathbb CP(1)^n/W.
\]
The second fixed point $p_2 = \widehat{Ar}\big((1:0), (0:1), \ldots, (0:1) \big).$


\begin{prop}\label{third}
    The functions $\frac{f_0}{f_n},\frac{f_1}{f_n},\dots,\frac{f_{n-1}}{f_n}$ of the coefficients of the form \eqref{formexpanded} induce the structure of a complex vector space in the tangent spaces to $S^{2n}$ at the fixed points  $p_1$ and $p_2$ of the $T$-action. The representation $\chi\colon T \to {\rm End}(\mathbb C^n)$ of the torus $T$ in the tangent spaces of $S^{2n}$ at the points $p_1$ and $p_2$ with respect to this complex structure has weights 
    \[
    (-n,n), (-n+1, n-1), \ldots, (-1, 1). 
    \]
\end{prop}
    \begin{proof}
    
        Near the point $p_1\in S^{2n}$ we can assume that its preimage under the map $\widehat{Ar}$ 
        has a representative $\left((z_1,1), \ldots, (z_n, 1)\right)$, where 
$|z_i|$ is small. Let $\sigma_i$ be the $i$th elementary symmetric polynomial in $z_1, \ldots, z_n$, and let $\bar \sigma_i$ be its complex conjugate, which is the elementary symmetric polynomial in $\bar z_1, \ldots, \bar z_n$. We can take $\sigma_i, \bar\sigma_i$ as local coordinates near the preimage $\tilde q_1$ of $p_1$  under the map 
    $$
    G\colon \mathbb{C}P(1)^n/S_n\cong \mathbb{C}P(n) \to \mathbb{C}P(1)^n/W\cong S^{2n}.
    $$ Note that (cf. \cite{Arnold1996})
\begin{equation}
    \label{fsigma}
f_k = \sum_{j=0}^k (-1)^j \bar \sigma_j \sigma_{n-k+j}, \quad 0\le k \le n,
\end{equation}
and $f_{2n-k} = (-1)^{n-k} \bar f_k$. In particular, $f_n$ is real and $f_n(0)=1$. The latter implies that $f_i/f_n$ for  $0\le i \le n-1$ can be taken as complex local coordinates on $S^{2n}$ near the point $p_1$. 
 
The action of the torus on $S^{2n}$ is induced by the map $G$. Note that for $t=(t_1, t_2) \in T$ we have 
\begin{equation*}
\label{tz}   
t (z_i:1) = ( t_1 t_2^{-1} z_i:1).
\end{equation*}
  In the local coordinates $\sigma_i,\bar \sigma_i$ we have 
$t(\sigma _i) = (t_1 t_2^{-1})^i \sigma_i$ and $t(\bar \sigma_i) = (t_1 t_2^{-1})^{-i} \bar\sigma_i$. Therefore $t(f_k/f_n) = (t_1 t_2^{-1})^{n-k} f_k/f_n$ by formula \eqref{fsigma}, which implies the statement. 

The claim for $p_2$ follows similarly. In this case points near a point $\tilde q_2\in \CP(1)^n/S_n$ such that $G(\tilde q_2)=p_2$ can be represented as $\left((z_1, 1), \ldots, (z_{n-1}, 1), (1, z_n)\right)$, where $|z_i|$ is small for all $i$. 
Let $\sigma_i$ now be the $i$th elementary symmetric polynomial in $z_1, \ldots, z_{n-1}$ and $-\bar z_n$, and let the coordinate $\bar \sigma_i$ be its complex conjugate. Then  formula \eqref{fsigma} holds after multiplying the right-hand side by -1. The action of an element $t=(t_1, t_2)\in T$  on the coordinate $z_i$ is given by $t(z_i) = t_1 t_2^{-1} z_i$ for $1\le i \le n-1$, and $t(z_n)=t_1^{-1} t_2 z_n$, $t(\bar z_n)=t_1 t_2^{-1} \bar z_n$, which implies the statement. 
    \end{proof}

One can use any of the two complex structures of Proposition \ref{third} to define an orientation of the sphere $S^{2n}$. We will need the following statement later.

\begin{lemma}
    \label{firstrem}
The two orientations of $S^{2n}$ induced by the complex structure of Proposition~\ref{third} at $p_1$ or $p_2$ are opposite one to another.
\end{lemma}
\begin{proof}
Our sphere $S^{2n}$ is realized as the set of rays in the space $\mathbb R^{2n+1}$ with coordinates $x_1=\Re f_0, x_2=\Im f_0, \ldots, x_{2n-1}=\Re f_{n-1}, x_{2n}=\Im f_{n-1}, x_{2n+1}=f_n$. Let us identify this topological sphere with the standard sphere given as $\sum_{i=1}^{2n+1} x_i^2=1$. Then $p_1$ corresponds to the northern pole, and $p_2$ corresponds to the southern pole, since $f_n=1$ or $f_n=-1$, respectively, and $f_i=0$ for $0\le i \le n-1$ at both $p_1$ and $p_2$. The two bases at the tangent spaces at the poles are $\partial_{x_1}, \ldots, \partial_{x_{2n}}$.  
Consider a path on the sphere along the big semicircle given by $x_i=0$ for $2\le i \le 2n$. Note that a continuation of the basis  $\partial_{x_1}, \partial_{x_2}, \ldots, \partial_{x_{2n}}$ at one of the poles leads to the basis  $-\partial_{x_1}, \partial_{x_2},\ldots, \partial_{x_{2n}}$ at the other pole. The statement follows.
\end{proof}


\subsection{Equivariant cohomology of the sphere}
\label{ECsec}

We calculate now the equivariant cohomology of the sphere $S^{2n}$ with respect to the $U(2)-$action $\Delta$. Let $\mathcal{P} = EU(2) \times_{\Delta}S^{2n}$ be the total space of the bundle 
$$
p\colon {\mathcal P} \xrightarrow{S^{2n}} BU(2)
$$ 
given by the Borel construction. The equivariant cohomology of $S^{2n}$ with respect to the action $\Delta$ is the ordinary cohomology of the space $\mathcal{P}$. We denote by $\tau$ the vector bundle with the base $\mathcal{P}$ and the fibre being the tangent space along the fibre of $\mathcal{P}$.
We fix an orientation of $S^{2n}$ and consider induced orientation on $\tau$. 

We consider all the cohomology over the field $\mathbb R$ unless other coefficients are specified.  The next lemma is a standard corollary from the structure of bundles with spherical fibres. Note that all the differentials on the  second and subsequent pages of the spectral sequence are trivial 
when the fibre is an even-dimensional sphere and all the cohomology of a base are in even degrees (cf. e.g. \cite{Tu_EqCoh}).

\begin{lemma}\label{firstlem}
    The cohomology ring 
    \[
 H^*(\mathcal{P}) \cong  H^*(BU(2))[s]/\left(s^2-p^*(\beta_1)s - p^*(\beta_2)\right)
    \]
    for some $\beta_1, \beta_2 \in H^*(BU(2)).$ Any element $s \in H^{2n}(\mathcal{P})$ whose restriction to every fibre generates $H^{2n}(S^{2n})$ can be chosen as the  generator of $H^*(\mathcal{P})$ as a module over $H^*(BU(2))$. In particular, the Euler class $e(\tau)$ of the vector bundle $\tau$ is a generator for $H^*(\mathcal{P}).$
\end{lemma}

Our aim now is to find $\beta_1$ and $\beta_2$ in the relation 
\[
e(\tau)^2= p^*(\beta_1) e(\tau)+p^*(\beta_2).
\]
Let $j\colon BT \to BU(2)$ be the canonical map induced by the diagonal embedding $T \subset U(2)$. The induced map on cohomology
\[
j^*\colon H^*(BU(2)) \to H^*(BT)
\]
is monomorphic, and  the images of the first and the second Chern classes $c_1(\eta)$ and $c_2(\eta)$ of the tautological complex vector bundle $\eta$ over $BU(2)$ satisfy
$$
j^*(c_1(\eta))= x_1+x_2, \quad 
j^*(c_2(\eta))= x_1 x_2, 
$$
where $x_1=c_1(\zeta_1),\, x_2=c_1(\zeta_2)$ are the first Chern classes of the tautological complex line bundles over the respective factors of $BT=\mathbb CP(\infty)\times \mathbb CP(\infty)$.

Let us consider the pullback bundle $\mathcal{P}' = j^*\mathcal{P}$,  
\[
p'\colon \mathcal{P}'
\xrightarrow{S^{2n}} BT.
\]
Since the map $j^*$ is monomorphic, in order to find  $\beta_1$ and $\beta_2$ it is sufficient  to calculate their images under $j^*$ in the cohomology ring $H^*(BT)$. Under  $j^*$ the Euler class $e(\tau)\in H^*(\mathcal P)$ is mapped to the Euler class $e(\tau')\in H^*(\mathcal P')$ of the bundle $\tau'$ with the base $\mathcal{P}'$ and fibres being the tangent spaces along the fibres of the bundle $\mathcal{P}'$. Here we assume that the vector bundle $\tau'$ gets orientation from the same orientation of $S^{2n}$ as the one used to orient the vector bundle $\tau$. Equivalently, the map $j$ extends to an orientation preserving bundle map between $\tau$ and $\tau'$. 

We have
\[
e(\tau')^2 = ({p'}^* \circ j^*)(\beta_1)e(\tau')+({p'}^*\circ j^*)(\beta_2).
\]

In order to calculate $j^*(\beta_k), k=1,2,$ we  apply the localisation property of the Gysin map which allows one to compute the Gysin map in terms of the fixed point data of the corresponding structure group action (see e.g. \cite{Audin1991book} or \cite{Tu} and references therein). 

Consider the normal bundle 
${\mathcal N}_k \to i_k(BT)$ of the section $i_k\colon BT \to \mathcal{P}'$ corresponding to the fixed point $p_k, k=1,2$. Let ${\xi}_k$ be the pullback of this bundle under the map $i_k$, which is a vector bundle over $BT$, and let $e({\xi}_k)\in H^*(BT)$ be the Euler class of the bundle ${\xi}_k$. Equivalently, the bundle $\xi_k$ can be given by the Borel construction for the action of the torus $T$ in the tangent space $T_{p_k} S^{2n}$. 
The localisation property takes the form 
\[
p'_!(\alpha) = \frac{i^*_1 \alpha}{e({\xi}_1)} + \frac{i^*_2 \alpha}{e({\xi}_2)},
\]
where $\alpha\in H^*(\mathcal P')$.

The Euler classes  can be computed by considering  induced representations of the torus in the tangent spaces to the fixed points. Indeed, it follows from Proposition \ref{third} that the bundle ${\xi}_k$ splits into a direct sum of $n$ one-dimensional complex bundles. 
The Euler class is equal to the product of the first Chern classes of these bundles. For a $k$th tensor power $L^{\otimes k}$ of a line bundle $L$ we have its first Chern class $c_1(L^{\otimes k})=k c_1(L)$. It follows from Proposition \ref{third} and Lemma \ref{firstrem} that, for a suitable choice of orientation of $S^{2n}$,  
\[
e({\xi}_1) = n!(x_1-x_2)^n, \quad e({\xi}_2) = -n! (x_1-x_2)^n.
\]
Therefore,
\begin{align*}
p'_!(1)=\frac{1}{e(i_1)}+\frac{1}{e(i_2)}=0, \qquad &
p'_!e(\tau')=\frac{i^*_1e(\tau')}{e(\xi_1)}+\frac{i^*_2e(\tau')}{e(\xi_2)}=2,
\\
p'_!e(\tau')^2=\frac{e(\xi_1)^2}{e(\xi_1)}+\frac{e(\xi_2)^2}{e(\xi_2)}=0, \qquad &
p'_!e(\tau')^3 = \frac{e(\xi_1)^3}{e(\xi_1)}+\frac{e(\xi_2)^3}{e(\xi_2)}=2(n!)^2(x_1-x_2)^{2n}.
\end{align*}
By applying the Gysin map $p'_!$ to the relation
\[
e(\tau')^2=(p'^*\circ j^*)(\beta_1)e(\tau') + (p'^* \circ j^*)(\beta_2)
\]
we deduce that $\beta_1=0$. By applying $p'_!$ to the relation
\[
e(\tau')^3=(p'^*\circ j^*)(\beta_2)e(\tau'),
\]
and using that the Gysin map is linear over the cohomology of the base we get
$$
(n!)^2 (x_1-x_2)^{2n} = j^*(\beta_2).
$$
Thus, we have established the following.

\begin{theorem}\label{fourth}
    The Euler class $e(\tau')$ satisfies the relation 
    \[
    e(\tau')^2 = (n!)^2p'^*(x_1-x_2)^{2n}.
    \]
\end{theorem}
\begin{cor}\label{firstcor}
    The equivariant cohomology ring of the sphere $S^{2n}$ with respect to the $U(2)-$action $\Delta$ as a module over $H^*(BU(2))$ is generated by an element $\gamma \in H^{2n}(\mathcal P)$ subject to the relation
    \[
    \gamma^2 = p^*(c_1(\eta)^2-4c_2(\eta))^n.
    \]
    \begin{proof}
     By Lemma \ref{firstlem} and Theorem \ref{fourth} one can take $\gamma=e(\tau)/n!$ as the required generator.
    \end{proof}
\end{cor}
Notice that it also follows from Theorem \ref{fourth} that the ring of equivariant cohomology of $S^{2n}$ with respect to the $T$-action is isomorphic to the quotient of the polynomial ring
$ 
\mathbb R[x_1, x_2, \gamma']
$
by the ideal generated by the relation $
(\gamma')^2=(x_1-x_2)^{2n}$.

\subsection{Relation to quasi-invariants}
\label{rqisec}

Recall that in dimension two the ring $\mathcal{Q}_m$ of $m$-quasi-invariants for the symmetric group $S_2$ consists of real polynomials $p(x_1,x_2)$ satisfying 
\[
\left(\frac{\partial}{\partial x_1}-\frac{\partial}{\partial x_2} \right)^{2i-1}p(x_1,x_2)|_{x_1=x_2}=0, \quad i=1,\dots,m, 
\]
where $m \in \mathbb N$ (see \cite{CV90}, \cite{FV02}). Since any symmetric in $x_1,x_2$ polynomial belongs to $\mathcal{Q}_m$, as well as any polynomial divisible by $(x_1-x_2)^{2m+1}$, we have the following direct sum decomposition:
\begin{equation}
\label{5}
    \mathcal{Q}_m={\mathbb R}[x_1,x_2]^{S_2}\oplus (x_1-x_2)^{2m+1}{\mathbb R}[x_1,x_2]^{S_2}.
\end{equation}
Quasi-invariants interpolate between symmetric and all the polynomials, and the following inclusions take place:
\[
\mathbb{R}[x_1,x_2] = \mathcal{Q}_0 \supset \mathcal{Q}_1 \supset\cdots \supset \mathcal{Q}_m \supset \cdots \supset \mathcal{Q}_{\infty} = \mathbb{R}[x_1, x_2]^{S_2}.
\]
As a ring $\mathcal{Q}_m$ is generated by $y_1=x_1 + x_2,\, y_2=x_1x_2$ and $u=(x_1-x_2)^{2m+1}$ with the only relation $u^2=(y_1^2-4y_2)^{2m+1}$.

From Corollary \ref{firstcor} we get
\begin{theorem} \la{heqar}
    The ring $\mathcal{Q}_m$ of $m$-quasi-invariants is isomorphic to the $U(2)-$equivariant cohomology ring of $S^{2(2m+1)}$ with the $U(2)-$action induced from the Arnold diffeomorphism 
    $$
    \mathbb CP(1)^{2m+1}/D_{2m+1} \cong S^{2(2m+1)}\,.
    $$
\end{theorem}
Let us consider relations of equivariant cohomology rings for different values of the parameter $m$. More generally, let ${\mathcal P} ={\mathcal P}_n$ and ${\mathcal P'}={\mathcal P}_n'$ be the bundles with the fibre $S^{2n}$ from Section \ref{ECsec}, and let $\tau=\tau_n$, $\tau'=\tau_n'$ be the vector bundles of tangent spaces along the fibres for ${\mathcal P}_n$ and ${\mathcal P}'_n$, respectively. Let $p_n'=p'\colon {\mathcal P_n'} \to BT$. Denote the  Euler classes $e_n=e(\tau_n)$ and $e_n'=e(\tau_n')$. 

Consider two maps $g_n^{\pm}\colon {\mathcal P_n'} \to {\mathcal P_{n+1}'}$ which are induced by the respective maps $h_n^{\pm}\colon \mathbb C P(1)^n \to \mathbb C P(1)^{n+1}$ sending a point $(a_1, \ldots, a_n)$ to $(a_1, \ldots, a_n, z)$, where $z=(1:0)$ for $g_n^+$ and $z=(0:1)$ for $g_n^-$. The map $h_n^{\pm}$ descends to the embedding $S^{2n} \subset S^{2n+2}$ through Arnold diffeomorphism. It leads to the map $h_n^{\pm}$ between Borel constructions since $z$ is a fixed point for the torus $T$. These maps induce the following maps on the cohomology.

\begin{prop}
The induced maps ${g_n^{\pm}}^*\colon H^*({\mathcal P}_{n+1}') \to H^*({\mathcal P}_{n}')$ are injections which  map the generator $e'_{n+1}$ of $H^*({\mathcal P}_{n+1}')$ as $H^*(BT) \cong \mathbb R[x_1, x_2]-$module to   
$$
{g_n^\pm}^* (e'_{n+1}) = \delta (n+1) (x_1-x_2) e'_n, 
$$
where $\delta \in \{1, -1\}$, and the maps are linear over $H^*(BT)$.
\end{prop}
\begin{proof}
The pullback ${g_n^{\pm}}^*(\tau_{n+1})$ of the bundle $\tau_{n+1}$ is the direct sum of the bundle $\tau_n$ and the rank 2 vector bundle $\sigma_n$ over the base $\mathcal P_n'$. The fibre of the bundle $\sigma_n$ over a point $w\in \mathcal P_n'$ consists of normals to $S^{2n}$ inside the sphere $S^{2n+2}$ containing $g_n^{\pm}(w) \in \mathcal P_{n+1}'$. By the functoriality of the Euler class and applying the Whitney sum formula we have
\begin{equation}
\label{homf}
{g_n^{\pm}}^*(e_{n+1}')=e_n' e(\sigma_n),   
\end{equation}
where $e(\sigma_n)$ is the Euler class for the bundle $\sigma_n$. 

Since the normal bundle to $S^{2n}$ in $S^{2n+2}$ is trivial, the restricton of the class $e(\sigma_n)$ to any fibre of the bundle $\mathcal P_n'$ is 0. Similarly to Lemma \ref{firstlem}, $H^*(\mathcal P_n')$ is generated by a single generator over $H^*(BT)$ whose restriction to any fibre is nonzero. It follows that $e(\sigma_n) \in p_n'^*(H^*(BT))$. 

Let us identify $p_n'^*(H^*(BT))$ and $p_{n+1}'^*(H^*(BT))$ with $\mathbb R[x_1, x_2]$. By Theorem \ref{fourth} we have relations 
\begin{equation}
\label{rels}
e_{n+1}'^2=((n+1)!)^2(x_1-x_2)^2, \quad e_{n}'^2=(n!)^2(x_1-x_2)^2.   
\end{equation}
Since ${g_n^{\pm}}^*$ is a homomorphism and $e(\sigma_n) \in \mathbb R[x_1, x_2]$ it follows from \eqref{homf} and \eqref{rels} that
$e(\sigma_n)^2 = (n+1)^2(x_1-x_2)^2$, which implies the statement.
\end{proof}
By composing maps $g_n^\pm$ for different $n$ we get the maps
$$
g_{n_1, n_2}^\varepsilon\colon \mathcal P_{n_1}' \to \mathcal P_{n_1+n_2}',
$$
where $g_{n_1,n_2}^\varepsilon = g_{n_1}^{\varepsilon_1} g_{n_1+1}^{\varepsilon_2} \ldots g_{n_1+n_2-1}^{\varepsilon_{n_2}}$, $\varepsilon = (\varepsilon_1, \ldots, \varepsilon_{n_2})$, $\varepsilon_i \in \{1,-1\}$ and $n_1, n_2 \in \mathbb N$. The corresponding inclusion on cohomology
$$
(g_{n_1,n_2}^{\varepsilon})^*\colon H^*(\mathcal P_{n_1+n_2}') \to H^*(\mathcal P_{n_1}')
$$
is determined by $(g_{n_1, n_2}^\varepsilon)^*(e_{n_1+n_2}')=\delta \prod_{k=1}^{n_2} (n_1+k) (x_1-x_2)^{n_2} e'_{n_1}$, where $\delta \in \{1, -1\}$. 

Note that we do no not seem to have natural maps between spaces $\mathcal P_n$ for different $n$ which for odd $n$ would induce inclusion of cohomology $\mathcal Q_{m_1} \subset \mathcal Q_{m_2}$, $m_1>m_2$.




\begin{remark}
The cohomology class $s_n=\frac12 e_n$ generates $H^*_{U(2)}(S^{2n},\mathbb{Z})$ since the value of the Euler class $e_n$ on a fibre of the bundle $\mathcal P_n$ is equal to the Euler characteristic $\chi(S^{2n})=2$.  This generator satisfies the  relation 
\[
4 s_n^2={(n!)^2}(c_1(\eta)^2-4c_2(\eta))^n
\]
in $H^*_{U(2)}(S^{2n},\mathbb{Z})$.
It may be interesting to clarify the structure inside the ring of quasi-invariants corresponding to the sub-lattice of integral cohomology in $H^*_{U(2)}(S^{2n},\mathbb{R})$ for odd~$n$.
\end{remark}

\subsection*{Acknowledgements} M.F. is grateful to A.P. Veselov for posing a question of topological realisation of quasi-invariants, to  H. Khudaverdian for interest in our work and to Yu.~Berest for useful discussions. 
The work of M.F. was partially supported by the Engineering and Physical Sciences Research Council [grant number  EP/W013053/1].



\bibliography{secondbibtexfile}{}
\bibliographystyle{amsalpha}
\end{document}